\documentclass[pdflatex,sn-mathphys-num]{sn-jnl}

\usepackage{graphicx}
\usepackage{multirow}
\usepackage{amsmath,amssymb,amsfonts}
\usepackage{amsthm}
\usepackage{mathrsfs}
\usepackage[title]{appendix}
\usepackage{xcolor}
\usepackage{textcomp}
\usepackage{booktabs}
\usepackage{tabularx}
\usepackage{url}
\usepackage{tikz}
\usepackage{subcaption}
\usetikzlibrary{arrows.meta,angles,quotes,positioning,calc,trees}

\theoremstyle{thmstyleone}
\newtheorem{theorem}{Theorem}
\newtheorem{lemma}[theorem]{Lemma}
\newtheorem{proposition}[theorem]{Proposition}
\newtheorem{corollary}[theorem]{Corollary}

\theoremstyle{thmstyletwo}

\newtheorem{remark}{Remark}

\theoremstyle{thmstylethree}
\newtheorem{definition}{Definition}

\newcommand{\conv}{\mathrm{conv}}
\newcommand{\R}{\mathbb{R}}
\newcommand{\predlines}[1]{\begin{minipage}[t]{\linewidth}\raggedright #1\end{minipage}}

\begin{document}

\title[Wetzel's Triangle Covers Unit Arcs]{Wetzel's $30^{\circ}\!-\!60^{\circ}\!-\!90^{\circ}$ Triangle Covers Unit Arcs}

\author[1]{\fnm{Wacharin} \sur{Wichiramala}}\email{wacharin.w@chula.ac.th}
\affil[1]{\orgdiv{Department of Mathematics and Computer Science}, Faculty of Science, \orgname{Chulalongkorn University}, \orgaddress{\street{Phayathai Road}, \city{Bangkok}, \postcode{10330}, \country{Thailand}}}

\author*[2]{\fnm{Chatchawan} \sur{Panraksa}}\email{chatchawan.pan@mahidol.ac.th}
\affil*[2]{\orgdiv{Applied Mathematics Program}, \orgname{Mahidol University International College}, \orgaddress{\street{999 Phutthamonthon 4 Road Salaya}, \city{Nakhonpathom}, \postcode{73170}, \country{Thailand}}}

\abstract{John E. Wetzel conjectured that the $30^{\circ}\!-\!60^{\circ}\!-\!90^{\circ}$ triangle $T$ obtained by placing a square of side $1/3$ on the hypotenuse covers every unit arc in the plane. We give a computer-assisted proof of this conjecture with independently checkable interval certificates. The proof reduces a hypothetical noncovered arc to a finite family of $599$ closed second-order cone models, covering all representative and raw tail-order branches, and certifies a polygonal-chain lower bound greater than one in every model by interval validation of stored dual certificates. Since every certified lower endpoint exceeds $1.0048$, the homothetic copy $T/1.0048$ still covers every unit arc. Its area is $0.260956\ldots$, below the area $\pi/12\approx0.261799$ of the $30^{\circ}$ unit sector, a certified area improvement over the sector cover within this convex Wetzel-cover setting.}

\keywords{Unit arc covers, Wetzel's conjecture, Convex geometry, Support lines, Computational geometry}
\pacs[MSC Classification]{52C15, 52A10, 52A38, 68U05}

\maketitle

\section{Introduction}\label{sec:intro}

Covering problems for plane curves of fixed length ask for small convex sets that contain a congruent copy of every unit-length arc; they are instances of Moser's worm problem~\cite{moser1966,moser1991} and of Wetzel's \emph{fits and covers} program.

Two conjectures of John E.~Wetzel from the early 1970s assert that a specific $30^{\circ}\!-\!60^{\circ}\!-\!90^{\circ}$ triangle $T$ and the $30^{\circ}$ unit sector each cover every unit arc. Wetzel's study of triangular covers goes back to the closed-curve case~\cite{wetzel1970}. The triangle conjecture is recorded as \cite[Conjecture~2]{wetzel2003} (with the square-on-hypotenuse description used below) and as \cite[Conjecture~4.2]{norwood1992}; the sectorial program goes back to \cite{wetzel1973}, with the sector conjecture first reported in \cite{norwood1992} and \cite[p.~358]{wetzel2003}. The triangle $T$ studied here is not the $30^{\circ}\!-\!60^{\circ}\!-\!90^{\circ}$ triangle of \cite[Conjecture~4.3]{norwood1992}, whose covering claim was refuted (reported in \cite[p.~358]{wetzel2003}). The sectorial conjecture is resolved: the $30^{\circ}$ unit sector covers all unit arcs \cite{panraksa2020}, giving the previous Wetzel-program convex-cover benchmark of area $\pi/12 \approx 0.261799$. The sector theorem of \cite{panraksa2020} proves the sector cover; the present paper proves the distinct triangular cover conjecture for the square-on-hypotenuse triangle $T$, and the certified margin then permits a homothetic shrink of $T$ whose area is below that of the unit sector.

\begin{figure}[htbp]
    \centering
    \includegraphics[width=0.8\textwidth]{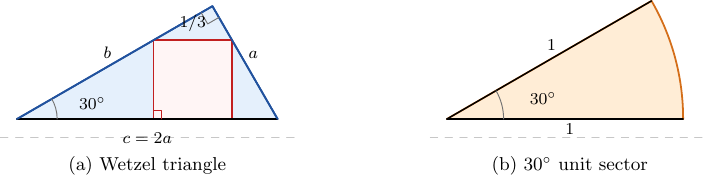}
    \caption{The two covering candidates conjectured by Wetzel: (a) The $30^{\circ}\!-\!60^{\circ}\!-\!90^{\circ}$ triangle $T$ with inscribed square of side $1/3$, and (b) the $30^{\circ}$ unit sector.}
    \label{fig:contenders}
\end{figure}
\subsection{Main results}

Our main theorem is as follows.

\begin{theorem}\label{thm:main}
Let $T$ be the $30^{\circ}\!-\!60^{\circ}\!-\!90^{\circ}$ triangle whose hypotenuse supports a square of side $1/3$, with side lengths
\[
a=\frac{3+4\sqrt{3}}{18},\qquad b=\frac{4+\sqrt{3}}{6},\qquad c=2a.
\]
Every unit arc in $\mathbb{R}^2$ is contained in a congruent copy of $T$.
\end{theorem}

The numerical Mathematica/SOCP lower-bound computation has its smallest representative value in Subcase~1.1a,
\[
L_{1.1a}^{\rm nb}=1.0048313701863976\ldots.
\]
The independent interval certificate in Appendix~\ref{app:opt} proves a more conservative but rigorous endpoint for the same closed model,
\[
L_{1.1a}^{\rm cert}=1.0048290490\ldots,
\]
using the main refined representative certificate radius $q=0.999999$. Across all $599$ representative and raw-order branch models, this is the weakest certified endpoint; the refined raw-order minimum is $1.0057165531\ldots$. Thus the supplementary computation also certifies the stronger margin $L_i^{\rm cert}>1.0048$ for every closed branch model. The following remark records the resulting computational scaling consequence separately; the proof of Theorem~\ref{thm:main} itself uses only the inequalities $L_i^{\rm cert}>1$.

\begin{remark}[Certified computational scaling margin]\label{rem:conditional-scale}
The triangle $T$ can be homothetically shrunk by the factor $s=1/1.0048$ to area $0.260956$, smaller than the $30^{\circ}$ unit sector area $\pi/12 \approx 0.261799$, and still cover every unit arc. This is stated and proved as Proposition~\ref{prop:scaled-cover} in Section~\ref{sec:quant}.
\end{remark}

Appendix~\ref{app:opt} records the notebook data for the representative models, the independent interval certificates, and the expanded raw-order certificates.

\subsection{Proof architecture}

Our approach rests on three pillars. \emph{(i)} We first use the Wetzel--Wichiramala covering theorem~\cite[Cor.~5]{wetzel2010} to reduce the covering problem to simple polygonal unit arcs, then work in the oriented support-line framework and use the $\Lambda$-property to control contact configurations: parallel- and angle-support versions interlace touch triples at prescribed directions, constraining escape patterns. \emph{(ii)} Via a compactness reduction, with closed limiting models allowed, we pass to weak support-reduced obstructions and a finite family of normalized configurations of $T$. \emph{(iii)} For each certified branch model $i$, we encode the geometry as a constrained optimization problem on a feasible set determined by contact, ordering, and escape constraints, and we minimize a polygonal lower bound for arc length over that set. Appendix~\ref{app:opt} records independent interval certificates for the sufficient inequalities $L_i^{\rm cert}>1$ on all corresponding closed feasible models, and Lemma~\ref{lem:row-audit} certifies, as part of the machine-checked verification, that every model row coincides with the affine form of its forcing placement.

The case tree is organized by which points are touched by support lines at angles $\pm 150^\circ$ and $\pm 120^\circ$. Three main branches yield twelve representative terminal subcases, including a direct both-late Subcase~1.1 model. To avoid any unproved Case~3 tail-order exclusion, the supplementary computation certifies every Case~3 order pairing of the four right-tail contacts $P_4,P_5,P_7,P_8$ with the four reflected left-tail contacts, beyond the three pairings among the representative models; in total the proof uses $599$ certified closed models, $12$ representative and $587$ auxiliary raw-order. In each certified branch, $T$ is placed in canonical orientations, the arc is forced to escape across the short side, and the polygonal chain in the recorded contact order is minimized. Figure~\ref{fig:proof-map} summarizes the pipeline.

\begin{figure}[t]
\centering
\begin{tikzpicture}[
  >=Latex,
  node distance=7mm and 8mm,
  box/.style={draw, rounded corners, align=center, font=\small, fill=blue!6, minimum width=4.9cm, minimum height=8mm},
  arrow/.style={->, thick}
]
\node[box] (u) {All unit arcs of length $1$};
\node[box, below=of u] (p) {Wetzel--Wichiramala reduction\\to simple polygonal arcs};
\node[box, below=of p] (h) {Weak support-reduced\\obstructions};
\node[box, below=of h] (c) {Canonical placements of $T$\\and forced escape inequalities};
\node[box, below=of c] (k) {Finite case tree\\($599$ certified closed models)};
\node[box, below=of k] (o) {Polygonal-chain lower bounds\\and interval certificates};
\node[box, below=of o, fill=green!10] (m) {Certified $L_i^{\rm cert}>1$ in every subcase\\$\Rightarrow$ contradiction with $\ell(\gamma)\le 1$};
\draw[arrow] (u) -- (p);
\draw[arrow] (p) -- (h);
\draw[arrow] (h) -- (c);
\draw[arrow] (c) -- (k);
\draw[arrow] (k) -- (o);
\draw[arrow] (o) -- (m);
\end{tikzpicture}
\caption{Proof pipeline from structural reduction to the numerical lower-bound stage.}
\label{fig:proof-map}
\end{figure}
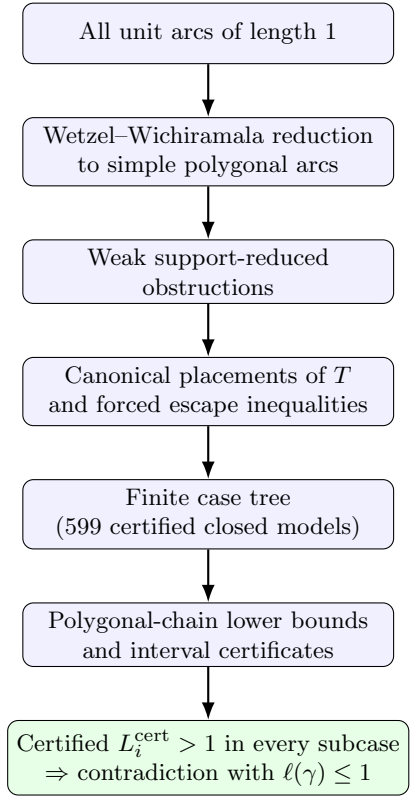

\subsection{Related work and context}

Wetzel initiated the fits-and-covers program in the early 1970s with the triangular covers paper for closed curves \cite{wetzel1970} and the sectorial covers paper \cite{wetzel1973}, posing conjectures for both the $30^{\circ}$ sector and the $30^{\circ}\!-\!60^{\circ}\!-\!90^{\circ}$ triangle. The sectorial conjecture was resolved by Panraksa and Wichiramala \cite{panraksa2020}, building on sectorial bounds \cite{wetzel2019} and drapeability arguments.\footnote{The worm problem and Wetzel's conjectures have also attracted popular attention; see, e.g., the science feature on the sector proof by van den Brandhof~\cite{nrc2020} in the Dutch newspaper \emph{NRC Handelsblad}.} A recent survey of the worm problem and of the program completed here is given by Movshovich \cite{movshovich2025}.

Structural ingredients for the triangle problem include the $\Lambda$-property and its consequences for interlaced support-line contacts, as well as drapeability, $\Lambda$-frame, and escape-path viewpoints \cite{maki2005,movshovich2017,movshovich2015,movshovich2020,movshovich2011,movshovich2021zarcs,movshovich2025frames}. Covering results for convex arcs and piecewise-linear arcs appear in work by Johnson, Poole, and Wetzel \cite{johnson2004}, Wichiramala \cite{wichiramala2010}, and others \cite{panraksa2007a,panraksa2007b,sroysang2008}.

Within the broader worm-problem landscape---seeking small regions covering every unit-length planar curve---Norwood, Poole, and Laidacker \cite{norwood1992} established influential upper bounds, and later unrestricted, nonconvex covers lowered the known area further \cite{norwoodpoole2003,ploymaklam2018}. These nonconvex covers should be distinguished from convex covers: in particular, the nonconvex cover of area $0.26007$ of Ploymaklam and Wichiramala \cite{ploymaklam2018} is not directly comparable to convex covers such as the sector or the triangle $T$, and the area statements in this paper are scoped to the convex triangular Wetzel cover. For convex covers, the bounds prior to the sector theorem were $0.232239<\alpha\le 0.27091$: the lower bound is due to Khandhawit, Pagonakis, and Sriswasdi \cite{khandhawit2013}, the upper bound to Wang \cite{wang2006}, and further bounds appear in \cite{wetzel2013}. Grechuk and Som-am \cite{grechuk2020} study a related lower-bound problem for convex covers of closed unit curves. Wetzel's survey \cite{wetzel2003} provides background and open directions. Classical covering themes, including sector and rhombus covers and the lineage of area bounds, trace back to Gerriets and Poole \cite{gerriets1974} and to foundational work of Lebesgue, Blaschke, and Besicovitch \cite{lebesgue1914,blaschke1916,besicovitch1965}. Besicovitch triangles that cover unit arcs were studied by Coulton and Movshovich \cite{coulton2006}.

\subsection{Organization}

Section~\ref{sec:prelim} sets up the support-line framework, the $\Lambda$-property, and the reduction to weak support-reduced obstructions. Section~\ref{sec:triangle} fixes the triangle $T$, the canonical orientations, and the escape predicate. Section~\ref{sec:cases} presents the case tree and proves Theorem~\ref{thm:main}; the optimization models and interval validation are in Appendix~\ref{app:opt}, and the auxiliary geometric proofs in Appendix~\ref{app:lemmas}. Section~\ref{sec:quant} records the certified computational scaling margin, and Section~\ref{sec:conclusion} concludes.

\subsection{Proof dependency checklist and terminology}\label{subsec:dependency-checklist}

Because the proof combines geometric reductions with machine-checked certificates, we list which statements are which.

\begin{itemize}
  \item \emph{Geometric statements, proved in the text or in Appendix~\ref{app:lemmas}:} the polygonal reduction and weak-obstruction compactness (Lemmas~\ref{lem:closed-contact-order}--\ref{lem:reduced-obstruction}), the support-contact order input and its consequences (Theorem~\ref{thm:imported-support-contact}, Proposition~\ref{prop:support-input}, Lemmas~\ref{lem:middle-skeleton}--\ref{lem:contact-alternatives}; the cyclic normal-fan order is tabulated in Table~\ref{tab:cyclic-order}), the wedge-shortening and tail-order enumerations (Lemmas~\ref{lem:wedge}--\ref{lem:case3-order-reduction}), the forced canonical placements (Lemmas~\ref{lem:canonical-tests} and~\ref{lem:forced-placements}), and the case-tree exhaustion (Proposition~\ref{prop:case-tree}).
  \item \emph{Computational statements, certified by the supplementary package:} the lower bounds $L_i^{\rm cert}>1$ for all $599$ closed models (Theorem~\ref{thm:socp-cert}, Proposition~\ref{prop:dual-certificate}, Tables~\ref{tab:interval-cert} and~\ref{tab:raw-order-cert}) and the row-to-placement soundness of the model rows (Lemma~\ref{lem:row-audit}). Both are validated by the one-command verifier of Appendix~\ref{app:repro}, which checks stored certificates and interval reports; rerunning the optimizers is not part of the proof.
  \item \emph{Branch-to-model traceability:} Table~\ref{tab:branch-model-ledger} in Appendix~\ref{app:branch-grammar} maps each branch to its model ids, chain order, and certificate log.
\end{itemize}

Throughout, a \emph{closed model} is a finite SOCP whose feasible set uses weak (non-strict) escape inequalities; the twelve \emph{representative models} are the terminal subcases displayed in the case tree (eleven notebook models plus the direct both-late model 1.1c); the $587$ \emph{raw-order models} certify the nonrepresentative Case~2 and Case~3 tail orders directly. ``Certified'' always refers to the interval-validated dual lower bounds of Appendix~\ref{app:opt}.

\section{Preliminaries}\label{sec:prelim}

An \emph{arc} is a rectifiable simple curve in $\R^2$; its length is denoted by $\ell(\gamma)$.
Support lines are always taken with respect to the arc or compact set under consideration.

\subsection{Support lines and directional distance}\label{subsec:support-distance}

For an angle $\theta \in (-\pi,\pi]$ we denote by
\[
e_\theta = (\cos\theta,\sin\theta), 
\qquad
n_\theta = (\sin\theta,-\cos\theta)
\]
the unit vector in direction $\theta$ and the unit normal obtained by rotating $e_\theta$ clockwise by $90^\circ$.
With this choice, the closed half-plane on the left of an oriented line in direction $\theta$ is described by an inequality of the form $\langle X,n_\theta\rangle\le h$.

Let $\gamma \subset \R^2$ be a compact set, in particular an arc image or a weak obstruction curve. A line $L$ is a \emph{support line} for $\gamma$ if there exists a closed
half-plane $H$ bounded by $L$ such that $\gamma \subset H$ and $L\cap\gamma\neq\emptyset$.
We say that an oriented support line has \emph{direction}~$\theta$ if it is parallel to $e_\theta$ and $\gamma$ lies in the
closed half-plane on its left.

\begin{definition}[Oriented support lines]
For each angle $\theta\in(-\pi,\pi]$ we define the oriented support line $L_\theta$ (when it exists) by
\[
L_\theta = \{ X\in\R^2 : \langle X, n_\theta\rangle = h_\theta\},
\]
where
\[
h_\theta = \max_{X\in\gamma} \langle X, n_\theta\rangle
\]
is the corresponding \emph{support offset}. Thus $\gamma$ lies in the closed half-plane
\[
H_\theta = \{X\in\R^2 : \langle X, n_\theta\rangle \le h_\theta\},
\]
and $L_\theta$ \emph{touches} $\gamma$ at a point $X\in\gamma$ if $X\in L_\theta$.
\end{definition}

With this convention,
\[
L_{0^\circ}=\{(x,y)\in\R^2: y=\min_{Z\in\gamma} y\},
\qquad
L_{180^\circ}=\{(x,y)\in\R^2: y=\max_{Z\in\gamma} y\}.
\]
In the normalized $\Lambda$-configuration used throughout, $L_{0^\circ}$ is the lower and $L_{180^\circ}$ the upper horizontal support.

\begin{definition}[Directional distance]
For a point $X=(x,y)\in\R^2$ and an angle $\theta\in(-\pi,\pi]$, the \emph{directional distance} of $X$ in
direction~$\theta$ is
\[
d(X,\theta) \;=\; \langle X, e_\theta\rangle \;=\; x\cos\theta + y\sin\theta.
\]
For a vector $\overrightarrow{AB}=B-A$, we write
\[
d(\overrightarrow{AB},\theta) \;=\; d(B-A,\theta) \;=\; \langle B-A, e_\theta\rangle.
\]
\end{definition}

An oriented support line $L_\theta$ \emph{supports} a triangle placement $T$ if one of its edges lies on $L_\theta$; escape predicates and case-split inequalities are expressed via $d(\overrightarrow{AX},\theta)$, matching the optimization constraints of Appendix~\ref{app:opt}.
The support and escape vectors are deliberately different: support functions use $n_\theta$, escape predicates use $e_\theta$.

\begin{table}[t]
\centering
\caption{Vector convention for the main special angles.}\label{tab:angle-vectors}
\scriptsize
\setlength{\tabcolsep}{5pt}
\renewcommand{\arraystretch}{1.08}
\begin{tabular}{@{}ccc@{}}
\toprule
\textbf{Angle} & \textbf{Escape vector $e_\theta$} & \textbf{Support normal $n_\theta$} \\
\midrule
$0^\circ$ & $(1,0)$ & $(0,-1)$ \\
$90^\circ$ & $(0,1)$ & $(1,0)$ \\
$120^\circ$ & $(-1/2,\sqrt3/2)$ & $(\sqrt3/2,1/2)$ \\
$150^\circ$ & $(-\sqrt3/2,1/2)$ & $(1/2,\sqrt3/2)$ \\
$180^\circ$ & $(-1,0)$ & $(0,1)$ \\
\bottomrule
\end{tabular}
\end{table}

\subsection{Lambda-property and special contact points}\label{subsec:lambda}

Let $\gamma\subset\R^2$ be a simple rectifiable arc with endpoints $P,Q$.
By the \emph{$\Lambda$-property} of simple arcs \cite[Thm.~6]{alexander2019} (see also \cite[Lem.~2.1]{panraksa2020} and the parallel-support discussion in~\cite{movshovich2015}), there exist points
\[
P_{-3},\;P_0,\;P_3\in \gamma
\]
occurring along $\gamma$ in this order, and two parallel support lines $L_{0^\circ},L_{180^\circ}$, such that
\begin{itemize}
  \item $\gamma$ lies between $L_{0^\circ}$ and $L_{180^\circ}$, with $L_{0^\circ}$ the lower support and $L_{180^\circ}$ the upper support;
  \item $L_{0^\circ}$ meets $\gamma$ at the two lower contact points $P_{-3}$ and $P_3$;
  \item $L_{180^\circ}$ meets $\gamma$ at the upper contact point $P_0$;
  \item the order of these contacts along $\gamma$ is $P_{-3}$--$P_0$--$P_3$.
\end{itemize}
We normalize the picture by choosing coordinates so that $L_{0^\circ}$ is the line $y=0$ and $L_{180^\circ}$ is the
line $y=H$ for some gap $H>0$, and by reflecting in a vertical line if necessary so that
\[
  x(P_{-3})\le x(P_3).
\]
In the contradiction argument below this positive-gap normalization loses no obstruction: a unit line segment is contained in $T$, since the hypotenuse has length
\[
c=\frac{3+4\sqrt3}{9}>1.
\]
Thus any noncovered unit arc is not a straight segment, and the horizontal $\Lambda$-normalization may be taken with $H>0$.
With this normalization the arc $\gamma$ is naturally decomposed into three
subarcs:
\begin{itemize}
  \item the left tail $\gamma_{P P_{-3}}$ from the left endpoint $P$ to the first lower contact $P_{-3}$,
  \item the middle segment $\gamma_{P_{-3}P_3}$, which crosses from $L_{0^\circ}$ at $P_{-3}$ to $L_{180^\circ}$ at $P_0$
        and back to $L_{0^\circ}$ at $P_3$, and
  \item the right tail $\gamma_{P_3 Q}$ from the last lower contact $P_3$ to the right endpoint $Q$.
\end{itemize}

Figure~\ref{fig:lambda} illustrates this basic configuration.

\begin{figure}[t]
  \centering
  \includegraphics[width=0.6\columnwidth]{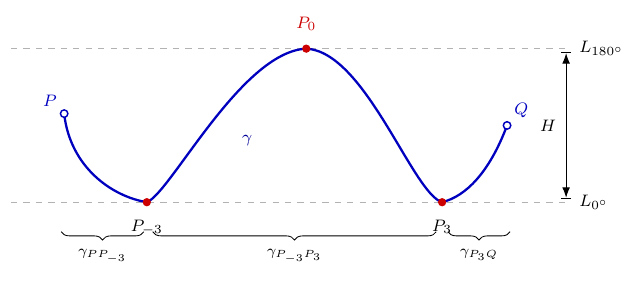}
  \caption{The $\Lambda$-property for a simple arc $\gamma$: parallel horizontal support lines
  $L_{0^{\circ}}$ (lower) and $L_{180^{\circ}}$ (upper) touch $\gamma$ at three points
  $P_{-3},P_0,P_3$, decomposing $\gamma$ into a left tail $\gamma_{P P_{-3}}$, a middle segment
  $\gamma_{P_{-3}P_3}$ passing through $P_0$, and a right tail $\gamma_{P_3 Q}$.}
  \label{fig:lambda}
\end{figure}

The angle-support form of the $\Lambda$-property (see \cite[Thm.~6]{alexander2019} and \cite[Lem.~2.1]{panraksa2020}) gives, for any $\theta\in(0,\pi)$, at least one pair of support lines forming angle $\theta$ and touching
$\gamma$ in an alternating triple---one line meets $\gamma$ twice and the other once, with the three contacts in order along $\gamma$---and when two such pairs occur their angle vertices interlace along the arc.

Since the $\Lambda$-property is load-bearing, we record its provenance precisely. We use exactly two statements about a simple arc: the parallel-support form displayed above and the angle-support interlacing form of the preceding paragraph. Both are proved for simple arcs in \cite[Thm.~6]{alexander2019}; the parallel-support form is also recorded in the published account \cite[Lem.~2.1]{panraksa2020}. We apply them only to the simple polygonal approximating arcs supplied by the Wetzel--Wichiramala reduction \cite[Cor.~5]{wetzel2010}. The support-contact order input (Theorem~\ref{thm:imported-support-contact}) and all of its consequences are proved self-containedly in Appendix~\ref{app:unimodality}, so no unpublished source is load-bearing for the contact-order machinery.

\medskip\noindent
Following \cite{panraksa2020}, fix the ten \emph{special angles}
\[
\theta \in \{\pm 30^{\circ},\; \pm 60^{\circ},\; \pm 90^{\circ},\; \pm 120^{\circ},\; \pm 150^{\circ}\},
\]
all measured in degrees with respect to the horizontal axis.
The closed models use the following finite witness vocabulary. The global support-contact order is applied only to the horizontal $\Lambda$ contacts and to the special-angle roles above, especially the $\pm120^\circ$ and $\pm150^\circ$ roles entering Proposition~\ref{prop:support-input}. The delimiter directions are
\[
\Theta_{\rm del}=\{0^\circ,180^\circ,\pm90^\circ\},
\]
which fix the lower and upper $\Lambda$ supports and the vertical tail contacts. Subcase~1.1 additionally records branch-local wedge witnesses in the finite set
\[
\Theta_{\rm wedge}=\{\pm15^\circ,\pm105^\circ\},
\]
equivalently the line/support labels used by the two steep sides of $T_W$ and $T_{W'}$. These wedge witnesses are carried through compactness as selected delimiter or side-incidence contacts; they are not used to infer any additional global support-contact order. Finally, the affine escape rows use the finite set of vector directions
\[
\Theta_{\rm esc}=\{0^\circ,180^\circ,30^\circ,150^\circ,-30^\circ,-90^\circ,-150^\circ,-75^\circ,-105^\circ\}.
\]
Directions in $\Theta_{\rm esc}$ are vectors in the row $d(\overrightarrow{AX},\theta)\ge b$, not necessarily support-line labels. Thus every support, delimiter, wedge, and escape direction appearing in the closed SOCP models is fixed before the compactness subsequence is taken.
For each arc under discussion we fix an arclength parametrization, still denoted
$\gamma:[0,\ell]\to\R^2$, in the order from $P$ to $Q$.
If $I\subset[0,\ell]$ is a compact parameter interval and $L_\theta$ is a support line, set
\[
  C_\theta(I)=\{t\in I:\gamma(t)\in L_\theta\}.
\]
When $C_\theta(I)\ne\emptyset$, define the first and last contacts on $I$ by
\[
  X_\theta^-(I)=\gamma(\min C_\theta(I)),\qquad
  X_\theta^+(I)=\gamma(\max C_\theta(I)).
\]
For $I=[0,\ell]$ we write simply $C_\theta$ and $X_\theta^\pm$.
For nondegenerate branches these endpoints provide canonical witnesses; in closed models, any selected witness on the exposed support face may be used, and coincident or exposed-face cases are weak-order degeneracies.
This is compatible with the support-contact parameter $T(\theta)$ used by Wichiramala~\cite[Lemmas~1--3]{wichiramala2019a}.
We use the following closed witness convention throughout. A named contact is a selected witness parameter on a prescribed closed subarc interval; if a support face is a segment, the witness is the first or last contact on that interval when an endpoint delimiter is required, and otherwise any selected point on the exposed face. If two selected witnesses coincide, or if several witnesses lie on one exposed segment, their recorded order is the corresponding weak equality. Passing from strict notebook inequalities to closed model rows always enlarges the feasible obstruction class.
The limiting form used below is the following closed finite-order principle. For a compact convex hull $K=\conv(\gamma)$ and the support normal $n_\theta$, write
\[
  C_{\theta,K}^{\varepsilon}(I)=
  \{t\in I: h_K(n_\theta)-\langle \gamma(t),n_\theta\rangle\le \varepsilon\},
\]
where $h_K$ is the support function of $K$.

\begin{lemma}[Closed witness-contact order]\label{lem:closed-contact-order}
Let $\gamma_n:[0,1]\to\R^2$ be constant-speed parametrizations of simple polygonal arcs of length at most one, extended constantly if necessary, and suppose that $\gamma_n$ converges uniformly to a rectifiable limiting curve $\gamma$. Let $K_n=\conv(\gamma_n)$ converge in the Hausdorff metric to $K=\conv(\gamma)$. Fix finitely many support normals $u_j$, closed parameter intervals $I_j$, and selected witness contact parameters $t_j^{(n)}\in I_j$ satisfying
\[
\langle\gamma_n(t_j^{(n)}),u_j\rangle=h_{K_n}(u_j).
\]
If the selected witnesses satisfy a finite list of weak parameter-order relations $t_i^{(n)}\le t_j^{(n)}$, then after passing to a subsequence $t_j^{(n)}\to t_j\in I_j$, the points $P_j=\gamma(t_j)$ lie on the support faces of $K$ in directions $u_j$, and every recorded weak order relation is preserved.
\end{lemma}

The limiting curve may have additional support contacts not selected by these witness data. They are irrelevant for the finite certificate, since the closed SOCP models are built from the selected witness contacts and therefore enlarge the obstruction class rather than shrink it.

We use the following consequence throughout. For the approximating polygonal arcs, the endpoint-contact function of Theorem~\ref{thm:imported-support-contact} is unimodal over a full period anchored at the hull vertex of smallest parameter, and at the lower horizontal support its two values are the parameters of $P_{-3}$ and $P_3$. Anchoring the sweep at these two lower-support values forces the one-sided role alternatives of Lemma~\ref{lem:contact-alternatives}: in particular, on the positive side, if the selected $L_{120^\circ}$ witness is inner (strictly before $P_3$ in parameter), then the selected $L_{150^\circ}$ witness is inner as well, since the opposite pattern would create a second local maximum of the contact function (Appendix~\ref{app:lemmas}). The reflected statement holds on the negative side. This is the imported input used in Lemma~\ref{lem:contact-alternatives}.

The symbols $P_k$ below are role labels selected as support witnesses on specified subarcs, not arbitrary choices from a multiple-contact set.
If a branch requires an endpoint witness, the notation uses $X_\theta^-(I)$ or $X_\theta^+(I)$ for the first or last contact of $L_\theta$ on the corresponding parameter interval $I$.
Coincident contacts and exposed support segments are treated as weak-order degeneracies of the same branch.
Depending on the branch, additional tail points are also introduced. The integral index set used in this paper is
\[
k\in\{-8,-7,-6,-5,-4,-3,-2,-1,0,1,2,3,4,5,6,7,8\},
\]
with not all indices active in every subcase. Subcase~1.1 also uses the auxiliary half-index labels
\[
P_{\pm2.5},\qquad P_{\pm3.5},\qquad P_{\pm6.5}.
\]
These are not additional global special-angle roles. They are endpoint witnesses introduced only after the canonical $T_W$ test has fixed the relevant wedge-side incidences on the active closed subarcs. The roles are:
\begin{itemize}
  \item $P_{-3},P_0,P_3$ are the three horizontal contacts from the $\Lambda$-configuration;
  \item $P_{\pm6}$ are vertical-support contacts ($L_{\pm90^\circ}$);
  \item $P_{\pm2},P_{\pm1}$ are inner contacts on $\gamma_{P_{-3}P_3}$ used in Cases~1 and~2;
  \item $P_{\pm7},P_{\pm8}$ are outer contacts that drive the $\pm120^\circ$ and $\pm150^\circ$ branches;
  \item $P_{\pm4},P_{\pm5}$ are auxiliary tail contacts used to enumerate and certify the Case~2 and Case~3 raw tail orders.
  \item In Subcase~1.1a, $P_{-2.5}$ and $P_{2.5}$ are the two $T_W$ side-incidence contacts on the middle intervals between $P_{-3},P_{-2}$ and between $P_2,P_3$, while $P_{-3.5}$ and $P_{3.5}$ are the corresponding tail escape contacts between $P_{-6},P_{-3}$ and between $P_3,P_6$. Their branch order is
        \[
        P_{-6},P_{-3.5},P_{-3},P_{-2.5},P_{-2}
        \quad\text{and}\quad
        P_2,P_{2.5},P_3,P_{3.5},P_6 .
        \]
        In Subcase~1.1b, $P_{-6.5}$ is the one-sided $T_W$ side-incidence contact selected on the left tail before the vertical contact; the reflected/right analogue is treated by the same convention. In Subcase~1.1c, both late contacts are retained: $P_{-6.5}$ lies on the left $T_W$ side before $P_{-6}$, and $P_{6.5}$ lies on the right $T_W$ side after $P_6$. The branch order is
        \[
        P_{-6.5},P_{-6},P_{-3}
        \quad\text{and}\quad
        P_3,P_6,P_{6.5}.
        \]
\end{itemize}

Along the middle segment, indices are fixed so that
\[
P_{-3},\;P_{-2},\;P_{-1},\;P_0,\;P_1,\;P_2,\;P_3
\]
occur in this order (Section~\ref{subsec:monotonicity}).
Figure~\ref{fig:indexmap} gives a schematic index map used in later case splits.

\begin{figure}[htbp]
    \centering
    \includegraphics[width=0.95\columnwidth]{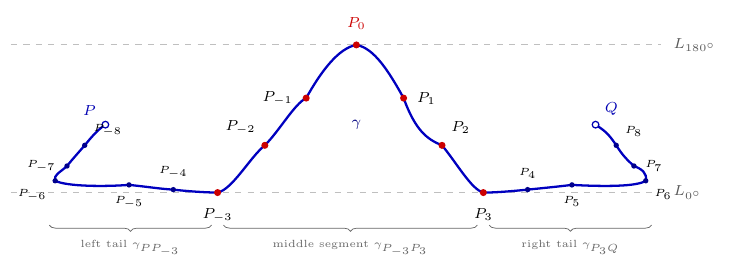}
    \caption{Index map of contact points on a unit arc~$\gamma$.
    The arc runs between two horizontal support lines $L_{0^\circ}$ and $L_{180^\circ}$.
    The red points $P_{-3}$, $P_0$, $P_3$ are the $\Lambda$-contacts;
    the orange points $P_{\pm1},P_{\pm2}$ record the middle contact order used in Cases~1 and~2;
    the blue tail points $P_{\pm k}$ ($k=4,\dots,8$) record the vertical, outer, and auxiliary tail contacts used in later branch splits.
    Braces indicate the decomposition into left tail, middle segment, and right tail.}
    \label{fig:indexmap}
\end{figure}

The facts used later are:
\begin{itemize}
  \item for each special angle $\theta\in\{\pm30^\circ,\pm60^\circ,\pm90^\circ,\pm120^\circ,\pm150^\circ\}$ there is at
        least one support contact, and any named contact is selected as a support witness on its assigned subarc, with endpoint witnesses used only when the branch role specifies one;
  \item the selected middle support-contact skeleton occurs in weak normal-fan order; in particular,
        \[
          P_{-3},\,P_{-2},\,P_{-1},\,P_0,\,P_1,\,P_2,\,P_3
        \]
        occur in this weak parameter order (Lemma~\ref{lem:middle-skeleton}, Figure~\ref{fig:monotone}).
\end{itemize}

\subsection{Weak obstructions and support-role data}\label{subsec:hull-survey}

The finite case analysis needs only limiting support-role data and length lower semicontinuity, supplied by the following compactness object.

\begin{definition}[Weak obstruction]\label{def:weak-obstruction}
A \emph{weak obstruction of length at most one} is a tuple
\[
\bigl(\eta;\,t_1,\ldots,t_N\bigr),
\]
where $\eta:[0,1]\to\R^2$ is a uniform limit of constant-speed simple polygonal arcs of length at most one (extended constantly when the length is smaller), and $t_1,\ldots,t_N$ are finitely many selected witness parameters, each lying in its prescribed closed interval $I_j\subseteq[0,1]$ and each satisfying the support-face, delimiter, side-incidence, or escape-row condition assigned to it by the stabilized role data in the finite vocabulary fixed above: the special-angle roles, the delimiter directions $\Theta_{\rm del}$, the branch-local wedge witnesses $\Theta_{\rm wedge}$ when the branch uses them, and the named escape-vector witnesses $\Theta_{\rm esc}$ recorded in the SOCP rows. The limiting curve $\eta$ is not assumed to be simple, and its length is bounded by lower semicontinuity of arclength. If $C$ is a compact convex set, the weak obstruction is \emph{noncovered} by $C$ when no congruent copy of $C$ contains $\conv(\eta)$. We often refer to the weak obstruction by its curve $\eta$, the witness list being understood as part of the data.
\end{definition}

\begin{lemma}[Weak-hull reduction]\label{lem:reduction}
Let $C\subset\R^2$ be a compact convex set.
If $C$ fails to cover some unit arc, then there is a weak obstruction $\gamma$ of length at most one whose convex hull is a Hausdorff limit of convex hulls of simple polygonal arcs of length at most one and remains noncovered by $C$.
\end{lemma}

\begin{proof}
See Appendix~\ref{app:lemmas}. Starting from a noncovered simple polygonal unit arc, compactness gives a closed limiting hull and a limiting weak obstruction with length at most one. Only the existence of such a noncovered limiting hull is used below.
\end{proof}

When a weak obstruction is translated and rotated so that the limiting horizontal $\Lambda$-configuration and the distinguished outer support contacts from Section~\ref{subsec:lambda} are expressed in the fixed notation of this paper, we shall call the resulting placement a \emph{normalized weak obstruction}. The $\Lambda$-configuration is chosen on the approximating simple polygonal arcs. Each approximant is first put into its own $\Lambda$-frame by a rigid motion, then a subsequence of these normalized approximants and normalized hulls is passed to the limit. After passing to a further subsequence, the normalizing rigid motions converge: the rotations have a convergent subsequence, the possible reflection choice is finite, and the translations are fixed by the normalization $P_{-3}^{(n)}=0$. Hence the normalized limiting hull is the image of the unnormalized limiting hull under a limiting congruence. Since congruence preserves arclength, convex-hull noncoverage, and all support-order statements after relabeling of the fixed finite directions, the limiting normalized hull is a congruent representative of the same obstruction. The simple-arc $\Lambda$ theorem is therefore used only on the approximating simple arcs, and Lemma~\ref{lem:closed-contact-order} passes the selected finite role data to the weak limit.

For the finite case analysis we use a slightly narrower representative class. The \emph{weak support-role data} of a normalized weak obstruction records, for each special angle used in the case tree, which closed named subarc has a selected support witness; it does not record uniqueness of the contact or any strict escape inequality. The named subarcs are delimited by selected witness contacts in the fixed line directions $0^\circ,180^\circ,\pm90^\circ$, the special line directions appearing in the case tree, and the branch-local wedge directions only when the Subcase~1.1 wedge routing requires them; when the branch requires an endpoint delimiter, that delimiter is a selected endpoint witness. Thus the data include the delimiter choices as well as the incidence choices. Under arclength and Hausdorff convergence, support faces in these fixed line directions and the corresponding selected witness parameters have convergent subsequences, so the finite weak incidences are closed after passing to the stabilized delimiter data. These weak incidences are the only support-role data fixed in the reduction below.

Let $K$ be one normalized limiting hull and let $\mathcal R$ be one stabilized finite weak support-role assignment, with witness slots $j=1,\ldots,N$, prescribed closed parameter intervals $I_1,\ldots,I_N\subseteq[0,1]$, and the delimiter convention for all named subarcs used in the case tree. We write
\[
\mathcal A(K,\mathcal R)\;\subseteq\; C\bigl([0,1],\R^2\bigr)\times I_1\times\cdots\times I_N
\]
for the class of weak obstructions $(\eta;\,t_1,\ldots,t_N)$ with the following properties:
\begin{itemize}
  \item $\eta$ is a uniform limit of constant-speed simple polygonal arcs of length at most one which realize the same stabilized finite weak support-role assignment $\mathcal R$ and whose own witness parameters converge to $t_1,\ldots,t_N$, with constant tails added when the approximating length is less than one;
  \item $\conv(\eta)=K$;
  \item for each slot $j$, the witness parameter satisfies $t_j\in I_j$, the witness point $\eta(t_j)$ satisfies the support-face, delimiter, or side-incidence condition that $\mathcal R$ assigns to slot $j$, and the finitely many weak parameter-order relations $t_i\le t_j$ recorded in $\mathcal R$ hold.
\end{itemize}
The class is used only after a nonempty stabilized assignment has been selected from approximating arcs. It is a closed subset of the product of the uniform topology on equi-Lipschitz curves and the Euclidean topology on $I_1\times\cdots\times I_N$: convex-hull equality is equality of support functions, each role incidence is a closed condition on the pair $(\eta,t_j)$ by Lemma~\ref{lem:closed-contact-order}, and each recorded weak order relation is a closed scalar inequality. Thus all fixed witness data and their closed intervals are visibly part of the compact object, and ``curve plus witness parameters'' is the object to which the compactness arguments below apply.

We record the sequencing of the witness slots explicitly. Support-role and delimiter witness slots are fixed before the case tree is entered, since they arise from the convex-hull support geometry of the approximating arcs. Branch-local escape witness slots are added only after a finite branch has fixed the corresponding canonical placement and closed affine row; they are carried in the same product compactness space as the support-role witnesses, but their defining conditions are closed affine incidence or escape inequalities rather than support-face conditions. Compactness applies in the enlarged product space because every added condition is closed.

A normalized weak obstruction $\gamma$ is called \emph{support-reduced} if, after the weak support-role data have been fixed, it is chosen as a length minimizer in $\mathcal A(K,\mathcal R)$. This is a representative choice for the finite support skeleton, not an assertion that every point of the corresponding subarc is ordered by one coordinate.

\begin{lemma}[Closed weak-order selection]\label{lem:closed-order-selection}
Let $K$ be a compact convex hull obtained as the Hausdorff limit of convex hulls of simple polygonal arcs $\gamma_n$ realizing one stabilized finite weak support-role assignment $\mathcal R$ in the witness vocabulary fixed above, including the selected witness delimiter contacts of the named subarcs. Then $\mathcal A(K,\mathcal R)$ is compact in the product topology of uniform convergence of curves and convergence of witness parameters, and it contains a length-minimizing representative. For any such representative selected through the closed witness-contact convention, the named support-face and delimiter contacts whose directions are part of the stabilized support-role assignment $\mathcal R$ occur, on each exposed boundary chain, in weak normal-fan order. Branch-local wedge and escape witnesses are not asserted to belong to this global order; they satisfy the closed affine incidence, guard, and escape rows recorded for their branch. Coincident contacts and exposed support segments are allowed and are interpreted as weak equalities among selected witness contacts.
\end{lemma}

\begin{proof}
See Appendix~\ref{app:lemmas}. The proof uses Wichiramala's endpoint-contact order for simple polygonal arcs and the closedness of the finite weak witness-role incidences under arclength/Hausdorff limits.
\end{proof}

\begin{lemma}[Reduction to weak support-reduced obstructions]\label{lem:reduced-obstruction}
If $T$ fails to cover some unit arc, then $T$ fails to cover a weak support-reduced normalized obstruction $\gamma$ with $\ell(\gamma)\le 1$ and stabilized weak support-role data for all support, delimiter, wedge, and escape witnesses used in the finite case tree.
\end{lemma}

\begin{proof}
See Appendix~\ref{app:lemmas}. The reduction first passes to a noncovered limiting hull for a noncovered arc of length at most one, and then minimizes length within the resulting hull and fixed weak support-role data.
\end{proof}

\subsection{Support-contact skeleton and ordering of contacts}\label{subsec:monotonicity}

The next proposition records the external support-contact input used in the later case analysis. To avoid mixing angle conventions, we use the following terminology. A \emph{line direction} is the angle $\theta$ in the oriented support line $L_\theta$; its chosen support normal is $n_\theta=(\sin\theta,-\cos\theta)$; an \emph{escape direction} is the vector angle appearing in a row $d(\overrightarrow{AX},\theta_{\rm esc})\ge b$. In Proposition~\ref{prop:support-input} we reserve $\vartheta$ for the cyclic boundary-sweep parameter used to import the support-contact unimodality. For the non-horizontal special contacts in the table below, $\vartheta$ has the same numerical value as the line direction. The two formal endpoint symbols $\vartheta_{\rm start}$ and $\vartheta_{\rm end}$ are not line directions and in particular are not the vertical support contacts $P_{\pm6}$; they denote the two anchor positions of the same upper-boundary sweep at the first and last contacts on the lower horizontal support $L_0$, where the contact function takes the values $t(P_{-3})$ and $t(P_3)$.

\begin{table}[t]
\centering
\caption{Angle convention for the finite support-contact input. The line direction is the angle in $L_\theta$, while $\vartheta$ is the cyclic sweep parameter used in the imported support-contact order.}\label{tab:support-angle-convention}
\scriptsize
\setlength{\tabcolsep}{3pt}
\renewcommand{\arraystretch}{1.08}
\begin{tabularx}{\textwidth}{@{}p{1.15cm}p{4.2cm}p{1.65cm}p{2.75cm}X@{}}
\toprule
\textbf{Role} & \textbf{Selected support face} & \textbf{Line direction} & \textbf{Support normal} & \textbf{Sweep parameter $\vartheta$} \\
\midrule
$P_{-3}$ & first lower $\Lambda$ contact on $L_0$ & $0^\circ$ & $n_{0^\circ}=(0,-1)$ & $\vartheta_{\rm start}$, lower-left anchor of the sweep. \\
$P_{-2}$ & selected inner $L_{-120^\circ}$ contact & $-120^\circ$ & $n_{-120^\circ}$ & $-120^\circ$. \\
$P_{-1}$ & selected inner $L_{-150^\circ}$ contact & $-150^\circ$ & $n_{-150^\circ}$ & $-150^\circ$. \\
$P_0$ & upper $\Lambda$ contact on $L_{180^\circ}$ & $180^\circ$ & $n_{180^\circ}=(0,1)$ & $180^\circ$. \\
$P_1$ & selected $L_{150^\circ}$ contact, inner or right-tail & $150^\circ$ & $n_{150^\circ}$ & $150^\circ$. \\
$P_2$ & selected $L_{120^\circ}$ contact, inner or right-tail & $120^\circ$ & $n_{120^\circ}$ & $120^\circ$. \\
$P_3$ & last lower $\Lambda$ contact on $L_0$ & $0^\circ$ & $n_{0^\circ}=(0,-1)$ & $\vartheta_{\rm end}$, lower-right anchor of the sweep. \\
\bottomrule
\end{tabularx}
\end{table}

The external input from Wichiramala~\cite[Lemmas~1--3]{wichiramala2019a} is the following unimodality statement for the endpoint-contact function. We state it faithfully to the source; a self-contained proof is included in Appendix~\ref{app:unimodality} for completeness, so the present paper does not rely on the unpublished source.

\begin{theorem}[Support-contact unimodality; {\cite[Lemmas~1--3]{wichiramala2019a}}]\label{thm:imported-support-contact}
Let $\gamma:[0,\ell]\to\R^2$ be a simple polygonal arc that is not a line segment. For each direction $\theta$ let
\[
C_\theta=\{t\in[0,\ell]:\gamma(t)\in L_\theta\},
\qquad
T(\theta)=\{\min C_\theta,\ \max C_\theta\}
\]
record the first and last contact parameters of the oriented support line $L_\theta$. Then:
\begin{enumerate}
  \item the graph of $T$ over one period of directions is a step function whose levels are the parameters of the vertices of $\conv(\gamma)$;
  \item for every $s\in[0,\ell]$, the direction set $\{\theta:\gamma(s)\in L_\theta\}$ is empty or a closed arc of angular length less than $\pi$;
  \item there is a direction $\hat\theta$ exposing the hull vertex of smallest parameter such that, over the period $[\hat\theta,\hat\theta+2\pi]$, every selection from $T$ is monotone nondecreasing and then monotone nonincreasing (unimodal).
\end{enumerate}
\end{theorem}

Here a \emph{selection} from $T(\theta)=\{\min C_\theta,\max C_\theta\}$ means a choice of one endpoint-contact branch, with directions exposing an edge interpreted by one-sided limits inside the normal cone of that edge; arbitrary switching between the two endpoint values at a single exposed-edge direction is not intended. This endpoint-branch form is the only form of selection consumed by Proposition~\ref{prop:support-input} and by the proofs in Appendix~\ref{app:unimodality}.

We emphasize what Theorem~\ref{thm:imported-support-contact} does \emph{not} assert: no monotonicity is claimed along a proper subinterval of directions, because such a subinterval may contain the unimodal peak, and the set of directions whose contact lies in a prescribed parameter interval may consist of up to two arcs. The order and exclusion statements used below are therefore always derived from the unimodal shape of $T$ over the \emph{full} period, anchored at the lower-support values; see the proofs of Proposition~\ref{prop:support-input} and Lemma~\ref{lem:contact-alternatives} in Appendix~\ref{app:lemmas}. Table~\ref{tab:cyclic-order} in Appendix~\ref{app:unimodality} records the explicit cyclic normal-fan order of the exposed faces consumed by these arguments, in the orientation convention of Section~\ref{subsec:support-distance}, together with the convention for exposed segments and coincident contacts. Lemma~\ref{lem:closed-contact-order} passes the resulting finite witness-contact data to weak limiting obstructions.

\begin{proposition}[Finite support-contact order input]\label{prop:support-input}
Let $\gamma$ be a weak support-reduced normalized obstruction, and let the selected support contacts be the witness contacts carried through the closed limiting convention of Section~\ref{subsec:lambda}. Consider the finite upper-boundary sweep with sweep parameters
\[
\vartheta_{\rm start},\ -120^\circ,\ -150^\circ,\ 180^\circ,\ 150^\circ,\ 120^\circ,\ \vartheta_{\rm end}.
\]
\begin{enumerate}
\item On the positive side, the selected contacts on $L_{150^\circ}$ and $L_{120^\circ}$ either both lie on the middle chain (inner roles $P_1,P_2$), or the $L_{150^\circ}$ contact lies on the middle chain while the $L_{120^\circ}$ contact lies on the right tail (roles $P_1,P_7$), or both lie on the right tail (roles $P_8,P_7$). The mixed configuration in which the $L_{150^\circ}$ contact is strictly outer while the $L_{120^\circ}$ contact is strictly inner cannot occur; coincidence of a selected witness with $P_3$ is a weak-order degeneracy of an adjacent alternative. The reflected statement holds on the negative side.
\item Whenever the roles on a given side are inner, the corresponding selected witnesses occur in weak parameter order; in the branch in which all four $L_{\pm150^\circ}$, $L_{\pm120^\circ}$ roles are inner,
\[
P_{-3}\preceq P_{-2}\preceq P_{-1}\preceq P_0
\preceq P_1\preceq P_2\preceq P_3.
\]
On a side whose roles are outer, the corresponding tail witnesses are not asserted to precede $P_3$ (resp.\ follow $P_{-3}$) in parameter; their admissible weak orders are enumerated exhaustively in Lemmas~\ref{lem:case2-order-reduction} and~\ref{lem:case3-order-reduction}.
\end{enumerate}
\end{proposition}

\begin{proof}
See Appendix~\ref{app:lemmas}.
\end{proof}

The following lemmas are immediate consequences of Proposition~\ref{prop:support-input} and the branch definitions.

\begin{lemma}[Ordered middle support skeleton]\label{lem:middle-skeleton}
Let $\gamma$ be a weak support-reduced normalized obstruction. Then the selected middle support witnesses---that is, the witnesses among $P_{\pm2},P_{\pm1}$ whose roles are inner in the sense of Proposition~\ref{prop:support-input}, together with $P_{-3},P_0,P_3$---satisfy, allowing coincident contacts or exposed support segments, the weak parameter order
\[
P_{-3}\preceq P_{-2}\preceq P_{-1}\preceq P_0\preceq P_1\preceq P_2\preceq P_3,
\]
restricted to the witnesses that are present in the branch under consideration. Moreover, the corresponding exposed faces on the upper boundary chain of $K=\conv(\gamma)$ occur in the same weak normal-fan order. No assertion is made that every point of the subarc $\gamma_{P_{-3}P_3}$ has monotone $x$-coordinate, and no parameter relation to $P_{\pm3}$ is asserted for tail witnesses.
\end{lemma}

\begin{proof}
See Appendix~\ref{app:lemmas}.
\end{proof}

\begin{lemma}[Order of special contacts]\label{lem:order}
Along $\gamma_{P_{-3}P_3}$ the contact points among
\[
P_{-3}, P_{-2}, P_{-1}, P_0, P_1, P_2, P_3
\]
that are present in the branch under consideration occur in this order, with no interleaving from contacts on the tails.
\end{lemma}

\begin{proof}
See Appendix~\ref{app:lemmas}. The argument combines Lemma~\ref{lem:middle-skeleton} with the angle-support form of the $\Lambda$-property and the support-witness convention above.
\end{proof}

\begin{lemma}[Support-contact alternatives]\label{lem:contact-alternatives}
Let $\gamma$ be a weak support-reduced normalized obstruction, with contacts selected as support witnesses as above.
On the positive side, the active $L_{150^\circ}$ contact has role $P_1$ on the middle subarc $\gamma_{P_0P_3}$ or role $P_8$ on the right tail $\gamma_{P_3Q}$, and the active $L_{120^\circ}$ contact has role $P_2$ on $\gamma_{P_0P_3}$ or role $P_7$ on $\gamma_{P_3Q}$.
Moreover, the mixed pair $(P_8,P_2)$ cannot occur: if $L_{120^\circ}$ has the inner role $P_2$, then $L_{150^\circ}$ has the inner role $P_1$.

By reflection, $L_{-150^\circ}$ has role $P_{-1}$ on $\gamma_{P_{-3}P_0}$ or role $P_{-8}$ on the left tail $\gamma_{PP_{-3}}$, and $L_{-120^\circ}$ has role $P_{-2}$ on $\gamma_{P_{-3}P_0}$ or role $P_{-7}$ on $\gamma_{PP_{-3}}$.
The reflected mixed pair $(P_{-8},P_{-2})$ cannot occur.
\end{lemma}

\begin{proof}
See Appendix~\ref{app:lemmas}.
\end{proof}

Figure~\ref{fig:monotone} illustrates the ordered support-contact skeleton used in the finite case analysis.

\begin{figure}[t]
\centering
\includegraphics[width=0.85\columnwidth]{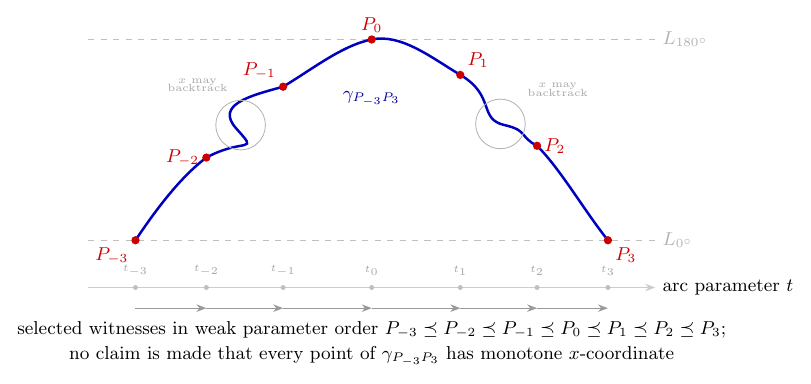}
\caption{Ordered middle support-contact skeleton: the selected support contacts appear in weak normal-fan and parameter order $P_{-3}, P_{-2}, P_{-1}, P_0, P_1, P_2, P_3$. The statement concerns these named support contacts, not every point of the intervening subarc: as the two magnified wiggles show, the $x$-coordinate of $\gamma$ may backtrack between consecutive contacts.}
\label{fig:monotone}
\end{figure}

\subsection{Wedge-shortening principle}

The following triangle-inequality lemma is used only for the Subcase~1.1 $T_W$ tail alternative. It is not used to route Case~2 or Case~3 raw tail orders into representative orders.

\begin{lemma}[Wedge-shortening]\label{lem:wedge}
Let $W$ be a closed planar wedge with apex $V$ and boundary rays $\ell_1$ and $\ell_2$. Fix points $X\in \ell_1$ and $Z\in \ell_2$, and for $Y\in \partial W=\ell_1\cup\ell_2$ set
\[
\phi(Y)=|XY|+|YZ|.
\]
Then
\[
|XY|+|YZ|\ge |XZ|.
\]
Equality can occur only in the usual collinear case $Y\in[X,Z]$. Consequently, whenever a closed branch model permits the degenerate contact obtained by deleting the boundary kink $X$--$Y$--$Z$, replacing that kink by the segment $X$--$Z$ does not increase the polygonal length.
\end{lemma}

\begin{proof}
See Appendix~\ref{app:lemmas}. No assertion is made about endpoint minimization on an arbitrary constrained subinterval of a wedge side.
\end{proof}

Figure~\ref{fig:wedge} illustrates the wedge-shortening principle.

\begin{figure}[t]
\centering
\includegraphics[width=0.6\columnwidth]{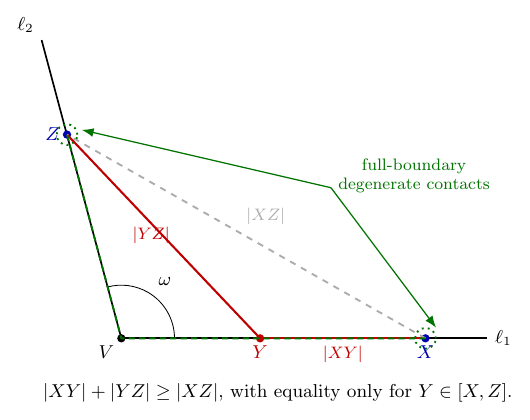}
\caption{Wedge-shortening on the boundary of a wedge: a kink through a boundary point satisfies $|XY|+|YZ|\ge |XZ|$. The proof uses this only when the closed branch model already includes the degenerate contact obtained by deleting the kink.}
\label{fig:wedge}
\end{figure}

\begin{lemma}[Subcase 1.1 tail reduction]\label{lem:subcase11-tail-reduction}
In Subcase~1.1, the additional $T_W$ contacts are represented by the chains used in Subcases~1.1a, 1.1b, the reflection of 1.1b, and the direct both-late model 1.1c. On a $T_W$ boundary ray parameterized as $Z(s)=W+s v$, the relevant closed escape row has the form $s\ge b$. A later boundary kink may therefore be deleted to the nearest vertical delimiter only in the branch where that delimiter itself satisfies the same closed escape row. If exactly one delimiter does not satisfy the row, the corresponding late contact is retained and is represented by the one-sided model 1.1b or its reflection. If both sides are late, both late contacts are retained and certified directly by model 1.1c.
\end{lemma}

\begin{proof}
See Appendix~\ref{app:lemmas}. A worked model-row illustration of the degenerate delimiter choice in the closed model 1.1a, including the allowed coincidences $P_{3.5}=P_6$ and $P_{-3.5}=P_{-6}$, follows Table~\ref{tab:wedge-audit} in Appendix~\ref{app:branch-grammar}.
\end{proof}

\begin{lemma}[Case 2 raw-order enumeration]\label{lem:case2-order-reduction}
In Case~2, every admissible right-tail order of $P_4,P_7,P_8$ is one of the six weak orders
\[
P_3P_4P_7P_8,\quad
P_3P_8P_7P_4,\quad
P_3P_4P_8P_7,\quad
P_3P_7P_4P_8,\quad
P_3P_7P_8P_4,\quad
P_3P_8P_4P_7.
\]
In Subcase~2.2, the left tail has one of the two weak orders
\[
P_{-7}P_{-6}P_{-3}
\qquad\text{or}\qquad
P_{-6}P_{-7}P_{-3}.
\]
The finite SOCP certificate treats the two representative right-tail orders in the original terminal models and certifies the remaining raw-order branches separately; no endpoint replacement is needed for Cases~2.1--2.2.
\end{lemma}

\begin{proof}
See Appendix~\ref{app:lemmas}.
\end{proof}

\begin{lemma}[Case 3 exhaustive raw-order enclosure]\label{lem:case3-order-reduction}
In Case~3, the right-tail contacts $P_4,P_5,P_7,P_8$ occur in a weak parameter order represented by one of the $24$ strict linear orders of these four labels. The reflected left tail is treated analogously. For each of the resulting $24\times24$ left-right order pairings, the corresponding closed SOCP model is included in the certified model family; the three pairings 3a--3c are representative notebook models, and the other $573$ pairings are auxiliary raw-order models.
\end{lemma}

\begin{proof}
See Appendix~\ref{app:lemmas}.
\end{proof}

\section{The triangle, canonical orientations, and escape inequalities}\label{sec:triangle}

\subsection{Definition of the triangle \texorpdfstring{$T$}{T}}

Let $T$ be the $30^{\circ}\!-\!60^{\circ}\!-\!90^{\circ}$ triangle for which a square of side $1/3$ rests on its hypotenuse. Denote the side lengths by
\[
a = \frac{3+4\sqrt{3}}{18}, \qquad b = \frac{4+\sqrt{3}}{6}, \qquad c = 2a = \frac{3+4\sqrt{3}}{9},
\]
where $a$ is the short leg opposite the $30^{\circ}$ angle, $b$ is the long leg opposite the $60^{\circ}$ angle, and $c$ is the hypotenuse. Thus
\[
\frac{b}{a}=\sqrt{3},
\qquad
\frac{c}{a}=2.
\]
The shortest side of $T$ is therefore the side of length $a$, and the area is
\[
\frac12ab=\frac{24+19\sqrt3}{216}\approx 0.263467.
\]

\subsection{Canonical orientations}

A \emph{canonical orientation} of $T$ places the $30^{\circ}$ corner at a named anchor point and aligns the two sides adjacent to that corner with prescribed support lines. We denote these orientations by symbols such as $T_L$, $T_R$, etc. When naming the supporting lines we distinguish, as in Section~\ref{subsec:monotonicity}, the undirected \emph{line direction} of a side from the \emph{oriented support label} $L_\theta$, which by our convention has $\gamma$ on the left of $e_\theta$; the two labels of the same side differ by $180^\circ$. For example, in orientation $T_L$ the $30^{\circ}$ corner is at an anchor $L$, the hypotenuse of length $c$ lies along the lower support $L_{0^{\circ}}$, and the long leg of length $b$ lies along the oriented support $L_{-150^{\circ}}$, whose line direction is $30^{\circ}$ and whose contact witness in Subcase~1.1 is $P_{-1}$. The remaining side, of length $a$, is the short side opposite the $30^{\circ}$ corner.

In this orientation the short side is the line
\[
\{X\in\R^2 : d(\overrightarrow{LX},30^{\circ}) = b\}.
\]
Indeed, if the anchor is placed at the origin, then the two endpoints of the short side are
\[
(c,0)
\qquad\text{and}\qquad
b(\cos 30^{\circ},\sin 30^{\circ}),
\]
and both satisfy $d(X,30^{\circ})=b$. Thus the angle $30^{\circ}$ gives the outward normal direction to the short side, not the direction of the short side itself. We denote this normal direction by $\theta_{\mathrm{short}}$.

Other orientations are obtained by rotations and reflections, in the congruent-copy convention used for the cover. In orientation $T_R$ the anchor is at $R$ with the hypotenuse along the lower support $L_{0^{\circ}}$ (side pointing in line direction $180^{\circ}$ from $R$) and the long leg along the oriented support $L_{150^{\circ}}$ through its witness $P_1$, so the short side is given by
\[
d(\overrightarrow{RX},150^{\circ})=b.
\]
Likewise in every canonical orientation, $\theta_{\mathrm{short}}$ denotes the outward normal direction to the short side of length $a$, and the corresponding short side is located at directional offset $b$ from the anchor. The finite family below is used as a set of forced test placements in the contradiction argument; it is not an attempt to parametrize all congruent placements of $T$.

Figure~\ref{fig:canonical} shows all twelve canonical orientations, grouped by anchor location: floor ($T_L$, $T_{L'}$, $T_R$, $T_{R'}$), walls ($T_S$, $T_T$), ceiling ($T_U$, $T_{U'}$, $T_V$, $T_{V'}$), and wedge ($T_W$, $T_{W'}$). The notebook anchor names are branch-local coordinate names rather than global orientation labels; the dictionary between them is fixed in Appendix~\ref{app:triangle-normalization} (in particular, in Case~2 the anchor $V$ represents the placement $T_{V'}$, while in Case~3 it represents $T_V$).

Table~\ref{tab:orientation-predicates} records the side directions and escape-normal directions used by these orientations. In the table, $A$ is the displayed anchor, the two side-direction columns identify the adjacent sides of lengths $b$ and $c$ meeting at the $30^\circ$ corner, and the branch-specific escaping point $X$ is the named support contact selected in the case tree. The predicate column is the affine inequality imposed when the corresponding canonical placement is forced; its direction is the length-$b$ side direction because the short side of length $a$ is reached at offset $b$ from the anchor.

\begin{table}[t]
\centering
\caption{Canonical orientations and escape predicates. The side-direction columns give undirected line directions; the last column lists the oriented support labels of the two adjacent sides in the convention of Section~\ref{subsec:support-distance}.}\label{tab:orientation-predicates}
\fontsize{6.5}{7.2}\selectfont
\setlength{\tabcolsep}{2pt}
\renewcommand{\arraystretch}{1.08}
\begin{tabular}{@{}llllll@{}}
\toprule
\textbf{Orientation} & \textbf{Anchor} & \textbf{$b$-side dir.} & \textbf{$c$-side dir.} & \textbf{Short-side escape predicate} & \textbf{Oriented supports} \\
\midrule
$T_L$ & $L$ & $30^\circ$ & $0^\circ$ & $d(\overrightarrow{LX},30^\circ)>b$ & $L_{0^\circ}$,\ $L_{-150^\circ}$ \\
$T_{L'}$ & $L$ & $0^\circ$ & $30^\circ$ & $d(\overrightarrow{LX},0^\circ)>b$ & $L_{0^\circ}$,\ $L_{-150^\circ}$ \\
$T_R$ & $R$ & $150^\circ$ & $180^\circ$ & $d(\overrightarrow{RX},150^\circ)>b$ & $L_{0^\circ}$,\ $L_{150^\circ}$ \\
$T_{R'}$ & $R$ & $180^\circ$ & $150^\circ$ & $d(\overrightarrow{RX},180^\circ)>b$ & $L_{0^\circ}$,\ $L_{150^\circ}$ \\
$T_S$ & $S$ & $-90^\circ$ & $-60^\circ$ & $d(\overrightarrow{SX},-90^\circ)>b$ & $L_{-90^\circ}$,\ $L_{120^\circ}$ \\
$T_T$ & $T$ & $-90^\circ$ & $-120^\circ$ & $d(\overrightarrow{TX},-90^\circ)>b$ & $L_{90^\circ}$,\ $L_{-120^\circ}$ \\
$T_U$ & $U$ & $-30^\circ$ & $0^\circ$ & $d(\overrightarrow{UX},-30^\circ)>b$ & $L_{180^\circ}$,\ $L_{-30^\circ}$ \\
$T_{U'}$ & $U$ & $0^\circ$ & $-30^\circ$ & $d(\overrightarrow{UX},0^\circ)>b$ & $L_{180^\circ}$,\ $L_{-30^\circ}$ \\
$T_V$ & $V$ & $-150^\circ$ & $180^\circ$ & $d(\overrightarrow{VX},-150^\circ)>b$ & $L_{180^\circ}$,\ $L_{30^\circ}$ \\
$T_{V'}$ & $V$ & $180^\circ$ & $-150^\circ$ & $d(\overrightarrow{VX},180^\circ)>b$ & $L_{180^\circ}$,\ $L_{30^\circ}$ \\
$T_W$ & $W$ & $-75^\circ$ & $-105^\circ$ & $d(\overrightarrow{WX},-75^\circ)>b$ & $L_{-105^\circ}$,\ $L_{105^\circ}$ \\
$T_{W'}$ & $W$ & $-105^\circ$ & $-75^\circ$ & $d(\overrightarrow{WX},-105^\circ)>b$ & $L_{-105^\circ}$,\ $L_{105^\circ}$ \\
\bottomrule
\end{tabular}
\end{table}

Because sign conventions are central to the proof, Table~\ref{tab:short-side-equations} makes each row of Table~\ref{tab:orientation-predicates} checkable without mental rotation or reflection: for every canonical orientation it lists the affine coordinate equation of the short side, namely the level line $d(\overrightarrow{AX},\theta_{\rm short})=b$, together with the two closed support half-planes containing $\gamma$, written with the branch-local anchor coordinates of Appendix~\ref{app:triangle-normalization}.

\begin{table}[t]
\centering
\caption{Short-side affine equations and adjacent support half-planes for the canonical orientations. Here $X=(x,y)$, the anchors are $L=(x_L,y_3)$, $R=(x_R,y_3)$, $S=(x_{-6},y_S)$, $T=(x_6,y_T)$, $U=(x_U,y_0)$, $V=(x_V,y_0)$, $W=(x_W,y_W)$, and each half-plane row is the closed support condition containing $\gamma$.}\label{tab:short-side-equations}
\scriptsize
\setlength{\tabcolsep}{2pt}
\renewcommand{\arraystretch}{1.14}
\begin{tabularx}{\textwidth}{@{}llp{4.55cm}>{\raggedright\arraybackslash}X@{}}
\toprule
\textbf{Orient.} & \textbf{$\theta_{\rm short}$} & \textbf{Short side $d(\overrightarrow{AX},\theta_{\rm short})=b$} & \textbf{Support half-planes containing $\gamma$} \\
\midrule
$T_L$ & $30^\circ$ & $\tfrac{\sqrt3}{2}(x-x_L)+\tfrac12(y-y_3)=b$ & $y\ge y_3$;\ \ $y-y_3\le\tan30^\circ\,(x-x_L)$ \\
$T_{L'}$ & $0^\circ$ & $x-x_L=b$ & same as $T_L$ \\
$T_R$ & $150^\circ$ & $-\tfrac{\sqrt3}{2}(x-x_R)+\tfrac12(y-y_3)=b$ & $y\ge y_3$;\ \ $y-y_3\le\tan(-30^\circ)\,(x-x_R)$ \\
$T_{R'}$ & $180^\circ$ & $x_R-x=b$ & same as $T_R$ \\
$T_S$ & $-90^\circ$ & $y_S-y=b$ & $x\ge x_{-6}$;\ \ $y-y_S\le\tan120^\circ\,(x-x_{-6})$ \\
$T_T$ & $-90^\circ$ & $y_T-y=b$ & $x\le x_6$;\ \ $y-y_T\le\tan60^\circ\,(x-x_6)$ \\
$T_U$ & $-30^\circ$ & $\tfrac{\sqrt3}{2}(x-x_U)-\tfrac12(y-y_0)=b$ & $y\le y_0$;\ \ $y-y_0\ge\tan(-30^\circ)\,(x-x_U)$ \\
$T_{U'}$ & $0^\circ$ & $x-x_U=b$ & same as $T_U$ \\
$T_V$ & $-150^\circ$ & $-\tfrac{\sqrt3}{2}(x-x_V)-\tfrac12(y-y_0)=b$ & $y\le y_0$;\ \ $y-y_0\ge\tan30^\circ\,(x-x_V)$ \\
$T_{V'}$ & $180^\circ$ & $x_V-x=b$ & same as $T_V$ \\
$T_W$ & $-75^\circ$ & $\cos75^\circ\,(x-x_W)-\sin75^\circ\,(y-y_W)=b$ & $y-y_W\le\tan75^\circ\,(x-x_W)$;\newline $y-y_W\le\tan(-75^\circ)\,(x-x_W)$ \\
$T_{W'}$ & $-105^\circ$ & $-\cos75^\circ\,(x-x_W)-\sin75^\circ\,(y-y_W)=b$ & same as $T_W$ \\
\bottomrule
\end{tabularx}
\end{table}

\begin{figure}[t]
\centering
\includegraphics[width=0.95\textwidth]{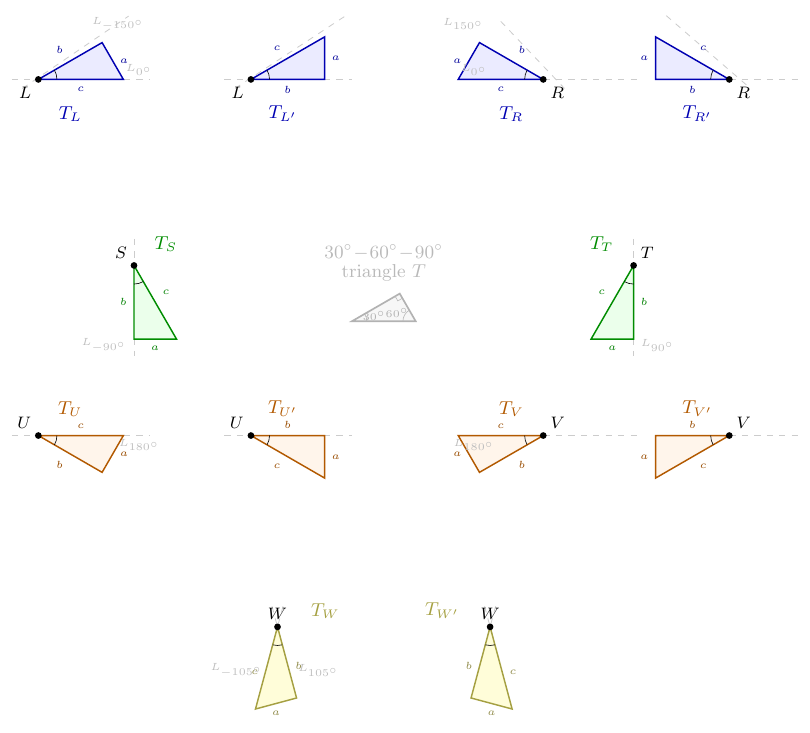}
\caption[Canonical orientations of $T$]{The twelve canonical orientations of $T$, grouped by anchor location. Row~1 (blue): floor orientations $T_L$, $T_{L'}$, $T_R$, $T_{R'}$ with anchors on $L_{0^{\circ}}$. Row~2 (green): wall orientations $T_S$, $T_T$ with anchors on $L_{\pm 90^{\circ}}$. Row~3 (orange): ceiling orientations $T_U$, $T_{U'}$, $T_V$, $T_{V'}$ with anchors on $L_{180^{\circ}}$. Row~4 (yellow): wedge orientations $T_W$, $T_{W'}$ with anchors at the intersection of steep support lines.}
\label{fig:canonical}
\end{figure}

\subsection{Escape predicate}

For a canonical orientation with anchor $A$ and outward normal direction $\theta_{\mathrm{short}}$ to the short side, the \emph{escape predicate} for a point $X$ is
\begin{equation}\label{eq:escape}
d(\overrightarrow{AX}, \theta_{\mathrm{short}}) > b.
\end{equation}
The short side is the level line $d(\overrightarrow{AX}, \theta_{\mathrm{short}})=b$, so \eqref{eq:escape} places $X$ strictly beyond it; the threshold is $b$---the directional offset from the $30^{\circ}$ corner to the short side---although the escaped side has length $a$. When the short side is vertical, the predicate takes the coordinate form $x_A + b < x_X$ or $x_X < x_A - b$.

\begin{lemma}[Canonical escape tests]\label{lem:canonical-tests}
Let $\gamma$ be a normalized obstruction not contained in any congruent copy of $T$. Consider any canonical placement whose $30^{\circ}$ corner is an anchor $A$ and whose two sides adjacent to that corner lie on supporting lines of $\gamma$. Then $\gamma$ has a point satisfying the escape predicate
\[
d(\overrightarrow{AX},\theta_{\mathrm{short}})>b.
\]
Moreover, the escaping point may be chosen on the support face of $\conv(\gamma)$ exposed by the direction $\theta_{\mathrm{short}}$; this gives the corresponding selected named witness, with coincident contacts and exposed segments interpreted as closed weak-order degeneracies.
\end{lemma}

\begin{proof}
The two adjacent sides of the canonical placement lie on supporting lines and are oriented so that $\gamma$ is contained in the closed $30^{\circ}$ angle bounded by them. The triangle is the intersection of this angle with the half-plane
\[
d(\overrightarrow{AX},\theta_{\mathrm{short}})\le b
\]
cut off by the short side. Since $\gamma$ is not contained in this congruent copy of $T$, and it is already contained in the angle, it must contain a point with $d(\overrightarrow{AX},\theta_{\mathrm{short}})>b$. The continuous linear functional $X\mapsto d(\overrightarrow{AX},\theta_{\mathrm{short}})$ attains its maximum on the compact set $\conv(\gamma)$ on the exposed support face in direction $\theta_{\mathrm{short}}$. If that face is a single named contact this gives the stated contact point; if it is a segment or several named contacts coincide, the selected witness convention records the required endpoint witnesses when the branch asks for them, and the closed numerical model records the corresponding weak inequality.
\end{proof}

\begin{lemma}[Forced canonical placements in the case tree]\label{lem:forced-placements}
Under the hypotheses of Lemma~\ref{lem:canonical-tests}, every canonical placement used in the certified branches of Section~\ref{sec:cases} satisfies the support-line hypotheses of Lemma~\ref{lem:canonical-tests}. Consequently each escape inequality listed in the case tree and in Appendix~\ref{app:summary} or Appendix~\ref{app:raw-order-cert} is forced by noncoverage.
\end{lemma}

\begin{proof}
The case tree is built only from support-role data selected as witnesses in Section~\ref{subsec:lambda}. In each branch, a canonical placement is introduced precisely when the branch data provide the two support faces carrying the sides adjacent to the $30^\circ$ corner of the corresponding orientation in Figure~\ref{fig:canonical}. The floor placements use the lower horizontal support together with the adjacent $30^\circ$ or $150^\circ$ support; the wall placements use the appropriate vertical support together with the adjacent steep support; the ceiling placements use the upper horizontal support together with the adjacent reflected steep support; and the wedge placements use the two steep support lines meeting at the wedge anchor. In every case the two side directions differ by $30^\circ$ and their oriented closed half-planes contain $\gamma$, because they are support half-planes of $\conv(\gamma)$.

The intersection of the two supporting lines is therefore the required anchor $A$, and $\gamma$ lies in the closed $30^\circ$ angle bounded by the adjacent sides of that canonical copy of $T$. Lemma~\ref{lem:canonical-tests} applies. The named predicate in the notebook is just the affine form of the resulting incidence or escape condition: \texttt{onlineL}, \texttt{onlineR}, \texttt{onlineU}, \texttt{onlineV}, \texttt{onlineWl}, and \texttt{onlineWr} record side incidences; \texttt{farL}, \texttt{farR}, \texttt{farU}, and \texttt{farV} record escape inequalities; \texttt{condS}, \texttt{condT}, \texttt{condSfarL}, and \texttt{condTfarR} bundle the indicated incidence and escape rows; and the Case~2 predicate \texttt{condV} records \texttt{onlineV} together with the guard $x_V-x_{-6}\ge b$ used in Subcases~2.2a--b. Reflected placements are identical after applying the fixed reflection convention.
\[
(\hbox{placement},\hbox{adjacent support sides},A,X,\theta_{\rm short})
\]
is the forcing certificate for an individual escape row: once the branch supplies the two adjacent support sides through anchor $A$, noncoverage of that canonical copy forces the named escaping witness $X$ to satisfy $d(\overrightarrow{AX},\theta_{\rm short})\ge b$ in the closed model. The compound predicates are shorthand for this certificate plus the stated side-incidence row, as displayed in Tables~\ref{tab:forced-placement-families} and~\ref{tab:forced-guard-rows}.
\end{proof}

\begin{table}[t]
\centering
\caption{Geometric meaning of the forced placement predicate families. The full affine rows are listed in Table~\ref{tab:predicate-dictionary}; this table records the support-side source of each family.}\label{tab:forced-placement-families}
\scriptsize
\setlength{\tabcolsep}{3pt}
\renewcommand{\arraystretch}{1.08}
\begin{tabularx}{\textwidth}{@{}p{2.05cm}p{1.2cm}p{3.2cm}p{2.4cm}X@{}}
\toprule
\textbf{Predicate family} & \textbf{Anchor} & \textbf{Adjacent support sides} & \textbf{Incidence row} & \textbf{Forced escape or guard row} \\
\midrule
\texttt{onlineL}, \texttt{condL}, \texttt{farL} & $L$ & lower support $L_0$ and the $30^\circ$ side through $L$ & $P_i$ lies on the $30^\circ$ side through $L$ & $d(\overrightarrow{LP_j},30^\circ)\ge b$; \texttt{condL} uses the closed guard $x_L+b\le x_6$. \\
\texttt{onlineR}, \texttt{condR}, \texttt{farR} & $R$ & lower support $L_0$ and the $150^\circ$ side through $R$ & $P_i$ lies on the $150^\circ$ side through $R$ & $d(\overrightarrow{RP_j},150^\circ)\ge b$; \texttt{condR} uses the closed guard $x_{-6}\le x_R-b$. \\
\texttt{onlineS}, \texttt{condS}, \texttt{condSfarL} & $S$ & vertical support at $P_{-6}$ and the $120^\circ$ side through $S$ & $P_i$ lies on the $120^\circ$ side through $S$ & \texttt{condS} uses $y_3\le y_S-b$; \texttt{condSfarL} also includes $d(\overrightarrow{LP_i},30^\circ)\ge b$. \\
\texttt{onlineT}, \texttt{condT}, \texttt{condTfarR} & $T$ & vertical support at $P_6$ and the $60^\circ$ side through $T$ & $P_i$ lies on the $60^\circ$ side through $T$ & \texttt{condT} uses $y_3\le y_T-b$; \texttt{condTfarR} also includes $d(\overrightarrow{RP_i},150^\circ)\ge b$. \\
\texttt{onlineU}, \texttt{farU} & $U$ & upper support $L_{180^\circ}$ and the reflected $-30^\circ$ side through $U$ & $P_{-4}$ lies on the side through $U$ & $d(\overrightarrow{UP_5},-30^\circ)\ge b$. \\
\texttt{onlineV}, \texttt{farV}, \texttt{condV} & $V$ & upper support $L_{180^\circ}$ and the $30^\circ$ side through $V$ & $P_4$ lies on the side through $V$ & \texttt{farV} is $d(\overrightarrow{VP_{-5}},-150^\circ)\ge b$; \texttt{condV} uses the Case~2 guard $x_V-x_{-6}\ge b$. \\
\texttt{onlineWl}, \texttt{onlineWr}, \texttt{dot\_gt\_b} & $W$ & the two steep wedge sides through $W$, on the oriented supports $L_{-105^\circ}$ and $L_{105^\circ}$ (line directions $75^\circ$ and $-75^\circ$) & selected $T_W$ side contacts lie on the corresponding wedge sides & explicit rows $d(\overrightarrow{WP_i},-105^\circ)\ge b$ or $d(\overrightarrow{WP_i},-75^\circ)\ge b$, as recorded in the JSON model. \\
\bottomrule
\end{tabularx}
\end{table}

The supplementary forced-placement CSV is generated from the certified model JSON. Each row records the model id, placement family, anchor, incidence rows, escape or guard row, predicate name, and point labels; Table~\ref{tab:forced-placement-families} is the family-level summary of those generated audit rows. Lemma~\ref{lem:row-audit} in Appendix~\ref{app:audit} states the corresponding row-level soundness obligation---every certified affine escape or guard row has exactly one audit record, and the record reconstructs the same coefficient vector and constant term as the model row---and the one-command verifier of Appendix~\ref{app:repro} checks this obligation as part of the proof verification.

\begin{table}[p]
\centering
\caption{Explicit forcing of compound guard rows. Each row is a closed escape inequality for the listed canonical placement, possibly bundled with an incidence equality.}\label{tab:forced-guard-rows}
\scriptsize
\setlength{\tabcolsep}{3pt}
\renewcommand{\arraystretch}{1.08}
\begin{tabularx}{\textwidth}{@{}>{\raggedright\arraybackslash}p{2.2cm}>{\raggedright\arraybackslash}p{2.3cm}>{\raggedright\arraybackslash}p{2.55cm}>{\raggedright\arraybackslash}p{3.25cm}>{\raggedright\arraybackslash}X@{}}
\toprule
\textbf{Predicate} & \textbf{Placement certificate} & \textbf{Escaping witness} & \textbf{Closed row} & \textbf{Where used} \\
\midrule
\texttt{condL} & $T_L$ at $L$; sides $L_0$ and $L_{-150^\circ}$ & vertical delimiter $P_6$ & $d(\overrightarrow{LP_6},0^\circ)=x_6-x_L\ge b$ & Subcases 1.1a--c. \\
\texttt{condR} & $T_R$ at $R$; sides $L_0$ and $L_{150^\circ}$ & vertical delimiter $P_{-6}$ & $d(\overrightarrow{RP_{-6}},180^\circ)=x_R-x_{-6}\ge b$ & Subcases 1.1a--c. \\
\texttt{condS} & $T_S$ at $S$; vertical side and $L_{120^\circ}$ side & lower delimiter $P_3$ & $d(\overrightarrow{SP_3},-90^\circ)=y_S-y_3\ge b$ & Included in \texttt{condSfarL}. \\
\texttt{condT} & $T_T$ at $T$; vertical side and $L_{-120^\circ}$ side & lower delimiter $P_3$ & $d(\overrightarrow{TP_3},-90^\circ)=y_T-y_3\ge b$ & Included in \texttt{condTfarR}. \\
\texttt{condSfarL} & $T_S$ at $S$ and $T_L$ at $L$ & $P_3$ for the wall guard; listed point for floor escape & $y_S-y_3\ge b$ and $d(\overrightarrow{LP_i},30^\circ)\ge b$ & Subcases 1.1a--c. \\
\texttt{condTfarR} & $T_T$ at $T$ and $T_R$ at $R$ & $P_3$ for the wall guard; listed point for floor escape & $y_T-y_3\ge b$ and $d(\overrightarrow{RP_i},150^\circ)\ge b$ & Subcases 1.1a--c. \\
\texttt{condV} & Case~2 placement $T_{V'}$ at $V$; sides $L_{180^\circ}$ and the $-150^\circ$ side through $V$ & reflected vertical delimiter $P_{-6}$ & $d(\overrightarrow{VP_{-6}},180^\circ)=x_V-x_{-6}\ge b$ plus \texttt{onlineV} & Subcases 2.2a--b and the ten raw-order 2.2 variants. \\
\bottomrule
\end{tabularx}
\end{table}

At row level, the supplementary file \path{r/forced_placement_audit.csv} contains one audit record for each forced predicate row in the $599$ certified models. The audit file has $77$ representative rows and $4658$ raw-order rows; the field \texttt{constraint\_index} identifies the corresponding JSON row, while \texttt{escape\_or\_guard\_row} records the exact affine inequality. As a row-level example, the audit record for representative model 1.1a, constraint~7, is the tuple
\[
(\text{representative},\ \texttt{1.1a},\ 7,\ \texttt{dot\_gt\_b},\ W,\ P_{3.5},\ \theta_{\rm short}=-75^\circ),
\]
and reconstructing the escape row $d(\overrightarrow{WP_{3.5}},-75^\circ)-b\ge0$ from this tuple gives the affine inequality
\[
\cos75^\circ\,(x_{3.5}-x_W)-\sin75^\circ\,(y_{3.5}-y_W)-b\ \ge\ 0,
\]
whose coefficient vector ($\cos75^\circ$ on $x_{3.5}$, $-\sin75^\circ$ on $y_{3.5}$, $-\cos75^\circ$ on $x_W$, $\sin75^\circ$ on $y_W$) and constant term $-b$ coincide with the corresponding JSON model row. The only generic \texttt{below\_line} row occurs in representative model 1.1b, constraint~7, where $P_{-2}$ is placed weakly below the line of direction $75^\circ$ through $W$. It is the closed version of the strict side-containment inequality printed in Table~\ref{tab:cert-bounds}; replacing $<$ by $\le$ enlarges the feasible set, so the certified lower bound remains valid for the original strict branch.

Thus, under the noncovering assumption, every canonical placement used below supplies its stated escape inequality by Lemmas~\ref{lem:canonical-tests} and~\ref{lem:forced-placements}, and Lemma~\ref{lem:row-audit} ties each certified model row to its forcing placement certificate.
\begin{figure}[htbp]
    \centering
    \includegraphics[width=0.95\textwidth]{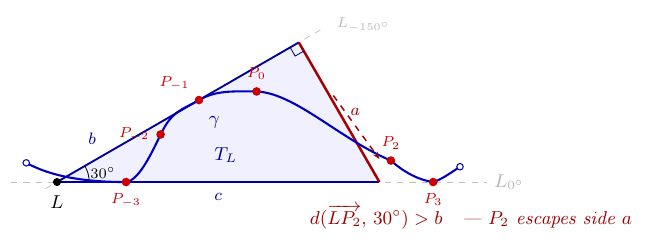}
    \caption[Escape predicate in action]{The escape predicate in action. The triangle $T_L$ is placed with its $30^{\circ}$ corner at anchor $L$, with side $c$ along the lower support $L_{0^{\circ}}$ and side $b$ along the oriented support $L_{-150^{\circ}}$ (line direction $30^{\circ}$). The short side has length $a$ and is the line $d(\protect\overrightarrow{LX},30^{\circ})=b$. The arc $\gamma$ enters the $30^{\circ}$ angle, touching $L_{0^{\circ}}$ at $P_{-3}$ and the $30^{\circ}$-direction side at its witness $P_{-1}$, with $P_{-2}$ strictly inside. The contact point $P_2$ satisfies $d(\protect\overrightarrow{LP_2},30^{\circ})>b$, so it lies beyond the short side of $T_L$.}
    \label{fig:orientations_setup}
\end{figure}
\section{Case analysis and numerical optimization data}\label{sec:cases}

Arguing by contradiction, assume that $T$ fails to cover a unit arc. By Lemma~\ref{lem:reduced-obstruction}, we may work with a weak support-reduced normalized obstruction $\gamma$ with $\ell(\gamma)\le 1$. Lemmas~\ref{lem:canonical-tests} and~\ref{lem:forced-placements} then turn each canonical test placement used below into the corresponding forced escape inequality.

\begin{lemma}[Ordered contact chains are lower bounds]\label{lem:ordered-chain}
Let $\gamma:[0,\ell]\to\mathbb R^2$ be a rectifiable curve, and let
\[
0\le t_1\le t_2\le\cdots\le t_m\le\ell.
\]
If $P_j=\gamma(t_j)$, then
\[
\sum_{j=1}^{m-1}|P_{j+1}-P_j|\le \ell(\gamma).
\]
Coincident parameters contribute zero-length segments.
\end{lemma}

\begin{proof}
This is the standard total-variation inequality for a parameter-ordered polygonal subsequence of a rectifiable curve.
\end{proof}

Let $\mathcal I$ be the certified index set consisting of the twelve representative terminal models and the $587$ additional raw-order models for the nonrepresentative tail orders in Cases~2 and~3. For each $i\in\mathcal I$, let $\mathcal F_i^{\rm open}$ denote the feasible set with the strict escape inequalities used in the original notebook when such a notebook model exists, and let $\mathcal F_i^{\rm cl}$ denote the corresponding closed feasible set obtained by replacing strict escape inequalities by weak inequalities while keeping the same contact and ordering constraints. Let $\Pi_i$ denote the polygonal chain in the recorded contact order. We write
\[
L_i^{\rm cert}=\inf_{\mathcal F_i^{\rm cl}}\ell(\Pi_i).
\]
Since $\mathcal F_i^{\rm open}\subseteq \mathcal F_i^{\rm cl}$ for the original notebook models, a certified lower bound $L_i^{\rm cert}>1$ is valid for every strict realization. For the raw-order models, the closed SOCP is built directly from the same affine contact, ordering, and escape rows and from the raw polygonal chain in parameter order. For any obstruction realizing branch $i$, Lemma~\ref{lem:middle-skeleton}, Lemma~\ref{lem:ordered-chain}, and the raw-order enclosures in Lemmas~\ref{lem:case2-order-reduction} and~\ref{lem:case3-order-reduction} give $\ell(\Pi_i)\le \ell(\gamma)$.
We record explicitly which ordering information the certified models consume: the Case~1 models impose the inner-role middle chain of Proposition~\ref{prop:support-input}(2), which is available there because all four $L_{\pm150^\circ},L_{\pm120^\circ}$ roles are inner; the Case~2 and Case~3 models impose only the branch tail orders enumerated exhaustively in Lemmas~\ref{lem:case2-order-reduction} and~\ref{lem:case3-order-reduction}, never a parameter relation between a tail witness and $P_{\pm3}$ beyond the recorded chain order. No certified model encodes the full seven-term chain outside the all-inner branch.
Appendix~\ref{app:opt} records, for each representative terminal subcase, the active contacts, order constraints, escape inequalities, the chain $\Pi_i$, the numerical output, and the independent interval certificate for $\mathcal F_i^{\rm cl}$, together with the raw-order SOCP certificates for the nonrepresentative Case~2 orders and the exhaustive Case~3 pairings; Appendix~\ref{app:branch-grammar} records the branch grammar, chain-order ledger, and model-id routing in Table~\ref{tab:branch-model-ledger}. This section presents the branching geometry; the rounded reported values for all representative models, including the direct both-late model 1.1c, are collected in Table~\ref{tab:bounds}.

\subsection{Case tree structure}

By Lemmas~\ref{lem:order} and~\ref{lem:contact-alternatives}, the support lines $L_{\pm 150^{\circ}}$ and $L_{\pm 120^{\circ}}$ have active contact roles either on the middle subarc $\gamma_{P_{-3}P_3}$ or on the corresponding tail. On the positive side, $L_{150^{\circ}}$ has role $P_1$ or $P_8$, and $L_{120^{\circ}}$ has role $P_2$ or $P_7$. The pair $(P_8,P_2)$ is impossible: an inner $120^{\circ}$ contact forces the corresponding $150^{\circ}$ contact to be inner. Thus the one-sided positive roles are exactly
\[
\begin{array}{c|c|c}
\text{type} & L_{150^{\circ}} & L_{120^{\circ}}\\
\hline
A & P_1 & P_2\\
B & P_1 & P_7\\
C & P_8 & P_7 .
\end{array}
\]
The negative side has the reflected alternatives $(P_{-1},P_{-2})$, $(P_{-1},P_{-7})$, and $(P_{-8},P_{-7})$. Pairing the three positive and three negative alternatives gives nine raw role branches.

\begin{proposition}[Finite case-tree exhaustion]\label{prop:case-tree}
After applying Lemmas~\ref{lem:order}, \ref{lem:contact-alternatives}, \ref{lem:subcase11-tail-reduction}, \ref{lem:case2-order-reduction}, and~\ref{lem:case3-order-reduction}, every weak support-reduced normalized obstruction belongs to one of the certified closed models: the twelve representative terminal models in Table~\ref{tab:case-tree-exhaustion} or the $587$ additional raw-order models in Table~\ref{tab:raw-order-cert}.
\end{proposition}

\begin{proof}
On the positive side the allowed role types are
\[
A=(P_1,P_2),\qquad B=(P_1,P_7),\qquad C=(P_8,P_7),
\]
where the ordered pair records the active $(150^\circ,120^\circ)$ roles. On the negative side the reflected types are
\[
A'=(P_{-1},P_{-2}),\qquad B'=(P_{-1},P_{-7}),\qquad C'=(P_{-8},P_{-7}).
\]
Thus there are $3\cdot 3=9$ raw left-right combinations. The complete grammar, including the excluded mixed alternatives and the raw-order certificate routing, is recorded in Appendix~\ref{app:branch-grammar}.

If both $150^\circ$ contacts are inner, the raw combinations are $A'A$, $A'B$, $B'A$, and $B'B$. The two mixed combinations are reflected copies of each other, so they give one numerical model. The $A'A$ branch is Subcase~1.1 and splits into the three $T_W$ alternatives 1.1a, 1.1b, and the direct both-late model 1.1c. The mixed branch is Subcase~1.2, and $B'B$ is Subcase~1.3.

If exactly one $150^\circ$ contact is outer, reflection lets us take the representative with left type $A'$ or $B'$ and right type $C$. These give Subcases~2.1 and~2.2. By Lemma~\ref{lem:case2-order-reduction}, the right tail has one of the six weak orders
\[
P_3P_4P_7P_8,\quad
P_3P_8P_7P_4,\quad
P_3P_4P_8P_7,\quad
P_3P_7P_4P_8,\quad
P_3P_7P_8P_4,\quad
P_3P_8P_4P_7.
\]
In Subcase~2.1 this gives six certified right-tail models. In Subcase~2.2 the reflected left tail has the two weak orders
\[
P_{-7}P_{-6}P_{-3}
\qquad\text{and}\qquad
P_{-6}P_{-7}P_{-3},
\]
so Subcase~2.2 gives twelve certified left-right order models. The four representative models 2.1a, 2.1b, 2.2a, and~2.2b are the two displayed right-tail orders with the representative left order; the remaining Case~2 order models are certified directly in Table~\ref{tab:raw-order-cert}.

If both $150^\circ$ contacts are outer, the raw type is $C'C$. By Lemma~\ref{lem:case3-order-reduction}, the right tail has a weak order represented by one of the $24$ strict orders of $P_4,P_5,P_7,P_8$, and the left tail is represented by one of the $24$ reflected strict orders. Pairing them gives $24\times24=576$ certified Case~3 order models. The three representative pairings 3a, 3b, and~3c are listed in Table~\ref{tab:case-tree-exhaustion}; the remaining $573$ raw-order pairings are certified directly in Table~\ref{tab:raw-order-cert}. Together with the five Case~1 models and eighteen Case~2 models, this exhausts all branch possibilities recorded by the support-contact grammar.
\end{proof}

\begin{table}[t]
\centering
\caption{Representative terminal models in the case tree. The nonrepresentative Case~2 and Case~3 raw orders are certified separately in Table~\ref{tab:raw-order-cert}.}\label{tab:case-tree-exhaustion}
\scriptsize
\setlength{\tabcolsep}{4pt}
\renewcommand{\arraystretch}{1.12}
\begin{tabular}{@{}lll@{}}
\toprule
\textbf{Main branch} & \textbf{Raw role data} & \textbf{Terminal subcases} \\
\midrule
Case 1 & $A'A$, $A'B\sim B'A$, $B'B$ & 1.1a, 1.1b, 1.1c, 1.2, 1.3 \\
Case 2 & $A'C$, $B'C$ up to reflection & 2.1a, 2.1b, 2.2a, 2.2b \\
Case 3 & $C'C$ with all tail orders certified & 3a, 3b, 3c \\
\bottomrule
\end{tabular}
\end{table}
\begin{figure}[htbp]
    \centering
    \includegraphics[width=0.95\textwidth]{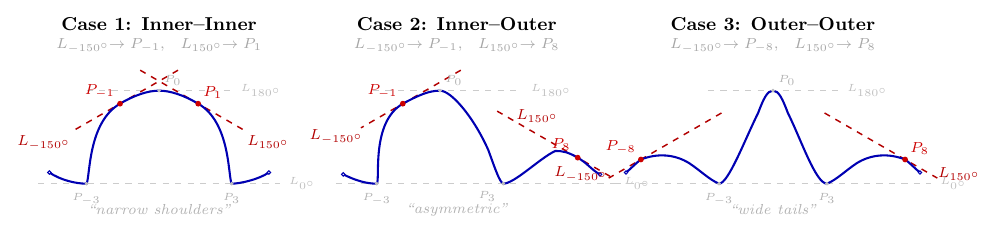}
    \caption[Three main case geometries]{The three main cases, determined by where $L_{\pm 150^{\circ}}$ touches the arc~$\gamma$. Each support line $L_{\pm 150^{\circ}}$ touches~$\gamma$ at either an \emph{inner} contact ($P_{\pm 1}$, on the middle segment~$\gamma_{P_{-3}P_3}$) or an \emph{outer} contact ($P_{\pm 8}$, on a tail). Case~1 (both inner) produces ``narrow shoulders''; Case~2 (one inner, one outer) is asymmetric; Case~3 (both outer) produces ``wide tails.'' The dashed lines show the floor $L_{0^{\circ}}$, ceiling $L_{180^{\circ}}$, and the two support lines $L_{\pm 150^{\circ}}$ through the respective contact points.}
    \label{fig:cases_geometry}
\end{figure}
Figure~\ref{fig:casetree} shows the complete case tree.

\begin{figure}[t]
\centering
\includegraphics[width=0.95\textwidth]{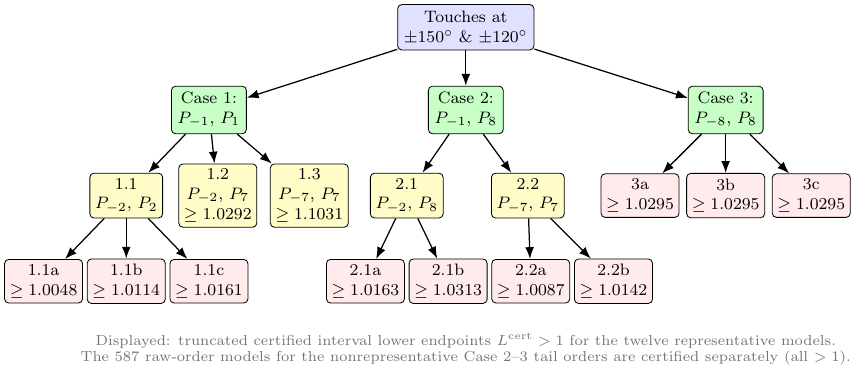}
\caption[Case tree for the finite analysis]{Case tree for the finite analysis: three main branches determined by touches at $\pm 150^{\circ}$ and $\pm 120^{\circ}$, with subcases branching on secondary touch alternatives, including the direct both-late model 1.1c. The displayed values are the truncated certified interval lower endpoints of Table~\ref{tab:interval-cert}; the $587$ nonrepresentative raw-order models of Cases~2--3 are certified separately (Table~\ref{tab:raw-order-cert}).}
\label{fig:casetree}
\end{figure}

\subsection{Case 1: Touches at \texorpdfstring{$P_{-1}$ and $P_1$}{P-1 and P1}}

In this case $L_{-150^{\circ}}$ touches at $P_{-1}$ and $L_{150^{\circ}}$ at $P_1$.

\medskip\noindent\textbf{Subcase 1.1: Touches at $P_{-2}$ and $P_2$.}
Place $T$ in orientations $T_L$ and $T_R$ with escape constraints
\[
d(\overrightarrow{LP_2}, 30^{\circ}) > b, \quad d(\overrightarrow{RP_{-2}}, 150^{\circ}) > b.
\]
Additionally, orientation $T_W$ with corner on $L_{0^{\circ}}$ provides further constraints. Lemma~\ref{lem:subcase11-tail-reduction} reduces the remaining tail-contact alternatives to the three certified branches below. Appendix~\ref{app:summary} lists the full feasible sets and the exact lower-bound chains for them.

\begin{itemize}
\item \textbf{Subcase 1.1a:} Both sides of the $30^{\circ}$ corner of $T_W$ touch $\gamma_{P_{-3}P_3}$. After applying Lemma~\ref{lem:subcase11-tail-reduction}, the notebook lower-bound chain is
\[
\begin{gathered}
P_{-6}\to P_{-3.5}\to P_{-3}\to P_{-2.5}\to P_{-2}\to P_{-1}\\
\to P_1\to P_2\to P_{2.5}\to P_3\to P_{3.5}\to P_6.
\end{gathered}
\]
\begin{figure}[t]
    \centering
    \begin{subfigure}[b]{0.48\textwidth}
        \centering
        \includegraphics[width=\textwidth]{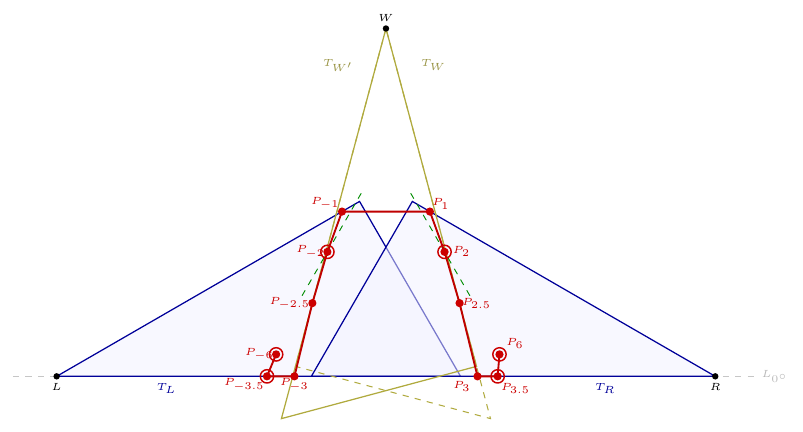}
        \caption[Subcase 1.1a]{Subcase 1.1a: both $T_W$ side contacts ($P_{\pm2.5}$) on the middle intervals; escape witnesses $P_{\pm3.5}$ at the delimiters.}
        \label{fig:subcase_1_1a}
    \end{subfigure}
    \hfill
    \begin{subfigure}[b]{0.48\textwidth}
        \centering
        \includegraphics[width=\textwidth]{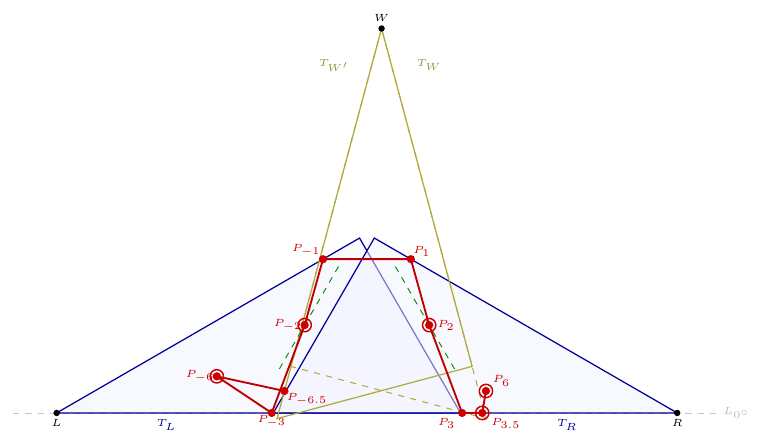}
        \caption[Subcase 1.1b]{Subcase 1.1b: the left contact $P_{-6.5}$ is late (one-sided model); the right escape row is kept at $P_{3.5}$.}
        \label{fig:subcase_1_1b}
    \end{subfigure}

    \smallskip
    \begin{subfigure}[b]{0.48\textwidth}
        \centering
        \includegraphics[width=\textwidth]{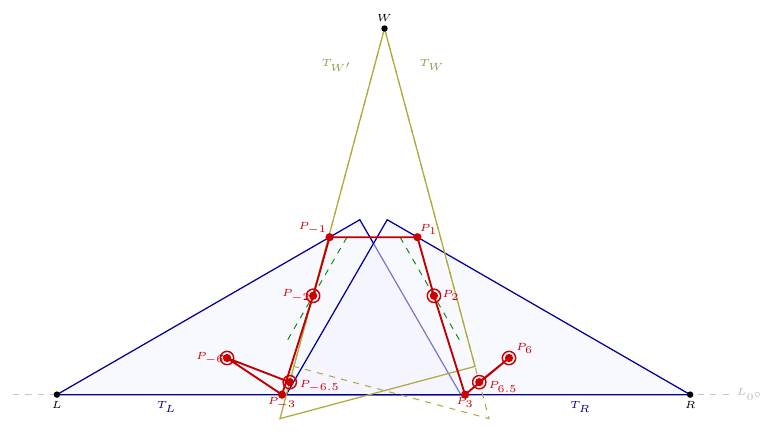}
        \caption[Subcase 1.1c]{Subcase 1.1c: both contacts late ($P_{-6.5}$ and $P_{6.5}$ retained); both late escape rows enforced directly.}
        \label{fig:subcase_1_1c}
    \end{subfigure}

    \caption[Feasible configurations for Subcase 1.1]{Feasible configurations of the three certified closed models for Subcase~1.1 (model ids \texttt{1.1a}--\texttt{1.1c}). Blue outlines denote the floor placements $T_L$ and $T_R$, forced by the lower support together with the inner $\mp150^\circ$ witnesses $P_{\mp1}$; the solid and dashed gold outlines denote the wedge placements $T_W$ and $T_{W'}$, sharing the corner $W$ on the two steep supports; green dashed segments indicate the $T_S$ and $T_T$ hypotenuse incidences through $P_{\pm2}$. The red polyline is the certified lower-bound chain in its recorded parameter order, so its length is at most $\ell(\gamma)$ by Lemma~\ref{lem:ordered-chain}. Every incidence row of the model holds exactly (points on the drawn sides), and every escape or guard row holds strictly at the ringed points, each of which lies beyond the short side of its forcing placement. Since the certified minimum of each closed model exceeds $1$ (Table~\ref{tab:interval-cert}), no such configuration can come from an arc of length one---the contradiction driving Theorem~\ref{thm:main}. The certified optimum is a degenerate limit of such configurations in which several contacts coincide; the exact chains and constraints are listed in Appendix~\ref{app:summary}.}
    \label{fig:subcase_1_1}
\end{figure}
\item \textbf{Subcase 1.1b:} Only one side of the $30^{\circ}$ corner touches $\gamma_{P_{-3}P_3}$. The corresponding notebook lower-bound chain is
\[
\begin{gathered}
P_{-6.5}\to P_{-6}\to P_{-3}\to P_{-2}\to P_{-1}\\
\to P_1\to P_2\to P_3\to P_{3.5}\to P_6.
\end{gathered}
\]
\item \textbf{Subcase 1.1c:} Both $T_W$ side contacts are late, so neither nearest vertical delimiter is substituted for the escaping contact. The direct both-late SOCP chain is
\[
\begin{gathered}
P_{-6.5}\to P_{-6}\to P_{-3}\to P_{-2}\to P_{-1}\\
\to P_1\to P_2\to P_3\to P_6\to P_{6.5},
\end{gathered}
\]
with reconstructed SOCP value $1.016\ldots$ and independent interval lower endpoint $1.0161871592\ldots$; see Table~\ref{tab:interval-cert}. This model has no notebook counterpart.
\end{itemize}

\medskip\noindent\textbf{Subcase 1.2: Touches at $P_{-2}$ and $P_7$.}
The escape constraints are $d(\overrightarrow{LP_7}, 30^{\circ}) > b$ and $d(\overrightarrow{RP_{-2}}, 150^{\circ}) > b$.
\begin{figure}[t]
    \centering
    \includegraphics[width=0.62\textwidth]{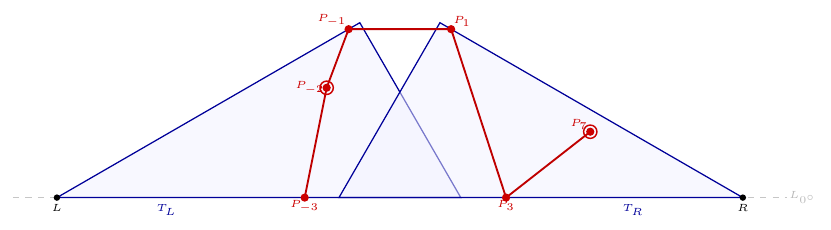}
    \caption[Feasible configuration for Subcase 1.2]{A feasible configuration of the certified closed model for Subcase~1.2 (model id \texttt{1.2}). The blue outlines show the active floor placements $T_L$ and $T_R$, forced by the lower support together with the incidences \texttt{onlineL} at $P_{-1}$ and \texttt{onlineR} at $P_1$ (exact). The ringed points carry the strict escape rows: $P_{-2}$ lies beyond the short side of $T_R$ (\texttt{farR}) and the outer $120^\circ$ witness $P_7$ beyond the short side of $T_L$ (\texttt{farL}). The red chain $P_{-3}P_{-2}P_{-1}P_1P_3P_7$ is the certified lower-bound chain of Appendix~\ref{app:summary}: it is parameter-ordered, so its length is at most $\ell(\gamma)$, while the certified model minimum exceeds $1$.}
    \label{fig:subcase_1_2}
\end{figure}
\medskip\noindent\textbf{Subcase 1.3: Touches at $P_{-7}$ and $P_7$.}
The escape constraints are symmetric on the two sides.
\begin{figure}[t]
    \centering
    \includegraphics[width=0.62\textwidth]{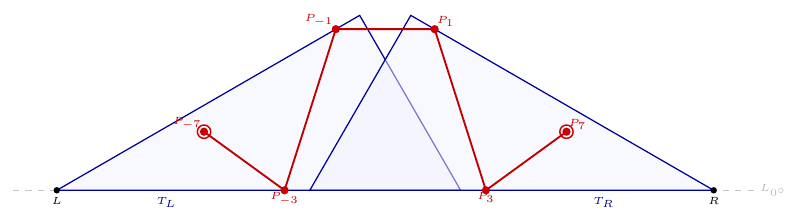}
    \caption[Feasible configuration for Subcase 1.3]{A feasible configuration of the certified closed model for Subcase~1.3 (model id \texttt{1.3}). Both $\pm150^\circ$ roles are inner ($P_{\mp1}$, incidences \texttt{onlineL}, \texttt{onlineR} exact), while both $120^\circ$ roles are outer: the ringed tail contacts $P_{\mp7}$ satisfy the strict escape rows \texttt{farR} and \texttt{farL}, lying beyond the short sides of $T_R$ and $T_L$ respectively. The red chain $P_{-7}P_{-3}P_{-1}P_1P_3P_7$ is the certified parameter-ordered lower-bound chain of Appendix~\ref{app:summary}, of length at most $\ell(\gamma)$ but certified greater than $1$.}
    \label{fig:subcase_1_3}
\end{figure}
\subsection{Case 2: Touches at \texorpdfstring{$P_{-1}$ and $P_8$}{P-1 and P8}}

In this case $L_{-150^{\circ}}$ touches at $P_{-1}$ and $L_{150^{\circ}}$ at $P_8$.
Since the right $150^{\circ}$ contact is outer, the one-sided branch rule forces the right $120^{\circ}$ contact to be $P_7$. This is why the right escape inequalities in Case~2 are imposed at $P_7$, although the main $150^{\circ}$ branch is labeled by $P_8$.

\medskip\noindent\textbf{Subcase 2.1: Touches at $P_{-2}$ and $P_8$.}
In orientation $T_{V'}$ the upper side coincides with $L_{180^{\circ}}$ touching $P_0$. On the right end, Lemma~\ref{lem:case2-order-reduction} gives six admissible weak orders of the points $P_4,P_7,P_8$. The two representative orders used in the original notebook models are
\[
P_3P_4P_7P_8
\qquad\text{and}\qquad
P_3P_8P_7P_4.
\]
The four nonrepresentative right-tail orders are certified directly as raw-order SOCP models in Appendix~\ref{app:raw-order-cert}. The two representative branches are displayed below.
\begin{figure}[t]
    \centering
    \begin{subfigure}[b]{0.47\textwidth}
        \centering
        \includegraphics[width=\textwidth]{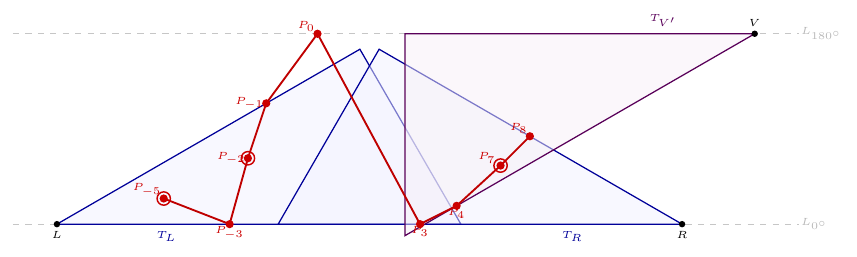}
        \caption[Subcase 2.1a]{Subcase 2.1a. Schematic right-end order $P_3P_4P_7P_8$.}
        \label{fig:subcase_2_1a}
    \end{subfigure}
    \hfill
    \begin{subfigure}[b]{0.47\textwidth}
        \centering
        \includegraphics[width=\textwidth]{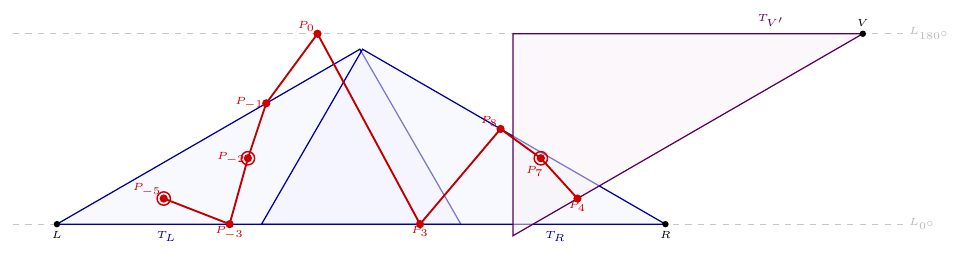}
        \caption[Subcase 2.1b]{Subcase 2.1b. Schematic right-end order $P_3P_8P_7P_4$.}
        \label{fig:subcase_2_1b}
    \end{subfigure}
    
    \caption[Feasible configurations for Subcase 2.1]{Feasible configurations of the certified closed models for Subcase~2.1 (model ids \texttt{2.1a}, \texttt{2.1b}). Blue outlines denote the floor placements $T_L$ and $T_R$ (incidences \texttt{onlineL} at $P_{-1}$ and \texttt{onlineR} at the outer $150^\circ$ witness $P_8$, exact), the purple outline denotes the ceiling placement $T_{V'}$ (anchor $V$, incidence \texttt{onlineV} placing $P_4$ on its slanted side), and dashed gray lines mark the horizontal supports. The ringed points carry the strict escape rows: $P_{-2}$ beyond the short side of $T_R$ (\texttt{farR}), the forced $120^\circ$ witness $P_7$ beyond the short side of $T_L$ (\texttt{farL}), and $P_{-5}$ beyond the short side of $T_{V'}$ (\texttt{farV}). The red polylines are the certified parameter-ordered chains in the two representative right-tail orders, of length at most $\ell(\gamma)$ but certified greater than $1$; the four nonrepresentative orders are certified separately (Table~\ref{tab:raw-order-cert}).}
    \label{fig:subcase_2_1}
\end{figure}
\begin{itemize}
\item \textbf{Subcase 2.1a:} right order $P_3P_4P_7P_8$.
\item \textbf{Subcase 2.1b:} right order $P_3P_8P_7P_4$.
\end{itemize}

\medskip\noindent\textbf{Subcase 2.2: Touches at $P_{-7}$ and $P_7$.}
The representative left order is $P_{-7}P_{-6}P_{-3}$; the reflected raw left order $P_{-6}P_{-7}P_{-3}$ is certified directly in Appendix~\ref{app:raw-order-cert}. On the right end, the two representative orders are again
\[
P_3P_4P_7P_8
\qquad\text{and}\qquad
P_3P_8P_7P_4.
\]
The remaining ten Case~2.2 left-right raw-order combinations are included in the raw-order SOCP certificate.

\begin{itemize}
\item \textbf{Subcase 2.2a:} right order $P_3P_4P_7P_8$.
\item \textbf{Subcase 2.2b:} right order $P_3P_8P_7P_4$.
\end{itemize}
\begin{figure}[t]
    \centering
    \begin{subfigure}[b]{0.47\textwidth}
        \centering
        \includegraphics[width=\textwidth]{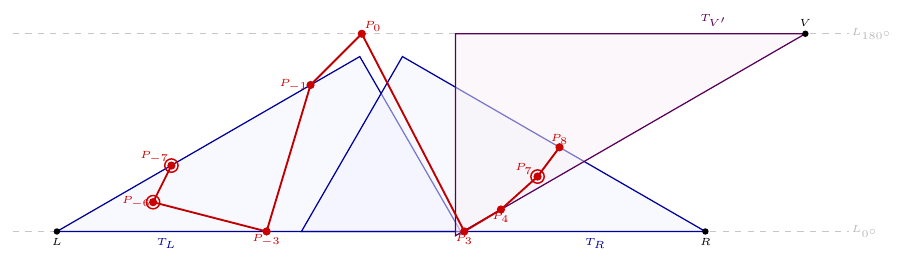}
        \caption[Subcase 2.2a]{Subcase 2.2a. Schematic left order $P_{-7}P_{-6}P_{-3}$ and right order $P_3P_4P_7P_8$.}
        \label{fig:subcase_2_2a}
    \end{subfigure}
    \hfill
    \begin{subfigure}[b]{0.47\textwidth}
        \centering
        \includegraphics[width=\textwidth]{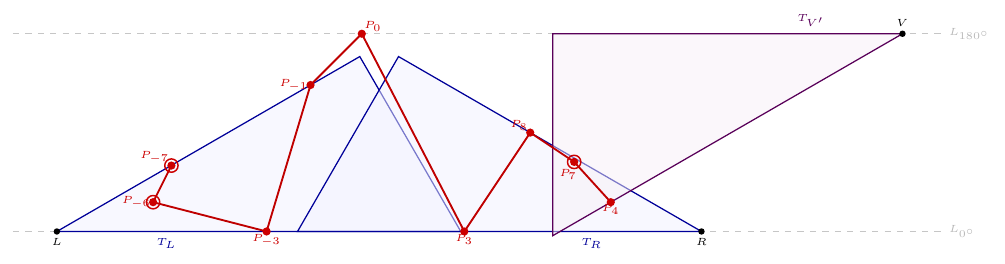}
        \caption[Subcase 2.2b]{Subcase 2.2b. Schematic left order $P_{-7}P_{-6}P_{-3}$ and right order $P_3P_8P_7P_4$.}
        \label{fig:subcase_2_2b}
    \end{subfigure}
    
    \caption[Feasible configurations for Subcase 2.2]{Feasible configurations of the certified closed models for Subcase~2.2 (model ids \texttt{2.2a}, \texttt{2.2b}). The purple triangle is the active $T_{V'}$ placement, the blue triangles are the floor placements $T_L$ and $T_R$, and the labeled contacts make the representative left order $P_{-7}P_{-6}P_{-3}$ and the two representative right-end orders visually explicit. The ringed points carry the strict escape and guard rows: the outer witnesses $P_{-7}$ and $P_7$ lie beyond the short sides of $T_R$ and $T_L$ (\texttt{farR}, \texttt{farL}), and the vertical delimiter $P_{-6}$ witnesses the Case~2 guard $x_V-x_{-6}\ge b$ of the compound predicate \texttt{condV} (whose incidence part places $P_4$ on the slanted side of $T_{V'}$, exactly). The red polylines are the certified parameter-ordered chains of Appendix~\ref{app:summary}, of length at most $\ell(\gamma)$ but certified greater than $1$; the remaining ten left-right order combinations are certified separately (Table~\ref{tab:raw-order-cert}).}
    \label{fig:subcase_2_2}
\end{figure}
\subsection{Case 3: Touches at \texorpdfstring{$P_{-8}$ and $P_8$}{P-8 and P8}}

In this case both $L_{-150^{\circ}}$ and $L_{150^{\circ}}$ touch at the outermost positions. The one-sided branch rule then forces the $120^{\circ}$ contacts to be $P_{-7}$ and $P_7$. On the right end, Lemma~\ref{lem:case3-order-reduction} encloses the weak order of $P_4,P_5,P_7,P_8$ by the $24$ strict orders of these four labels, and the left end is treated analogously. The original notebook records the three representative pairings
\[
\begin{array}{ll}
\text{3a:} & P_{-8}P_{-7}P_{-5}P_{-4}\quad\text{and}\quad P_4P_5P_7P_8,\\
\text{3b:} & P_{-8}P_{-7}P_{-5}P_{-4}\quad\text{and}\quad P_8P_7P_5P_4,\\
\text{3c:} & P_{-4}P_{-5}P_{-7}P_{-8}\quad\text{and}\quad P_8P_7P_5P_4.
\end{array}
\]
The remaining $573$ Case~3 left-right order pairings are generated as closed SOCP models and certified directly in Appendix~\ref{app:raw-order-cert}.

\begin{figure}[p]
    \centering
    \begin{subfigure}[b]{0.82\textwidth}
        \centering
        \includegraphics[width=0.94\textwidth]{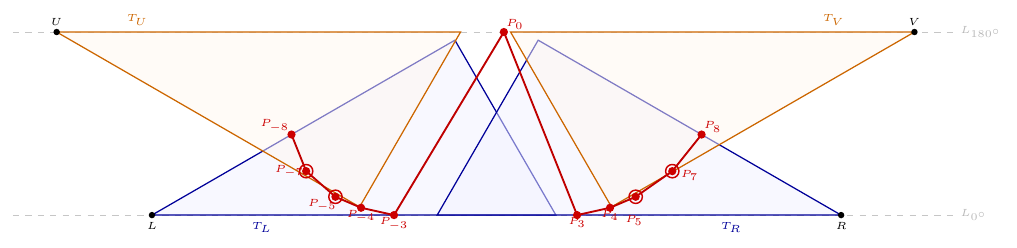}
        \caption{Subcase 3a: schematic chain $P_{-8}P_{-7}P_{-5}P_{-4}P_{-3}P_0P_3P_4P_5P_7P_8$.}
        \label{fig:subcase_3a}
    \end{subfigure}
    
    \smallskip
    \begin{subfigure}[b]{0.82\textwidth}
        \centering
        \includegraphics[width=0.94\textwidth]{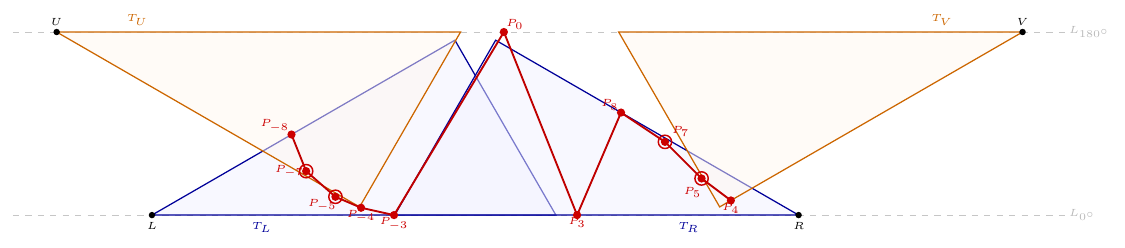}
        \caption{Subcase 3b: schematic chain $P_{-8}P_{-7}P_{-5}P_{-4}P_{-3}P_0P_3P_8P_7P_5P_4$.}
        \label{fig:subcase_3b}
    \end{subfigure}
    
    \smallskip
    \begin{subfigure}[b]{0.82\textwidth}
        \centering
        \includegraphics[width=0.94\textwidth]{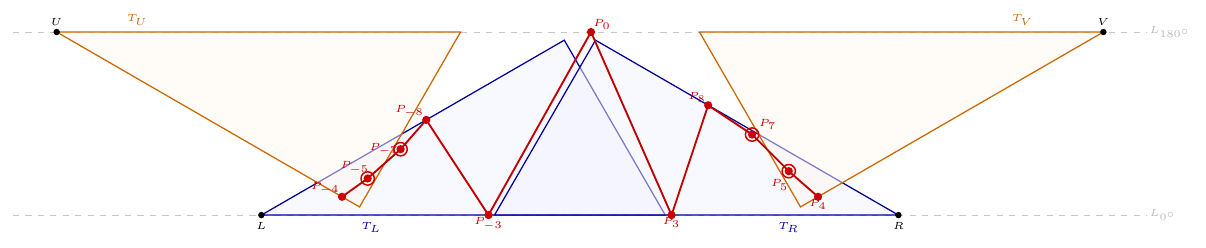}
        \caption{Subcase 3c: schematic chain $P_{-4}P_{-5}P_{-7}P_{-8}P_{-3}P_0P_3P_8P_7P_5P_4$.}
        \label{fig:subcase_3c}
    \end{subfigure}
    
    \caption[Feasible configurations for Case 3]{Feasible configurations of the three representative certified models for Case~3 (model ids \texttt{3a}--\texttt{3c}). Blue outlines denote the floor placements $T_L$ and $T_R$ (incidences \texttt{onlineL}, \texttt{onlineR} placing the outer $150^\circ$ witnesses $P_{-8}$ and $P_8$ on their slanted sides, exactly), orange outlines denote the ceiling placements $T_U$ and $T_V$ (incidences \texttt{onlineU}, \texttt{onlineV} at $P_{-4}$ and $P_4$), and dashed gray lines indicate the horizontal supports. The ringed points carry the strict escape rows: the $120^\circ$ witnesses $P_{\mp7}$ lie beyond the short sides of $T_R$ and $T_L$ (\texttt{farR}, \texttt{farL}), and the tail witnesses $P_{\mp5}$ beyond the short sides of $T_V$ and $T_U$ (\texttt{farV}, \texttt{farU}). The labeled red chains are the certified parameter-ordered chains of Appendix~\ref{app:summary} in the three representative tail-order pairings, each of length at most $\ell(\gamma)$ but certified greater than $1$; the other $573$ pairings are certified separately (Table~\ref{tab:raw-order-cert}).}
    \label{fig:case_3}
\end{figure}
\subsection{Summary of reported numerical values}

Table~\ref{tab:bounds} gives a condensed main-text summary of the representative terminal subcases. Appendix~\ref{app:summary} records the full subcase data and the saved numerical outputs. We write $L_i^{\mathrm{nb}}$ for the saved reported value associated with representative subcase $i$.

\begin{table}[t]
\centering
\caption{Rounded reported values by subcase (condensed summary; all saved reported values are $>1$).}
\label{tab:bounds}
\small
\begin{tabular}{@{}lcc@{}}
\toprule
\textbf{Subcase} & \textbf{Key constraints} & \textbf{Rounded $L_i^{\mathrm{nb}}$} \\
\midrule
1.1a & $P_{\pm 2}$ touched; $T_W$ two-side; $T_L/T_R$ escapes & $1.0048$ \\
1.1b & $P_{\pm 2}$ touched; $T_W$ one-side; $T_L/T_R$ escapes & $1.011$ \\
1.1c & $P_{\pm 2}$ touched; both $T_W$ contacts late; $T_L/T_R$ escapes & $1.016$ \\
1.2 & $P_{-2}, P_7$ touched; $T_L/T_R$ escapes & $1.029$ \\
1.3 & $P_{-7}, P_7$ touched; symmetric escapes & $1.10$ \\
2.1a & $P_{-2}, P_8$ touched; right order $P_3 P_4 P_7 P_8$ & $1.016$ \\
2.1b & $P_{-2}, P_8$ touched; right order $P_3 P_8 P_7 P_4$ & $1.031$ \\
2.2a & $P_{-7}, P_7$ touched; right order $P_3 P_4 P_7 P_8$ & $1.0087$ \\
2.2b & $P_{-7}, P_7$ touched; right order $P_3 P_8 P_7 P_4$ & $1.014$ \\
3a, 3b, 3c & $P_{-8}, P_8$ touched; reduced orders both ends & $1.029$ \\
\bottomrule
\end{tabular}
\end{table}

The smallest rounded reported value in Table~\ref{tab:bounds} occurs in Subcase~1.1a. Appendix~\ref{app:summary} records the corresponding numerical output
\[
L_{1.1a}^{\mathrm{nb}}=1.0048313701863976\ldots > 1.
\]
The refined representative interval certificate in Appendix~\ref{app:interval-cert} proves
\[
\min_{\rm rep} L_i^{\rm cert} \ge 1.0048290490\ldots > 1.0048.
\]
The additional raw-order interval certificate in Appendix~\ref{app:raw-order-cert} proves that every nonrepresentative Case~2 and Case~3 raw-order model also has certified lower bound greater than $1.0048$; its weakest refined coarse-interval lower endpoint is
\[
1.0057165531\ldots>1.0048.
\]

\begin{proof}[Proof of Theorem~\ref{thm:main}]
Assume, for contradiction, that a unit arc is not contained in any congruent copy of $T$. By Lemma~\ref{lem:reduced-obstruction}, there is a weak support-reduced normalized obstruction $\gamma$ with $\ell(\gamma)\le1$. Proposition~\ref{prop:case-tree}, using the support-contact and raw-order lemmas above, places $\gamma$ in one of the $599$ certified closed SOCP models. The escape and guard rows of that model are forced by noncoverage through Lemmas~\ref{lem:canonical-tests} and~\ref{lem:forced-placements}, and Lemma~\ref{lem:row-audit} certifies that the machine-checked model rows coincide with these forced rows. For that model, the recorded contact chain $\Pi_i$ occurs in parameter order on $\gamma$, so Lemma~\ref{lem:ordered-chain} gives
\[
\ell(\Pi_i)\le \ell(\gamma)\le1.
\]
On the other hand, Appendix~\ref{app:interval-cert} and Appendix~\ref{app:raw-order-cert} certify $L_i^{\rm cert}>1$ for every representative and raw-order model, and by definition $L_i^{\rm cert}$ is a lower bound for $\ell(\Pi_i)$ over the corresponding closed feasible set. Hence $\ell(\Pi_i)>1$, a contradiction. Therefore every unit arc is contained in some congruent copy of $T$.
\end{proof}

\section{Certified computational scaling margin}\label{sec:quant}

\begin{proposition}[Scale invariance]\label{prop:scale}
Let $s > 0$ and scale the entire configuration by homothety $X \mapsto sX$. Then:
\begin{enumerate}
\item Support-line contacts are preserved with the same angles.
\item Each escape inequality $d(\overrightarrow{AX}, \theta_{\mathrm{short}}) > b$ transforms to $d(\overrightarrow{s A\, sX}, \theta_{\mathrm{short}}) > s b$.
\item Polygonal lengths scale by $s$.
\end{enumerate}
Consequently, each branch lower bound scales linearly: $L_i \mapsto s L_i$.
\end{proposition}

\begin{proof}
All constraints are homogeneous. Directional distance satisfies $d(sX, \theta) = s\, d(X, \theta)$, and Euclidean lengths scale by $s$. The claims follow directly.
\end{proof}

\begin{proposition}[Certified scaled triangular cover]\label{prop:scaled-cover}
The homothetic copy $T/1.0048$ covers every unit arc. Its area is
\[
\Bigl(\frac{1}{1.0048}\Bigr)^{2}\cdot\frac{24+19\sqrt3}{216}
=0.260956\ldots\;<\;\frac{\pi}{12}\approx0.261799,
\]
the area of the $30^{\circ}$ unit sector. This proves Remark~\ref{rem:conditional-scale}.
\end{proposition}

\begin{proof}
From Appendix~\ref{app:summary}, the smallest reported numerical output among the representative terminal models is
\[
L_{1.1a}^{\mathrm{nb}}=1.0048313701863976\ldots.
\]
Appendix~\ref{app:interval-cert} gives the refined representative endpoint $L_{1.1a}^{\rm cert}=1.0048290490\ldots$, and Appendix~\ref{app:raw-order-cert} gives the weakest raw-order endpoint $1.0057165531\ldots$. Hence every certified representative and raw-order branch satisfies $L_i^{\rm cert}>1.0048$. By Proposition~\ref{prop:scale}, scaling $T$ by $s = 1/1.0048$ scales each lower bound by the same factor, giving $sL_i^{\rm cert}>1$ in every branch. Using the weakest certified endpoint,
\[
\frac{1.0048290490\ldots}{1.0048}\approx 1.00003>1.
\]
The raw-order minimum gives the larger ratio $1.0057165531\ldots/1.0048\approx 1.00091>1$. Hence no unit arc can satisfy the escape constraints for the scaled triangle, so the same case analysis excludes every unit arc for the scaled copy as well.

For the area, the unscaled triangle has $\tfrac{1}{2}ab=(24+19\sqrt3)/216\approx 0.263467$ by our normalization. Multiplying by $s^2 = (1/1.0048)^2$ gives $0.260956$, which is strictly smaller than $\pi/12 \approx 0.261799$.
\end{proof}

\begin{remark}
The main representative table uses the $q=0.999999$ refined logs, giving Subcase~1.1a endpoint $1.0048290490\ldots$. The supplementary tier report also records $q=0.9999$ and $q=0.99999$ comparison runs; the $q=0.999999$ run has thinner $q$-norm slack than $q=0.99999$ but passes the interval checks with guard $10^{-7}$ and singular-value floor $0.25$. Any larger shrink factor would require using the corresponding weakest certified endpoint across both the representative and raw-order families.
\end{remark}

\section{Conclusion}\label{sec:conclusion}

Theorem~\ref{thm:main} settles Wetzel's conjecture: the $30^{\circ}\!-\!60^{\circ}\!-\!90^{\circ}$ triangle $T$ obtained by placing a square of side $1/3$ on the hypotenuse covers every unit arc in the plane. Moreover, by Propositions~\ref{prop:scale} and~\ref{prop:scaled-cover}, the certified margin $L_i^{\rm cert}>1.0048$ in every branch shows that the homothetic copy $T/1.0048$ is again a cover, of area
\[
0.260956\ldots\;<\;\frac{\pi}{12}\approx 0.261799,
\]
strictly smaller than the area of Wetzel's $30^{\circ}$ unit sector cover \cite{panraksa2020}. Within the convex Wetzel-cover setting, the scaled triangle $T/1.0048$ is therefore, to our knowledge, the smallest cover for unit arcs currently known.

\section*{Acknowledgements}

The authors would like to thank the Development and Promotion of Science and Technology Talents Project (DPST), a Thai government scholarship.

\section*{Statements and Declarations}

\subsection*{Funding}
This research received no specific grant from any funding agency in the public, commercial, or not-for-profit sectors.

\subsection*{Competing interests}
The authors have no competing interests to declare.

\subsection*{Code availability}
The computational code and verification package supporting the proof are available at \url{https://github.com/chatchawanpan-dev/wetzel-triangle-computation}. The repository contains the Mathematica notebook and rendered notebook PDF, the SOCP model JSON files, the Python validation scripts, the stored certificate logs, the verification outputs, and the package manifest.

\subsection*{Data availability}
Appendix~\ref{app:opt} documents the representative and raw-order optimization models, the subcase-specific constraints, the numerical outputs used in the case analysis, and the independent interval-certificate files. The model files, certificate logs, verification reports, and SHA-256 manifest are included in the GitHub repository and are also provided as supplementary material (Online Resource~1).

\backmatter

\begin{appendices}

\section{Numerical optimization data}\label{app:opt}

This appendix records the Mathematica notebook formulations and numerical outputs for the eleven original representative terminal subcases of Section~\ref{sec:cases}, together with the direct Python SOCP model 1.1c for the both-late Subcase~1.1 branch, and then the independent interval certificates for the inequalities $L_i^{\rm cert}>1$ on the representative and raw-order models. The notebook file is \texttt{WetzelTriangle\_OptimizationCode.nb}; its header identifies Mathematica~10.4 as the creating version, and the saved front-end metadata identifies Mathematica~12.3 for Mac OS X ARM.

\subsection{Triangle parameters, notebook environment, and common primitives}\label{app:triangle-normalization}

We use exactly the same normalization as in Theorem~\ref{thm:main}. The triangle has side lengths
\[
a=\frac{3+4\sqrt{3}}{18},\qquad
b=\frac{4+\sqrt{3}}{6},\qquad
c=2a=\frac{3+4\sqrt{3}}{9},
\]
and area
\[
\frac12ab=\frac{24+19\sqrt3}{216}=0.2634674321\ldots
\]
Numerically, these are
\[
a\approx0.5515668461,\qquad b\approx0.9553418013,\qquad c\approx1.1031336923.
\]
These constants match the values used in the notebook.

\medskip\noindent\textbf{Common variables and objective.}
For each terminal subcase, the notebook introduces:
\begin{itemize}
  \item contact-point coordinates $p[i]=\{x[i],y[i]\}$ for the active index set;
  \item anchor points such as $L,R,S,T,U,V,W$ determining the canonical triangle placements;
  \item derived support parameters such as \texttt{yfloor}, \texttt{yceiling}, \texttt{xmin}, and \texttt{xmax}.
\end{itemize}
The polygonal objective is
\[
\texttt{ds[arc] := Sum[d2[arc[[i]], arc[[i+1]]], \{i, Length[arc]-1\}]},
\]
and each reported value is the first component of a call of the form
\[
\texttt{nm = NMinimize[\{ds[arc], constraints\}, var]}.
\]

\medskip\noindent\textbf{Common predicates.}
The notebook uses affine incidence predicates such as \texttt{onlineL}, \texttt{onlineR}, \texttt{onlineU}, \texttt{onlineV}, \texttt{onlineWl}, and \texttt{onlineWr} to place active contact points on the relevant supporting sides; compound predicates such as \texttt{condL}, \texttt{condR}, \texttt{condS}, and \texttt{condT} encode a canonical placement together with its side-incidence conditions, and \texttt{condSfarL}, \texttt{condTfarR} are the combined incidence-and-escape constraints of Subcases~1.1a--1.1c. Anchor names follow the notebook convention and are local to the subcase; see the branch-local substitutions below.

The directional-distance convention is
\[
d(\overrightarrow{AX},\theta):=\langle X-A,e_\theta\rangle,
\]
where the notebook predicate \texttt{dd[\(\cdot\),theta]} uses the escape vector $e_\theta$ from Section~\ref{subsec:support-distance}. Let $P_i=(x_i,y_i)$, and let the branch-local anchors be
\[
L=(x_L,y_3),\quad R=(x_R,y_3),\quad S=(x_{-6},y_S),\quad T=(x_6,y_T),
\]
\[
U=(x_U,y_0),\quad V=(x_V,y_0),\quad W=(x_W,y_W).
\]
The gauge is $x_{-3}=0$, $y_3=0$, and $y_{-3}=y_3$. Table~\ref{tab:predicate-dictionary} gives the explicit affine dictionary for the predicate names that occur in Table~\ref{tab:cert-bounds}. The machine-readable JSON files supplied as Online Resource~1 store the same constraints as row vectors, with the branch-local substitutions described below the table.

\begin{table}[t]
\centering
\caption{Notebook predicate dictionary for the terminal SOCP models.}\label{tab:predicate-dictionary}
\scriptsize
\setlength{\tabcolsep}{2.5pt}
\renewcommand{\arraystretch}{1.08}
\begin{tabularx}{\textwidth}{@{}p{2.45cm}X@{}}
\toprule
\textbf{Predicate} & \textbf{Closed affine or conic row in the certified model} \\
\midrule
\texttt{onlineL[P\_i]} & $y_i-y_3=\tan30^\circ\,(x_i-x_L)$. \\
\texttt{onlineR[P\_i]} & $y_i-y_3=\tan(-30^\circ)\,(x_i-x_R)$. \\
\texttt{onlineS[P\_i]} & $y_i-y_S=\tan120^\circ\,(x_i-x_{-6})$. \\
\texttt{onlineT[P\_i]} & $y_i-y_T=\tan60^\circ\,(x_i-x_6)$. \\
\texttt{onlineWl[P\_i]} & $y_i-y_W=\tan75^\circ\,(x_i-x_W)$. \\
\texttt{onlineWr[P\_i]} & $y_i-y_W=\tan(-75^\circ)\,(x_i-x_W)$. \\
\texttt{onlineU} & $y_{-4}-y_0=\tan(-30^\circ)\,(x_{-4}-x_U)$. \\
\texttt{onlineV} & $y_4-y_0=\tan30^\circ\,(x_4-x_V)$. \\
\texttt{farL[P\_i]} & $d(\overrightarrow{LP_i},30^\circ)\ge b$. \\
\texttt{farR[P\_i]} & $d(\overrightarrow{RP_i},150^\circ)\ge b$. \\
\texttt{farU} & $d(\overrightarrow{UP_5},-30^\circ)\ge b$. \\
\texttt{farV} & $d(\overrightarrow{VP_{-5}},-150^\circ)\ge b$. \\
\texttt{condL[P\_i]} & The equality \texttt{onlineL[P\_i]} and the floor-placement guard $x_L+b\le x_6$. \\
\texttt{condR[P\_i]} & The equality \texttt{onlineR[P\_i]} and the floor-placement guard $x_{-6}\le x_R-b$. \\
\texttt{condS[P\_i]} & The equality \texttt{onlineS[P\_i]}, namely $y_i-y_S=\tan120^\circ\,(x_i-x_{-6})$, and the wall-placement guard $y_3\le y_S-b$. \\
\texttt{condT[P\_i]} & The equality \texttt{onlineT[P\_i]}, namely $y_i-y_T=\tan60^\circ\,(x_i-x_6)$, and the wall-placement guard $y_3\le y_T-b$. \\
\texttt{condSfarL[P\_i]} & The equality $y_i-y_S=\tan120^\circ\,(x_i-x_{-6})$, the wall-placement guard $y_3\le y_S-b$, and the closed floor escape row $d(\overrightarrow{LP_i},30^\circ)\ge b$. \\
\texttt{condTfarR[P\_i]} & The equality $y_i-y_T=\tan60^\circ\,(x_i-x_6)$, the wall-placement guard $y_3\le y_T-b$, and the closed floor escape row $d(\overrightarrow{RP_i},150^\circ)\ge b$. \\
\texttt{condV} & Case~2 compound row: \texttt{onlineV}, namely $y_4-y_0=\tan30^\circ\,(x_4-x_V)$, together with the guard $x_V-x_{-6}\ge b$. This is distinct from \texttt{farV}. \\
\texttt{dot\_gt\_b} & Generic closed directional escape row $d(\overrightarrow{AX},\theta)\ge b$ for the anchor and direction recorded in the JSON model. \\
\texttt{below\_line} & Generic affine half-plane row placing a point weakly below the specified support side; strict notebook forms are closed by replacing $<$ with $\le$. \\
\bottomrule
\end{tabularx}
\end{table}

\medskip\noindent\textbf{Branch-local substitutions.}
In Case~2, the reflected ceiling placement denoted $T_{V'}$ in the main text is represented in the JSON rows by the anchor name $V$. The rows \texttt{onlineV} and \texttt{farV} are separate in Subcases~2.1a--b, while \texttt{condV} in Subcases~2.2a--b expands to \texttt{onlineV} plus $x_V-x_{-6}\ge b$. In Case~3, the same predicate names \texttt{onlineU}, \texttt{onlineV}, \texttt{farU}, and \texttt{farV} refer instead to the unprimed ceiling placements $T_U$ and $T_V$ in Figure~\ref{fig:case_3}. The certified closed model replaces each strict notebook escape inequality by the corresponding weak row in Table~\ref{tab:predicate-dictionary}.

\medskip\noindent\textbf{Row-vector example.}
For Subcases~2.2a--b, the JSON row
\[
\texttt{\{"type": "condV", "online\_point": "4", "escape\_point": "-6"\}}
\]
generates one equality and one inequality. With variables ordered as in the script output \texttt{var\_names}, the equality row is
\[
y_4-y_0-\tan30^\circ\,(x_4-x_V)=0,
\]
so the nonzero coefficients are $-\tan30^\circ$ on $x_4$, $1$ on $y_4$, $\tan30^\circ$ on $x_V$, and $-1$ on $y_0$. The guard row is
\[
x_V-x_{-6}-b\ge0,
\]
with nonzero coefficients $1$ on $x_V$, $-1$ on $x_{-6}$, and constant term $-b$. By contrast, the separate predicate \texttt{farV} used in Subcases~2.1a--b gives
\[
\langle P_{-5}-V,e_{-150^\circ}\rangle-b\ge0,
\]
which is a different affine row involving $P_{-5}$.

\medskip\noindent\textbf{Open notebook models, closed certified models, and solver settings.}
For representative terminal subcase $i$, let $F_i^{\rm open}$ be the feasible set obtained from the strict escape inequalities exactly as they appear in the notebook, and let $F_i^{\rm cl}$ be the closed feasible set obtained by replacing every strict escape inequality by the corresponding weak inequality. Let
\[
L_i^{\rm nb}
\]
denote the numerical value reported by \texttt{NMinimize} for the open notebook model, and let
\[
L_i^{\rm cert}=\inf_{F_i^{\rm cl}}\ell(\Pi_i)
\]
denote the closed-model infimum certified by the interval validation. Since $F_i^{\rm open}\subseteq F_i^{\rm cl}$, we have
\[
\inf_{F_i^{\rm open}}\ell(\Pi_i)\ge L_i^{\rm cert}.
\]
Together with the raw-order certificates in Appendix~\ref{app:raw-order-cert}, this proves the theorem by establishing $L_i^{\rm cert}>1$ for every certified branch. No explicit \texttt{Method}, \texttt{WorkingPrecision}, \texttt{AccuracyGoal}, or \texttt{PrecisionGoal} options are specified in the recorded notebook cells, so the notebook values are used only as numerical model outputs, not as rigorous proof certificates.

\subsection{Worked example: Subcase \texorpdfstring{1.1a}{1.1a}}\label{app:case11a}

\medskip\noindent\textbf{Active supports and contacts.}
In subcase 1.1a the unit arc is touched by supports at angles
\[
\Theta=\{\pm150^\circ,\ \pm120^\circ,\ \pm90^\circ,\ \pm15^\circ\},
\]
with active contact points at $P_{\pm1}$ and $P_{\pm2}$ as in the main text. The horizontal base contacts are $P_{-3},P_3$ with $y_{-3}=y_3=0$, and the apex contact $P_0$ has height $y_0=H>0$.

\medskip\noindent\textbf{Polygonal chain used for the lower bound.}
The notebook cell for Subcase~1.1a sets
\[
\texttt{arc = Table[p[i], \{i, \{-6,-3.5,-3,-2.5,-2,-1,1,2,2.5,3,3.5,6\}\}]},
\]
then minimizes \texttt{ds[arc]} subject to the corresponding incidence and escape constraints
(\texttt{condL}, \texttt{condR}, \texttt{condTfarR}, \texttt{condSfarL}, \texttt{onlineWl}, \texttt{onlineWr}, and two \texttt{dd[\(\cdot\)]>b} guards).

\medskip\noindent\textbf{Reported numerical value.}
Running \texttt{NMinimize} on this cell returns
\[
L^*_{1.1a}\approx 1.0048313701863976
\]
(\texttt{Out[291]} in the notebook). Consequently,
\[
L^*_{1.1a}-1\approx 4.83\times 10^{-3}>0,
\]
which is the strongest numerical gap suggested by the notebook. The certified gap used for Theorem~\ref{thm:main} is the interval lower endpoint in Appendix~\ref{app:interval-cert}.

\subsection{Per-subcase notebook data}\label{app:summary}

Table~\ref{tab:cert-bounds} supplements the condensed main-text Table~\ref{tab:bounds} by recording, for each representative terminal subcase, the chain used in the model, the principal predicates, and the saved reported output. Subcase~1.1c is the direct Python SOCP model; the other entries are the original Mathematica notebook models.

\begin{table}[p]
\centering
\setlength{\abovecaptionskip}{2pt}
\setlength{\belowcaptionskip}{3pt}
\caption{Per-subcase model data and saved solver outputs.}\label{tab:cert-bounds}
\fontsize{7}{7.1}\selectfont
\setlength{\tabcolsep}{3pt}
\renewcommand{\arraystretch}{0.96}
\begin{tabularx}{\textwidth}{@{}p{1.30cm}p{2.60cm}X p{2.85cm}@{}}
\toprule
\textbf{Subcase} & \textbf{\texttt{arc} list} & \textbf{Principal predicates} & \textbf{Value} \\
\midrule
1.1a &
\shortstack[l]{$\{-6,-3.5,-3,$\\$-2.5,-2,-1,1,$\\$2,2.5,3,3.5,6\}$} &
\predlines{\texttt{condL[p[-1]]}, \texttt{condR[p[1]]},\\
\texttt{condTfarR[p[-2]]}, \texttt{condSfarL[p[2]]},\\
\texttt{onlineWl[p[-2.5]]}, \texttt{onlineWr[p[2.5]]},\\
\texttt{dd[p[3.5]-W,-75 Degree]>b},\\
\texttt{dd[p[-3.5]-W,-105 Degree]>b}} &
$1.0048313701863976$ \\
1.1b &
\shortstack[l]{$\{-6.5,-6,-3,-2,$\\$-1,1,2,3,$\\$3.5,6\}$} &
\predlines{\texttt{condL[p[-1]]}, \texttt{condR[p[1]]},\\
\texttt{condTfarR[p[-2]]}, \texttt{condSfarL[p[2]]},\\
\texttt{onlineWl[p[-6.5]]},\\
\texttt{dd[p[3.5]-W,-75 Degree]>b},\\
\texttt{y[-2] < Tan[75 Degree](x[-2]-xW)+yW}} &
$1.011404204881239$ \\
1.1c &
\shortstack[l]{$\{-6.5,-6,-3,-2,$\\$-1,1,2,3,$\\$6,6.5\}$} &
\predlines{\texttt{condL[p[-1]]}, \texttt{condR[p[1]]},\\
\texttt{condTfarR[p[-2]]}, \texttt{condSfarL[p[2]]},\\
\texttt{onlineWl[p[-6.5]]}, \texttt{onlineWr[p[6.5]]},\\
\texttt{dd[p[-6.5]-W,-105 Degree]>b},\\
\texttt{dd[p[6.5]-W,-75 Degree]>b}} &
$1.016189066467773$ \\
1.2 &
\shortstack[l]{$\{-3,-2,-1,$\\$1,3,7\}$} &
\predlines{\texttt{onlineL[p[-1]]}, \texttt{onlineR[p[1]]},\\
\texttt{farR[p[-2]]}, \texttt{farL[p[7]]}} &
$1.029235945606682$ \\
1.3 &
\shortstack[l]{$\{-7,-3,-1,$\\$1,3,7\}$} &
\predlines{\texttt{onlineL[p[-1]]}, \texttt{onlineR[p[1]]},\\
\texttt{farR[p[-7]]}, \texttt{farL[p[7]]}} &
$1.1031325980781586$ \\
2.1a &
\shortstack[l]{$\{-5,-3,-2,-1,$\\$0,3,4,7,8\}$} &
\predlines{\texttt{onlineL[p[-1]]}, \texttt{onlineR[p[8]]},\\
\texttt{farR[p[-2]]}, \texttt{farL[p[7]]},\\
\texttt{onlineV}, \texttt{farV}} &
$1.0163224937990105$ \\
2.1b &
\shortstack[l]{$\{-5,-3,-2,-1,$\\$0,3,8,7,4\}$} &
\predlines{\texttt{onlineL[p[-1]]}, \texttt{onlineR[p[8]]},\\
\texttt{farR[p[-2]]}, \texttt{farL[p[7]]},\\
\texttt{onlineV}, \texttt{farV}} &
$1.0313464651372972$ \\
2.2a &
\shortstack[l]{$\{-7,-6,-3,-1,$\\$0,3,4,7,8\}$} &
\predlines{\texttt{onlineL[p[-1]]}, \texttt{onlineR[p[8]]},\\
\texttt{farR[p[-7]]}, \texttt{farL[p[7]]},\\
\texttt{condV}} &
$1.0087527063430077$ \\
2.2b &
\shortstack[l]{$\{-7,-6,-3,-1,$\\$0,3,8,7,4\}$} &
\predlines{\texttt{onlineL[p[-1]]}, \texttt{onlineR[p[8]]},\\
\texttt{farR[p[-7]]}, \texttt{farL[p[7]]},\\
\texttt{condV}} &
$1.0142026474571697$ \\
3a &
\shortstack[l]{$\{-8,-7,-5,-4,$\\$-3,0,3,4,$\\$5,7,8\}$} &
\predlines{\texttt{onlineL[p[-8]]}, \texttt{onlineR[p[8]]},\\
\texttt{farR[p[-7]]}, \texttt{farL[p[7]]},\\
\texttt{onlineU}, \texttt{onlineV}, \texttt{farU}, \texttt{farV}} &
$1.0295071769383997$ \\
3b &
\shortstack[l]{$\{-8,-7,-5,-4,$\\$-3,0,3,8,$\\$7,5,4\}$} &
\predlines{\texttt{onlineL[p[-8]]}, \texttt{onlineR[p[8]]},\\
\texttt{farR[p[-7]]}, \texttt{farL[p[7]]},\\
\texttt{onlineU}, \texttt{onlineV}, \texttt{farU}, \texttt{farV}} &
$1.0295069260672876$ \\
3c &
\shortstack[l]{$\{-4,-5,-7,-8,$\\$-3,0,3,8,$\\$7,5,4\}$} &
\predlines{\texttt{onlineL[p[-8]]}, \texttt{onlineR[p[8]]},\\
\texttt{farR[p[-7]]}, \texttt{farL[p[7]]},\\
\texttt{onlineU}, \texttt{onlineV}, \texttt{farU}, \texttt{farV}} &
$1.0295069435171533$ \\
\midrule
\textbf{Global min} & & & $\mathbf{1.0048313701863976}$ \\
\bottomrule
\end{tabularx}
\end{table}

\subsection{Independent interval certificate}\label{app:interval-cert}

Each model in Table~\ref{tab:cert-bounds} was reconstructed as a convex second-order cone program (SOCP). The variables are the active contact-point and anchor coordinates; the constraints are the affine incidence, ordering, and escape inequalities; and the objective is the same polygonal chain length. The SOCP optimum reconstruction reproduces the eleven original Mathematica values and includes the direct both-late model 1.1c. The certified endpoints in Table~\ref{tab:interval-cert} are conservative dual lower bounds for the closed SOCP models and are smaller by construction.

For a model with coordinate vector $z$ and chain segments $s_j=S_jz$, the primal closed SOCP is written in the standard epigraph form
\[
\inf_{z,\tau}\ \sum_{j=1}^m \tau_j
\quad\text{subject to}\quad
\|S_jz\|_2\le \tau_j,\qquad
a_k\cdot z+c_k=0,\qquad
g_\ell\cdot z+d_\ell\ge0 .
\]
The sign convention for inequalities is always $g_\ell\cdot z+d_\ell\ge0$ with multiplier $\mu_\ell\ge0$; equality multipliers $\lambda_k$ are unrestricted. All exact constants used in the row matrices lie in the real algebraic field generated by $\sqrt2,\sqrt3,\sqrt6$, and the scripts check the displayed certificate inequalities using outward-rounded rational interval enclosures of those algebraic constants. The active-row convention is conservative: rows with tiny inactive multipliers may be dropped only after the perturbation audit proves that setting them to zero preserves feasibility of the repaired dual certificate.

\begin{theorem}[Finite SOCP perturbation certificate]\label{thm:socp-cert}
Consider one closed SOCP model with coordinate vector $z$, affine equality rows
\[
a_k\cdot z+c_k=0,
\]
affine inequality rows
\[
g_\ell\cdot z+d_\ell\ge0,
\]
and polygonal segment maps $s_j=S_jz$ (coordinate differences between consecutive contact points), so that the objective is $\sum_j\|s_j\|$. A numerical dual certificate consists of segment dual vectors $q_j$, unrestricted equality multipliers $\lambda_k$, and nonnegative inequality multipliers $\mu_\ell$ such that the stationarity residual is
\[
r=\sum_j S_j^\top q_j-\sum_k\lambda_k a_k-\sum_\ell\mu_\ell g_\ell .
\]
Let $A$ be the active set of inequality multipliers retained after dropping tiny nonactive multipliers, and let
\[
u=(q_1,\ldots,q_m,\lambda,\mu_A)
\]
be the retained certificate-variable vector. The norm on $u$ is the product Euclidean norm: each two-dimensional block $q_j$, each equality multiplier, and each retained inequality multiplier is included in the same Euclidean vector. Let $M$ be the stationarity matrix whose columns are the coefficients of these retained variables in the stationarity equation. Write $n$ for the dimension of the coordinate vector $z$, so that the residual satisfies $r\in\R^n$, and regard $M$ as a linear map from the retained certificate-variable space, of dimension $2m+K+|A|$ with $K$ the number of equality rows, onto the stationarity-residual space $\R^n$. Assume that $M$ has full row rank, and let $\sigma_{\min}(M)$ denote its smallest singular value as such a map; the row-rank hypothesis and the floor $\sigma_{\min}(M)\ge\sigma_0$ are certified by the row-basis singular-value checks of the algebraic interval layer (Table~\ref{tab:interval-cert}). The active certificate value is
\[
\alpha=-\sum_k c_k\lambda_k-\sum_{\ell\in A}d_\ell\mu_\ell,
\]
where dropped multipliers are set to zero. Let $h$ be the coefficient vector of this scalar functional in the retained variables $u$; in particular the $q_j$ blocks have coefficient zero, the equality-multiplier coefficients are $-c_k$, and the active inequality-multiplier coefficients are $-d_\ell$. Suppose outward-rounded interval checks verify
\[
\|r\|_2\le \rho,\qquad \sigma_{\min}(M)\ge\sigma_0>0.
\]
Then there is a correction $\Delta$ of the retained certificate variables with
\[
\|\Delta\|_2\le \rho/\sigma_0
\]
that makes stationarity exact. If the verified guard satisfies
\[
\rho/\sigma_0+\eta
<
\min\{1-\max_j\|q_j\|,\ \min_{\ell\in A}\mu_\ell\},
\]
then the corrected certificate still has $\|q_j\|\le1$ and nonnegative inequality multipliers. If, in addition,
\[
\alpha-\|h\|_2\,(\rho/\sigma_0)-\eta>1,
\]
then every feasible point of the closed SOCP model has objective greater than one.
\end{theorem}

\begin{proof}
For an exact dual certificate with $\|q_j\|\le1$ and $\mu_\ell\ge0$, the standard SOCP dual inequality gives
\[
\sum_j\|s_j\|\ge \sum_j\langle q_j,s_j\rangle
=
-\sum_k c_k\lambda_k-\sum_\ell d_\ell\mu_\ell
+\sum_\ell \mu_\ell(g_\ell\cdot z+d_\ell),
\]
because the equality rows vanish on feasible points. The final term is nonnegative, so the displayed constant is a valid lower bound.

The interval matrix check gives $\sigma_{\min}(M)\ge\sigma_0$. Because $M$ has full row rank, $MM^\top$ is invertible and the least-norm solution $\Delta=-M^\top(MM^\top)^{-1}r$ of $M\Delta=-r$ satisfies $\|\Delta\|_2\le\|r\|_2/\sigma_{\min}(M)\le\rho/\sigma_0$. Applying this correction makes stationarity exact. Since the product Euclidean norm dominates every block, $\|\Delta\|_2$ bounds the correction of each individual block: each $q_j$ block, each $\lambda_k$, and each retained $\mu_\ell$ moves by at most $\|\Delta\|_2$. This is why the guard inequality keeps every corrected segment dual vector $q_j$ inside the unit dual ball ($\|q_j\|\le1$) and every retained inequality multiplier nonnegative ($\mu_\ell\ge0$); the dropped multipliers are set to zero. The correction changes the lower-bound constant by at most $\|h\|_2\|\Delta\|_2$, and the scalar guard $\eta$ accounts for the outward-rounded interval allowances used in the scripts. Therefore the corrected exact certificate proves the stated lower bound.
\end{proof}

In the certificate logs, $\rho$ is the field \texttt{residual\_bound} and the audited residual is \texttt{stationarity\_residual\_norm2}; $\sigma_0$ is \texttt{sigma\_floor}; and $\eta$ is the field \texttt{guard}. The active set $A$ is stored as \texttt{active\_mu}, with omitted inactive multipliers stored as \texttt{dropped\_mu}. Conic-block feasibility of the repaired dual vector is certified by the fields \texttt{q\_margin\_after\_guard\_interval} and \texttt{q\_margin\_to\_unit}; multiplier nonnegativity is certified by \texttt{mu\_margin\_after\_guard\_interval} and \texttt{active\_mu\_margin}. The lower interval is stored as \texttt{guarded\_lower\_bound\_interval} and is accepted only when its lower endpoint is strictly greater than $1$.

\begin{proposition}[Dual lower-bound certificate]\label{prop:dual-certificate}
For each representative terminal subcase, the lower endpoint reported in Table~\ref{tab:interval-cert} is a rigorous lower bound for the corresponding closed SOCP model. The same certificate theorem and validation scripts are used for the raw-order models in Table~\ref{tab:raw-order-cert}.
\end{proposition}

\begin{proof}
The validation pipeline instantiates Theorem~\ref{thm:socp-cert}. The numerical dual certificate files provide the vectors \texttt{q}, equality multipliers \texttt{lambda}, inequality multipliers \texttt{mu}, and the equality and inequality constants. The perturbation audit selects the active inequality multipliers, records the dropped ones, and bounds the stationarity residual and repair size. The refined representative table uses $q=0.999999$ dual vectors with coarse bounds $\rho=10^{-7}$ and $\sigma_0=0.25$, together with the scalar guard $\eta=10^{-7}$. The refined raw-order table uses the retained $q=0.997$ raw-order dual vectors with $\rho=10^{-8}$, $\sigma_0=0.25$, and $\eta=10^{-5}$. In both families the scalar interval check verifies that the repaired certificate still has $\|q_j\|\le1$, nonnegative multipliers, and lower bound greater than one. The algebraic matrix check rebuilds the stationarity matrices from outward-rounded algebraic constants and verifies the residual and row-span bounds reported in Table~\ref{tab:interval-cert}. Hence the table entries are valid lower bounds for the enlarged closed feasible sets, and therefore also for the original strict notebook models.
\end{proof}

\begin{table}[t]
\centering
\caption{Certificate input files and model fields used to instantiate Theorem~\ref{thm:socp-cert}.}\label{tab:certificate-fields}
\scriptsize
\setlength{\tabcolsep}{4pt}
\renewcommand{\arraystretch}{1.16}
\begin{tabularx}{\textwidth}{@{}p{3.35cm}X@{}}
\toprule
\textbf{File or field} & \textbf{Mathematical role} \\
\midrule
Model JSON & Closed SOCP specifications in \texttt{wetzel\_socp\_models.json}: subcase id, chain \texttt{arc}, active constraints, anchors, and stored reported value. The scripts reconstruct the segment maps and affine rows from these entries. \\
Variables & Coordinate vector $z$, stored as \texttt{var\_names}, for the reconstructed equality and inequality matrices. \\
Constants & The fields \texttt{eq\_constants} and \texttt{ineq\_constants}, giving $c_k$ and $d_\ell$ in $a_k\cdot z+c_k=0$ and $g_\ell\cdot z+d_\ell\ge0$. \\
Segment duals & Segment dual vectors $q_j$, generated with radius $\texttt{q\_radius}=0.999999$ for the refined representative table and $\texttt{q\_radius}=0.997$ for the raw-order table. \\
Multipliers & Equality multipliers $\lambda_k$ and inequality multipliers $\mu_\ell$, stored as \texttt{lambda} and \texttt{mu}. \\
\bottomrule
\end{tabularx}
\end{table}

\begin{table}[t]
\centering
\caption{Perturbation-audit fields used in the independent interval certificate.}\label{tab:certificate-audit-fields}
\scriptsize
\setlength{\tabcolsep}{4pt}
\renewcommand{\arraystretch}{1.16}
\begin{tabularx}{\textwidth}{@{}p{3.55cm}X@{}}
\toprule
\textbf{Field} & \textbf{Mathematical role} \\
\midrule
Active rows & The \texttt{active\_mu} and \texttt{dropped\_mu} fields record inequality rows retained in the perturbation repair and rows fixed to zero. \\
Residual norm & Euclidean norm of the numerical stationarity residual before repair, stored as \texttt{stationarity\_residual\_norm2}. \\
Singular value & Numerical singular-value check for the active stationarity matrix; the interval proof uses the weaker floor \texttt{sigma\_floor}. \\
Repair bounds & Coarse interval assumptions, stored as \texttt{residual\_bound} and \texttt{sigma\_floor}. The refined representative table uses $\rho=10^{-7}$ and $\sigma_0=0.25$; the refined raw-order table uses $\rho=10^{-8}$ and $\sigma_0=0.25$. \\
Lower interval & Outward-rounded lower-bound interval after repair and guard subtraction, stored as \texttt{guarded\_lower\_bound\_interval}. \\
Algebraic residual & Interval upper bound for the residual rebuilt from outward-rounded constants. \\
Row-span residual & Interval check that omitted rows are in the certified row span. \\
\bottomrule
\end{tabularx}
\end{table}

For the main representative validation, we use refined dual SOCP certificates with segment vectors satisfying $\|q_j\|\le 0.999999$. The scalar interval layer applies a $10^{-7}$ guard and a $4\times10^{-7}$ stationarity-repair allowance to obtain the certified lower endpoints in Table~\ref{tab:interval-cert}. The raw-order validation retains the $q=0.997$ certificates and applies a $10^{-5}$ guard and a $4\times10^{-8}$ stationarity-repair allowance. The algebraic matrix layer rebuilds the stationarity matrices from outward-rounded interval enclosures of the exact constants
\[
\sqrt2,\quad \sqrt3,\quad \sqrt6,
\]
and the derived constants $\tan 30^\circ$, $\tan 60^\circ$, $\tan 75^\circ$, and the corresponding unit normal vectors. It then interval-checks the stationarity residuals, row-basis singular-value floor, and omitted-row dependencies.

\begin{table}[p]
\centering
\caption{Independent interval-certificate lower endpoints and algebraic matrix checks. The column ``Lower endpoint'' is the lower endpoint of the stored interval field \texttt{guarded\_lower\_bound\_interval}. All printed decimals are displayed conservatively: lower endpoints are truncated, and the residual and row-span upper bounds are rounded upward, so every printed value is itself a valid bound in the stated direction.}\label{tab:interval-cert}
\scriptsize
\setlength{\tabcolsep}{3pt}
\renewcommand{\arraystretch}{1.12}
\begin{tabular}{@{}lccc@{}}
\toprule
\textbf{Subcase} & \textbf{Lower endpoint} & \textbf{Residual upper} & \textbf{Row-span upper} \\
\midrule
1.1a & $1.0048290490$ & $1.2366\times 10^{-9}$ & $5.7380\times 10^{-12}$ \\
1.1b & $1.0114028836$ & $8.4971\times 10^{-8}$ & $8.4666\times 10^{-12}$ \\
1.1c & $1.0161871592$ & $2.3954\times 10^{-8}$ & $9.4685\times 10^{-12}$ \\
1.2 & $1.0292360763$ & $1.0033\times 10^{-11}$ & $3.1553\times 10^{-12}$ \\
1.3 & $1.1031319489$ & $9.9978\times 10^{-12}$ & $3.1553\times 10^{-12}$ \\
2.1a & $1.0163207040$ & $1.2280\times 10^{-11}$ & $3.1559\times 10^{-12}$ \\
2.1b & $1.0313447787$ & $1.2172\times 10^{-11}$ & $3.1563\times 10^{-12}$ \\
2.2a & $1.0087511086$ & $1.2095\times 10^{-11}$ & $3.1564\times 10^{-12}$ \\
2.2b & $1.0142008269$ & $1.2031\times 10^{-11}$ & $3.1562\times 10^{-12}$ \\
3a & $1.0295048227$ & $1.3874\times 10^{-11}$ & $3.1571\times 10^{-12}$ \\
3b & $1.0295048226$ & $1.4074\times 10^{-11}$ & $3.1564\times 10^{-12}$ \\
3c & $1.0295048221$ & $1.4257\times 10^{-11}$ & $3.1558\times 10^{-12}$ \\
\midrule
\textbf{Representative minimum} & $\mathbf{1.0048290490}$ & & \\
\bottomrule
\end{tabular}
\end{table}

The decimal entries in Table~\ref{tab:interval-cert} are displayed conservatively, as described in the caption, and are not used as certificates by themselves. These are guarded dual lower endpoints, not the Mathematica optimum values in Table~\ref{tab:cert-bounds}; for example, Subcase~1.1a has Mathematica value $1.0048313701863976\ldots$, while the $q=0.999999$ certified endpoint used in the main representative table is $1.0048290490\ldots$. The proof uses the full outward-rounded intervals stored in the JSON logs, in particular the field \texttt{guarded\_lower\_bound\_interval}; those intervals have lower endpoints strictly greater than one. The supplementary tier report also records $q=0.9999$ and $q=0.99999$ comparison checks; the $q=0.99999$ Subcase~1.1a endpoint is $1.0048177015\ldots$ with wider $q$-norm slack.

The algebraic matrix check also proves a row-basis singular-value floor of $0.25$ in every representative subcase. The smallest numerical row-basis singular value found in the checked bases is $0.252350$, and the smallest preconditioned interval Gershgorin lower endpoint is approximately $0.999999999948$. Thus every representative closed SOCP model has certified lower bound greater than one.

\subsection{Raw-order SOCP certificates}\label{app:raw-order-cert}

The nonrepresentative Case~2 and Case~3 tail orders are certified directly rather than reduced by endpoint replacement or by an unproved Case~3 order exclusion. The supplementary file \texttt{data/wetzel\_socp\_raw\_order\_models.json} contains the $587$ additional closed SOCP models: $14$ nonrepresentative Case~2 models and $573$ nonrepresentative Case~3 models. Together with the three representative Case~3 notebook models, these cover all $24\times24$ Case~3 left-right order pairings. The auxiliary models use the same affine predicate dictionary as Table~\ref{tab:predicate-dictionary}, but replace the representative \texttt{arc} list by the raw contact order in the corresponding branch. There are no separate Mathematica notebook outputs for these auxiliary models; they are reconstructed and certified directly by the Python SOCP and interval-validation scripts.

\begin{table}[p]
\centering
\caption{Additional raw-order SOCP certificate groups. The lower endpoints are the weakest outward-rounded coarse interval lower bounds within each group, printed truncated; full per-model intervals are stored in the corresponding raw-order coarse interval log.}\label{tab:raw-order-cert}
\fontsize{7}{7.5}\selectfont
\setlength{\tabcolsep}{2pt}
\renewcommand{\arraystretch}{1.04}
\begin{tabularx}{\textwidth}{@{}>{\raggedright\arraybackslash}p{2.35cm}>{\raggedright\arraybackslash}X>{\centering\arraybackslash}p{0.75cm}>{\raggedright\arraybackslash}p{2.05cm}>{\raggedright\arraybackslash}p{2.20cm}@{}}
\toprule
\textbf{Family} & \textbf{Raw orders certified} & \textbf{No.} & \textbf{Weakest lower endpoint} & \textbf{Notes} \\
\midrule
Case~2.1 & Right-tail rows C2-3--C2-6:
$P_4P_8P_7$, $P_7P_4P_8$, $P_7P_8P_4$, $P_8P_4P_7$. & 4 & $1.0132634481$ & Same rows as 2.1a--b, with raw right-tail \texttt{arc} lists. \\
Case~2.2, representative left order & Right-tail rows C2-3--C2-6 with left order $P_{-7}P_{-6}P_{-3}$. & 4 & $1.0057165531$ & Same \texttt{condV} row as 2.2a--b. \\
Case~2.2, reflected raw left order & Left order $P_{-6}P_{-7}P_{-3}$ paired with all six right-tail rows C2-1--C2-6. & 6 & $1.0057165542$ & Certifies the reflected left-tail alternative directly. \\
Case~3 exhaustive all-order family & Every nonrepresentative pairing in the $24\times24$ grid of left and right orders of $P_4,P_5,P_7,P_8$ and reflected labels. & 573 & $1.0264081173$ & Together with 3a--3c, certifies all $576$ Case~3 pairings. \\
\midrule
\textbf{Raw-order minimum} & & \textbf{587} & $\mathbf{1.0057165531}$ & All raw-order lower endpoints are $>1.0048$. \\
\bottomrule
\end{tabularx}
\end{table}

The raw-order dual data are stored in \path{r/socp_dual_certificates_raw_order_q0997_numeric.json}. The perturbation audit is \path{r/perturbation_certificate_audit_raw_order_q0997.json}. The two refined interval checks are \path{r/coarse_interval_raw_order_q0997_refined.json} and \path{r/algebraic_interval_raw_order_q0997_refined.json}. The algebraic matrix check passes for every one of the $587$ raw-order models, using the same validation script and the refined raw-order $q=0.997$ scalar guard settings.

\subsection{Row-to-placement soundness audit}\label{app:audit}

The geometric forcing of Lemmas~\ref{lem:canonical-tests} and~\ref{lem:forced-placements} enters the certified models through individual affine rows. The following lemma states the row-level obligation that ties the two layers together. It is part of the formal proof and is machine-checked by the one-command verifier of Appendix~\ref{app:repro}.

\begin{lemma}[Row-to-placement soundness]\label{lem:row-audit}
For each of the $599$ certified closed models and each constraint row of its JSON specification, the audit file \path{r/forced_placement_audit.csv} contains exactly one record
\[
\begin{gathered}
(\textit{source},\ \textit{model id},\ \textit{constraint index},\ \textit{predicate},\\
\textit{placement anchor},\ \textit{witness data},\ \textit{direction data}),
\end{gathered}
\]
and the affine rows reconstructed from this record alone---the side-incidence equalities $y_i-y_A=\tan(\mathrm{dir})\,(x_i-x_A)$, the escape rows $d(\overrightarrow{AX},\theta_{\rm short})-b\ge0$, and the recorded delimiter guards---have the same coefficient vectors and constant terms as the rows that the certificate pipeline builds from the JSON constraint. Consequently, every escape or guard inequality in every certified model is the affine form of a forcing certificate $(\hbox{placement},\hbox{adjacent support sides},A,X,\theta_{\rm short})$ in the sense of Lemma~\ref{lem:forced-placements}; the single \texttt{below\_line} record (representative model 1.1b, constraint~7, discussed in Section~\ref{sec:triangle}) is a closed side-containment row rather than an escape row and is audited under the same schema.
\end{lemma}

\begin{proof}
The obligation is finite and is verified mechanically. The script \path{scripts/check_forced_placement_audit.py} (i) regenerates the audit tuples from the model files \path{data/wetzel_socp_models.json} and \path{data/wetzel_socp_raw_order_models.json}; (ii) checks that the shipped CSV is in bijection with the constraint rows---$4735$ records in total, $77$ representative and $4658$ raw-order, exactly one per pair (model id, constraint index)---and that every stored field agrees with the regenerated tuple; and (iii) independently reconstructs each incidence, escape, and guard row from the recorded tuple, using only the anchor coordinates of Appendix~\ref{app:triangle-normalization} and the directions in the tuple, and compares the resulting coefficient vectors and constant terms with the rows used by the certificate pipeline. The one-command verifier runs this check and fails unless all rows match. The geometric content---that the branch hypotheses force each placement and hence each escape row---is Lemma~\ref{lem:forced-placements}.
\end{proof}

\subsection{Reproducibility}\label{app:repro}

The supplementary computation package (Online Resource~1), mirrored at \url{https://github.com/chatchawanpan-dev/wetzel-triangle-computation}, contains the original Mathematica notebook, a rendered notebook PDF, the closed SOCP reconstructions in JSON form, the Python validation scripts, and the certificate logs used in this appendix. We emphasize the distinction between \emph{certificate validation} and \emph{solver regeneration}: the proof rests on the stored dual-certificate vectors and interval reports, validated by the scripts below; rerunning the SOCP solvers or the Mathematica notebook reproduces the numerical models but is optional and is not part of the proof. The files are listed in the package manifest. To reproduce the notebook values reported in Table~\ref{tab:cert-bounds}:

\begin{enumerate}
  \item Open \texttt{WetzelTriangle\_OptimizationCode.nb} and evaluate the initialization cells defining \texttt{d2}, \texttt{ds}, \texttt{u[\(\theta\)]}, \texttt{dd}, the contact-point coordinates \texttt{p[i]}, and the incidence and escape predicates.
  \item Evaluate each labeled case cell: Subcases~1.1a, 1.1b, 1.2, 1.3, 2.1a, 2.1b, 2.2a, 2.2b, and Cases~3a, 3b, 3c.
  \item Record the value of \texttt{nm[[1]]} in each case cell; these are the outputs reported in Table~\ref{tab:cert-bounds}.
  \item Confirm that the smallest reported output is the value for Subcase~1.1a, namely $1.0048313701863976\ldots > 1$.
\end{enumerate}

The direct both-late model 1.1c is not an original Mathematica notebook cell; it is an additional SOCP model stored in the representative model JSON and certified by the same Python interval pipeline.

The independent interval certificate is reproduced from the supplementary \texttt{nv} folder. The file \texttt{data/wetzel\_socp\_models.json} stores the closed SOCP reconstruction of the representative models, including 1.1c. The refined representative dual data are in \texttt{r/socp\_dual\_certificates\_q0999999\_numeric.json}. Running
\begingroup\fontsize{6.5}{7.2}\selectfont
\[
\begin{gathered}
\texttt{python scripts/coarse\_interval\_certificate\_check.py}\\
\texttt{\quad --audit r/perturbation\_certificate\_audit\_q0999999.json}\\
\texttt{\quad --guard 1e-7 --residual-bound 1e-7 --sigma-floor 0.25}\\
\texttt{\quad --input-radius 1e-12}\\
\texttt{\quad --json-out r/coarse\_interval\_certificate\_check\_q0999999\_refined.json},\\
\texttt{python scripts/algebraic\_interval\_certificate\_check.py}\\
\texttt{\quad --cert r/socp\_dual\_certificates\_q0999999\_numeric.json}\\
\texttt{\quad --audit r/perturbation\_certificate\_audit\_q0999999.json}\\
\texttt{\quad --residual-bound 1e-7 --sigma-floor 0.25}\\
\texttt{\quad --json-out r/algebraic\_interval\_certificate\_check\_q0999999\_refined.json}
\end{gathered}
\]
\endgroup
verifies the scalar interval lower endpoints and algebraic matrix checks in Table~\ref{tab:interval-cert}. The summary log \texttt{r/representative\_refined\_interval\_endpoints.json} records the displayed endpoints and safety margins, and \texttt{r/representative\_refinement\_tiers.json} records the $q=0.9999$, $q=0.99999$, and $q=0.999999$ comparison runs.

The raw-order certificate in Table~\ref{tab:raw-order-cert} is reproduced by running the same scripts with the raw-order inputs:
\begingroup\fontsize{6.5}{7.2}\selectfont
\[
\begin{gathered}
\texttt{python scripts/coarse\_interval\_certificate\_check.py}\\
\texttt{\quad --audit r/perturbation\_certificate\_audit\_raw\_order\_q0997.json}\\
\texttt{\quad --guard 1e-5 --residual-bound 1e-8 --sigma-floor 0.25}\\
\texttt{\quad --input-radius 1e-12}\\
\texttt{\quad --json-out r/coarse\_interval\_raw\_order\_q0997\_refined.json},\\
\texttt{python scripts/algebraic\_interval\_certificate\_check.py}\\
\texttt{\quad --spec data/wetzel\_socp\_raw\_order\_models.json}\\
\texttt{\quad --cert r/socp\_dual\_certificates\_raw\_order\_q0997\_numeric.json}\\
\texttt{\quad --audit r/perturbation\_certificate\_audit\_raw\_order\_q0997.json}\\
\texttt{\quad --residual-bound 1e-8 --sigma-floor 0.25}\\
\texttt{\quad --json-out r/algebraic\_interval\_raw\_order\_q0997\_refined.json}.
\end{gathered}
\]
\endgroup

The one-command verifier
\[
\texttt{python scripts/verify\_all\_certificates.py}
\]
checks the representative and raw-order model ledgers, verifies that all stored coarse and algebraic interval reports pass, validates the row-to-placement soundness of Lemma~\ref{lem:row-audit} against both model files, and prints the expected counts $12/12$, $587/587$, $599/599$, and the $4735$ verified audit records. It validates stored certificates and interval reports only; it does not rerun any solver. The script \texttt{scripts/generate\_raw\_order\_models.py} regenerates the $587$ auxiliary raw-order models from the representative model file; \texttt{scripts/check\_chain\_orders.py} verifies the Case~2 and Case~3 order ledger; \texttt{scripts/emit\_forced\_placement\_audit.py} writes \texttt{r/forced\_placement\_audit.csv}; and the standalone command \texttt{python scripts/check\_forced\_placement\_audit.py} re-derives and verifies that audit file (Lemma~\ref{lem:row-audit}) outside the one-command verifier if desired. The supplementary file \texttt{SHA256SUMS.txt} stores the full SHA-256 manifest with Unix line endings, so that \texttt{sha256sum -c SHA256SUMS.txt} verifies the package directly. Table~\ref{tab:hash-manifest} displays hash prefixes regenerated from the shipped manifest at submission time; the supplementary file contains the full SHA-256 digest for every tracked file.

\begin{table}[p]
\centering
\caption{Selected SHA-256 manifest entries for the supplementary certificate package. Full hashes are stored in \texttt{SHA256SUMS.txt}.}\label{tab:hash-manifest}
\scriptsize
\setlength{\tabcolsep}{3pt}
\renewcommand{\arraystretch}{1.08}
\begin{tabularx}{\textwidth}{@{}>{\raggedright\arraybackslash}p{6.1cm}>{\raggedright\arraybackslash}p{2.5cm}>{\raggedright\arraybackslash}X@{}}
\toprule
\textbf{File} & \textbf{SHA-256 prefix} & \textbf{Role} \\
\midrule
\path{nv/data/wetzel_socp_models.json} & \texttt{9b147cf6d32a} & Twelve representative closed SOCP models. \\
\path{nv/data/wetzel_socp_raw_order_models.json} & \texttt{24d3becfdb62} & $587$ auxiliary raw-order models. \\
\path{nv/r/coarse_interval_certificate_check_q0999999_refined.json} & \texttt{950b1e8e5c42} & Refined representative coarse interval endpoints. \\
\path{nv/r/algebraic_interval_certificate_check_q0999999_refined.json} & \texttt{e3a3ca65bd91} & Refined representative algebraic matrix checks. \\
\path{nv/r/representative_refined_interval_endpoints.json} & \texttt{84e17fad9349} & Refined representative endpoint summary. \\
\path{nv/r/representative_refinement_tiers.json} & \texttt{eef530b8149f} & q-radius refinement comparison. \\
\path{nv/r/coarse_interval_raw_order_q0997_refined.json} & \texttt{9c0cfb1ed105} & Refined raw-order coarse interval endpoints. \\
\path{nv/r/algebraic_interval_raw_order_q0997_refined.json} & \texttt{dd30e2e83d24} & Refined raw-order algebraic matrix checks. \\
\path{nv/r/forced_placement_audit.csv} & \texttt{764988b98afb} & Predicate-level forced-placement audit. \\
\path{nv/scripts/check_forced_placement_audit.py} & \texttt{ea2cd15398b6} & Row-to-placement soundness check (Lemma~\ref{lem:row-audit}). \\
\path{nv/scripts/verify_all_certificates.py} & \texttt{f8f445bfcd9b} & One-command verifier, including the forced-placement audit check. \\
\path{SHA256SUMS.txt} & \texttt{581755717697} & Full package hash manifest (Unix line endings). \\
\bottomrule
\end{tabularx}
\end{table}

\section{Proofs of auxiliary lemmas}\label{app:lemmas}

This appendix supplies the proofs of Lemmas~\ref{lem:reduction}, \ref{lem:closed-contact-order}, \ref{lem:closed-order-selection}, and~\ref{lem:reduced-obstruction}, gives a self-contained proof of the imported support-contact unimodality (Theorem~\ref{thm:imported-support-contact}) together with the interleaving consequences used for Proposition~\ref{prop:support-input} and Lemmas~\ref{lem:middle-skeleton} and~\ref{lem:contact-alternatives}, and expands the arguments for Lemmas~\ref{lem:order} and~\ref{lem:wedge}.

\subsection{Complete branch grammar and raw-order certification}\label{app:branch-grammar}

The finite enumeration in Proposition~\ref{prop:case-tree} is generated from the one-sided alternatives in Table~\ref{tab:one-sided-grammar}. The excluded row is not a numerical case: it is ruled out by Lemma~\ref{lem:contact-alternatives} and the closed finite contact-order principle.

\begin{table}[t]
\centering
\caption{One-sided support-role grammar for the $150^\circ$ and $120^\circ$ contacts.}\label{tab:one-sided-grammar}
\scriptsize
\setlength{\tabcolsep}{4pt}
\renewcommand{\arraystretch}{1.1}
\begin{tabularx}{\textwidth}{@{}llllX@{}}
\toprule
\textbf{Side} & \textbf{Type} & \textbf{$150^\circ$ role} & \textbf{$120^\circ$ role} & \textbf{Status} \\
\midrule
positive & $A$ & $P_1$ & $P_2$ & kept; both contacts are inner. \\
positive & $B$ & $P_1$ & $P_7$ & kept; $150^\circ$ is inner and $120^\circ$ is outer. \\
positive & $C$ & $P_8$ & $P_7$ & kept; both contacts are outer. \\
positive & -- & $P_8$ & $P_2$ & excluded; an inner $120^\circ$ contact forces the $150^\circ$ contact to be inner. \\
negative & $A'$ & $P_{-1}$ & $P_{-2}$ & reflected copy of $A$. \\
negative & $B'$ & $P_{-1}$ & $P_{-7}$ & reflected copy of $B$. \\
negative & $C'$ & $P_{-8}$ & $P_{-7}$ & reflected copy of $C$. \\
negative & -- & $P_{-8}$ & $P_{-2}$ & excluded by reflection. \\
\bottomrule
\end{tabularx}
\end{table}

\begin{table}[t]
\centering
\caption{Certified branch families after the one-sided grammar.}\label{tab:raw-grammar}
\scriptsize
\setlength{\tabcolsep}{4pt}
\renewcommand{\arraystretch}{1.1}
\begin{tabularx}{\textwidth}{@{}p{2.8cm}p{2.0cm}p{2.2cm}X@{}}
\toprule
\textbf{Raw role combination} & \textbf{Branch family} & \textbf{Certified models} & \textbf{Treatment} \\
\midrule
$A'A$ & Case~1.1 & 3 & Split by the two one-sided $T_W$ alternatives and the direct both-late model in Lemma~\ref{lem:subcase11-tail-reduction}. \\
$A'B$ and $B'A$ & Case~1.2 & 1 & Reflected copies; one representative model is kept. \\
$B'B$ & Case~1.3 & 1 & Already fixed after the one-sided grammar. \\
$A'C$ and $C'A$ & Case~2.1 & 6 & Reflection fixes one outer side; all six right-tail orders of Lemma~\ref{lem:case2-order-reduction} are certified. \\
$B'C$ and $C'B$ & Case~2.2 & 12 & The two reflected left-tail orders are paired with all six right-tail orders. \\
$C'C$ & Case~3 & 576 & All $24$ right-tail orders are paired with all $24$ reflected left-tail orders. \\
\midrule
\textbf{Total} & & \textbf{599} & Twelve representative models plus $587$ raw-order models. \\
\bottomrule
\end{tabularx}
\end{table}

Table~\ref{tab:raw-certificate-routing} records where the raw tail orders are certified. A row marked ``representative'' is certified by the corresponding closed SOCP model in Table~\ref{tab:interval-cert}; a row marked ``raw'' is certified by the corresponding model in the raw-order JSON file and by Table~\ref{tab:raw-order-cert}. No row in this table is justified by replacing it with a different tail order.

\begin{table}[p]
\centering
\caption{Certificate routing for raw tail-order rows.}\label{tab:raw-certificate-routing}
\scriptsize
\setlength{\tabcolsep}{3pt}
\renewcommand{\arraystretch}{1.08}
\begin{tabularx}{\textwidth}{@{}p{1.25cm}p{3.0cm}X@{}}
\toprule
\textbf{Row} & \textbf{Raw order} & \textbf{Certificate source} \\
\midrule
C2-1 & $P_3P_4P_7P_8$ & Representative models 2.1a and 2.2a; with reflected Subcase~2.2 left order, raw model \texttt{2.2-leftalt-C2-1}. \\
C2-2 & $P_3P_8P_7P_4$ & Representative models 2.1b and 2.2b; with reflected Subcase~2.2 left order, raw model \texttt{2.2-leftalt-C2-2}. \\
C2-3 & $P_3P_4P_8P_7$ & Raw models \texttt{2.1-C2-3}, \texttt{2.2-C2-3}, and \texttt{2.2-leftalt-C2-3}. \\
C2-4 & $P_3P_7P_4P_8$ & Raw models \texttt{2.1-C2-4}, \texttt{2.2-C2-4}, and \texttt{2.2-leftalt-C2-4}. \\
C2-5 & $P_3P_7P_8P_4$ & Raw models \texttt{2.1-C2-5}, \texttt{2.2-C2-5}, and \texttt{2.2-leftalt-C2-5}. \\
C2-6 & $P_3P_8P_4P_7$ & Raw models \texttt{2.1-C2-6}, \texttt{2.2-C2-6}, and \texttt{2.2-leftalt-C2-6}. \\
\midrule
Case~3 & $R_1,\ldots,R_{24}$ are the $24$ strict orders of $P_4,P_5,P_7,P_8$; $L_1,\ldots,L_{24}$ are their reflected left-tail orders. & The pairings $(L_1,R_1)$, $(L_1,R_{24})$, and $(L_{24},R_{24})$ are representative models 3a--3c. The other $573$ pairings in $\{L_1,\ldots,L_{24}\}\times\{R_1,\ldots,R_{24}\}$ are the Case~3 raw-order models in the raw-order certificate logs. \\
\bottomrule
\end{tabularx}
\end{table}

\begin{table}[p]
\centering
\caption{Branch-to-model ledger. Concrete model ids are the keys in the model JSON files \texttt{wetzel\_socp\_models.json} and \texttt{wetzel\_socp\_raw\_order\_models.json}; the listed log files certify exactly those ids.}\label{tab:branch-model-ledger}
\fontsize{7}{7.4}\selectfont
\setlength{\tabcolsep}{2pt}
\renewcommand{\arraystretch}{1.08}
\begin{tabularx}{\textwidth}{@{}p{1.55cm}p{2.55cm}p{2.7cm}X p{2.7cm}@{}}
\toprule
\textbf{Branch} & \textbf{Model ids} & \textbf{Chain source} & \textbf{Active predicate pattern} & \textbf{Certificate log key} \\
\midrule
1.1 & \texttt{1.1a}, \texttt{1.1b}, \texttt{1.1c} & Table~\ref{tab:cert-bounds} and Table~\ref{tab:wedge-audit} & \texttt{condL}, \texttt{condR}, \texttt{condSfarL}, \texttt{condTfarR}, wedge incidences, and recorded wedge escapes or \texttt{below\_line}. & Refined representative logs \path{coarse_interval_certificate_check_q0999999_refined.json} and \path{algebraic_interval_certificate_check_q0999999_refined.json}. \\
1.2--1.3 & \texttt{1.2}, \texttt{1.3} & Table~\ref{tab:cert-bounds} & floor incidences \texttt{onlineL}, \texttt{onlineR} and escapes \texttt{farL}, \texttt{farR}. & Refined representative $q=0.999999$ logs. \\
2.1 & \texttt{2.1a}, \texttt{2.1b}; \texttt{2.1-C2-3}--\texttt{2.1-C2-6} & C2 rows in Table~\ref{tab:raw-certificate-routing} & \texttt{onlineL}, \texttt{onlineR}, \texttt{farL}, \texttt{farR}, \texttt{onlineV}, \texttt{farV}. & Refined representative $q=0.999999$ logs for C2-1--C2-2; refined raw-order $q=0.997$ logs for C2-3--C2-6. \\
2.2 & \texttt{2.2a}, \texttt{2.2b}; \texttt{2.2-C2-3}--\texttt{2.2-C2-6}; \texttt{2.2-leftalt-C2-1}--\texttt{2.2-leftalt-C2-6} & C2 rows plus the two left-tail orders in Proposition~\ref{prop:case-tree} & \texttt{onlineL}, \texttt{onlineR}, \texttt{farL}, \texttt{farR}, and \texttt{condV}. & Refined representative $q=0.999999$ logs for displayed rows; refined raw-order $q=0.997$ logs for the other ten ids. \\
3 & \texttt{3a}, \texttt{3b}, \texttt{3c}; \texttt{3all-L\#\#-R\#\#} for the other $573$ pairs & $24\times24$ order grid in Table~\ref{tab:raw-certificate-routing} & \texttt{onlineL}, \texttt{onlineR}, \texttt{farL}, \texttt{farR}, \texttt{onlineU}, \texttt{onlineV}, \texttt{farU}, \texttt{farV}. & Refined representative $q=0.999999$ logs for \texttt{3a--3c}; refined raw-order $q=0.997$ logs for all other pairings. \\
\bottomrule
\end{tabularx}
\end{table}

The Subcase~1.1 $T_W$ routing is summarized in Table~\ref{tab:wedge-audit}. Degenerate-kink deletion is used only in rows where the nearest vertical delimiter itself satisfies the closed escape row; the both-late row is certified directly by the additional model 1.1c. In Cases~2 and~3, the nonrepresentative raw orders are not routed into representative orders by endpoint replacement; they are certified directly by the additional SOCP models in Table~\ref{tab:raw-order-cert}. Thus the proof of Proposition~\ref{prop:case-tree} requires the Case~2 raw-order enumeration and the exhaustive Case~3 all-order enclosure in Lemmas~\ref{lem:case2-order-reduction} and~\ref{lem:case3-order-reduction}.

\begin{table}[p]
\centering
\caption{Subcase~1.1 routing for late $T_W$ tail contacts.}\label{tab:wedge-audit}
\scriptsize
\setlength{\tabcolsep}{3pt}
\renewcommand{\arraystretch}{1.10}
\begin{tabularx}{\textwidth}{@{}p{2.1cm}p{2.1cm}p{1.8cm}X@{}}
\toprule
\textbf{Left delimiter} & \textbf{Right delimiter} & \textbf{Certified model} & \textbf{Witness assignment and preserved or dropped rows} \\
\midrule
satisfies escape row & satisfies escape row & 1.1a & Both later kinks may be deleted to the nearest vertical delimiters. The model keeps the two middle side-incidence rows \texttt{onlineWl[p[-2.5]]}, \texttt{onlineWr[p[2.5]]} and the two closed escape rows at $P_{-3.5}$ and $P_{3.5}$, allowing the delimiter coincidences $P_{-3.5}=P_{-6}$ and $P_{3.5}=P_6$. \\
late contact retained & satisfies escape row & 1.1b & The left late witness is $P_{-6.5}$ and the row \texttt{onlineWl[p[-6.5]]} is kept. The right delimiter-side escape row $d(\overrightarrow{WP_{3.5}},-75^\circ)\ge b$ is kept, and the nonactive side-containment row for $P_{-2}$ is kept. The absent right late-side incidence is not imposed. \\
satisfies escape row & late contact retained & reflected 1.1b & The reflected statement keeps the right late witness $P_{6.5}$, keeps the left delimiter-side escape row, and omits the absent left late-side incidence. \\
late contact retained & late contact retained & 1.1c & Both late witnesses are kept. The direct SOCP model uses the chain through $P_{-6.5},P_{-6}$ and $P_6,P_{6.5}$, keeps \texttt{onlineWl[p[-6.5]]} and \texttt{onlineWr[p[6.5]]}, and enforces the two late escape rows $d(\overrightarrow{WP_{-6.5}},-105^\circ)\ge b$ and $d(\overrightarrow{WP_{6.5}},-75^\circ)\ge b$. No pre-delimiter opposite-side escape row is substituted. \\
\bottomrule
\end{tabularx}
\end{table}

As a worked illustration of the wedge routing at model-row level, consider the representative model 1.1a, whose JSON constraint list is reproduced in Table~\ref{tab:cert-bounds}. The wedge variables are the anchor $W=(x_W,y_W)$ and the contacts $P_{\pm2.5}$ and $P_{\pm3.5}$. Constraints~5 and~6 are the two side-incidence equalities
\[
y_{-2.5}-y_W=\tan75^\circ\,(x_{-2.5}-x_W),
\qquad
y_{2.5}-y_W=\tan(-75^\circ)\,(x_{2.5}-x_W),
\]
and constraints~7 and~8 are the two closed wedge escape rows
\[
\begin{gathered}
\cos75^\circ\,(x_{3.5}-x_W)-\sin75^\circ\,(y_{3.5}-y_W)\ \ge\ b,\\
-\cos75^\circ\,(x_{-3.5}-x_W)-\sin75^\circ\,(y_{-3.5}-y_W)\ \ge\ b,
\end{gathered}
\]
the affine forms of $d(\overrightarrow{WP_{3.5}},-75^\circ)\ge b$ and $d(\overrightarrow{WP_{-3.5}},-105^\circ)\ge b$. The degenerate delimiter choice of the first routing row of Table~\ref{tab:wedge-audit} is represented in this closed model by the absence of any separation constraint between $P_{3.5}$ and $P_6$ (and between $P_{-3.5}$ and $P_{-6}$): the chain order allows the coincidences $P_{3.5}=P_6$ and $P_{-3.5}=P_{-6}$, in which case the corresponding chain segments have length zero and the escape rows above are inherited verbatim by the delimiters. This is exactly the configuration produced by the kink deletion of Lemma~\ref{lem:subcase11-tail-reduction}, and it is feasible for the certified model, so the certified lower endpoint $1.0048290490\ldots$ of Table~\ref{tab:interval-cert} covers the degenerate branch as well.

\subsection{Compactness and reduction proofs}\label{app:compactness-proofs}

\begin{proof}[Proof of Lemma~\ref{lem:reduction}]
Suppose that $C$ fails to cover a unit arc. By the Wetzel--Wichiramala polygonal reduction~\cite[Cor.~5]{wetzel2010}, $C$ fails to cover some simple polygonal unit arc $\alpha$. If a compact convex set $K$ satisfies $\conv(\alpha)\subseteq K$ and a congruent copy of $C$ contains $K$, then the same copy contains $\alpha$, a contradiction. Thus every enlarged hull considered below remains a noncovered obstruction hull.

Take $\beta_m=\alpha$ for all $m$ for the present lemma; in the normalized reduction below we instead apply the same compactness argument after selecting a subsequence on which the finite witness-role data stabilize. Put
\[
K_m=\conv(\beta_m).
\]
Every $K_m$ has diameter at most $1$, because $\operatorname{diam}(\beta_m)\le \ell(\beta_m)\le1$. After one fixed translation, the sequence lies in a common closed ball.
By Blaschke's selection theorem, after passing to a subsequence we may assume that $K_m$ converges in the Hausdorff metric to a compact convex set $K^*$. Since $\conv(\alpha)\subseteq K_m$ for every $m$, we also have $\conv(\alpha)\subseteq K^*$. Continuity of support functions also gives convergence of the support faces in every fixed direction used later.

Choose arclength parametrizations of the polygonal arcs realizing $K_m$, extended constantly to $[0,1]$ if their length is less than one. These maps are uniformly bounded and $1$-Lipschitz. By Arzel\`a--Ascoli, a subsequence converges uniformly to a $1$-Lipschitz rectifiable curve $\gamma_0$. Uniform convergence of compact images implies Hausdorff convergence of the images, and taking convex hulls is continuous for the Hausdorff metric in the plane; hence
\[
\conv(\gamma_0)=K^*.
\]
Lower semicontinuity gives $\ell(\gamma_0)\le 1$.

The limiting curve $\gamma_0$, together with the subsequential limits of the finitely many selected support-role witnesses used in the case tree, is a weak obstruction in the sense of Definition~\ref{def:weak-obstruction}. It is noncovered by $C$: if a congruent copy of $C$ contained $K^*=\conv(\gamma_0)$, then it would contain $\conv(\alpha)$ and hence $\alpha$, contrary to the choice of $\alpha$.

This proves the weak-hull reduction without requiring a congruence-dependent support-functional maximality, loop erasure, or padding to a genuine unit arc.
\end{proof}

\begin{proof}[Proof of Lemma~\ref{lem:closed-contact-order}]
For each of the finitely many witness directions and named intervals, compactness of $[0,1]$ gives a common subsequence on which all selected parameters $t_{n,j}$ converge to limits $t_j$. Since the intervals $I_j$ are closed, each $t_j$ remains in its prescribed interval.

Hausdorff convergence of compact convex sets implies uniform convergence of support functions on $S^1$. Hence, for every fixed support normal $u_j$,
\[
h_{K_n}(u_j)\to h_K(u_j).
\]
Uniform convergence of the parametrized arcs gives $\gamma_n(t_{n,j})\to\gamma(t_j)$, so the exact support-contact identity for the witnesses passes to the limit:
\[
\langle\gamma(t_j),u_j\rangle
=
\lim_{n\to\infty}\langle\gamma_n(t_{n,j}),u_j\rangle
=
\lim_{n\to\infty}h_{K_n}(u_j)
=h_K(u_j).
\]
Thus $P_j=\gamma(t_j)$ lies on the corresponding support face of $K$.

Finally, each recorded weak order relation is a closed scalar inequality. If $t_i^{(n)}\le t_j^{(n)}$ for all $n$ in the selected subsequence, then $t_i\le t_j$. Additional support contacts may appear in the limit, but they are not part of the stored witness data and do not remove any feasible closed SOCP realization.
\end{proof}

\begin{proof}[Proof of Lemma~\ref{lem:closed-order-selection}]
Fix $K$ and one stabilized finite weak role assignment $\mathcal R$. By hypothesis there is at least one sequence of simple polygonal arcs realizing this same finite assignment and converging to $K$, so $\mathcal A(K,\mathcal R)$ is nonempty. All curves in $\mathcal A(K,\mathcal R)$ are $1$-Lipschitz after constant-speed parametrization on $[0,1]$, with constant tails if necessary, and their images lie in a common compact ball because $\operatorname{diam}K\le1$. Arzel\`a--Ascoli therefore gives compactness for uniformly convergent subsequences.

To see explicitly that the role-witnessed condition is closed, let $\eta_m\to\eta$ uniformly with $\eta_m\in\mathcal A(K,\mathcal R)$. For each $m$, choose a simple polygonal witness $\pi_{m,n(m)}$ realizing $\mathcal R$ and satisfying
\[
  \|\pi_{m,n(m)}-\eta_m\|_\infty < 1/m .
\]
Then $\pi_{m,n(m)}\to\eta$ uniformly and the diagonal witness sequence realizes the same stabilized finite role assignment $\mathcal R$. Equality $\conv(\eta)=K$ follows because $\conv(\eta_m)=K$ and convex hulls vary continuously under uniform, hence Hausdorff, convergence of the images.

The finite role data remain closed under this convergence. Equality of convex hulls is equality of support functions, and support functions are continuous under Hausdorff convergence of convex hulls. The witness contacts are selected in fixed line or sweep directions and fixed closed intervals, so Lemma~\ref{lem:closed-contact-order} passes them to contacts on the corresponding limiting exposed faces. Each weak role incidence is therefore a closed statement about a stored witness parameter lying in a fixed closed interval and on a fixed exposed face of $K$. Thus $\mathcal A(K,\mathcal R)$ is closed. Lower semicontinuity of arclength on uniformly convergent equi-Lipschitz curves gives a length-minimizing representative $\eta_*\in\mathcal A(K,\mathcal R)$. By the first bullet in the definition of $\mathcal A(K,\mathcal R)$, this minimizer itself is obtained as a uniform limit of simple polygonal arcs realizing the same stabilized role data; we apply Wichiramala's endpoint-contact theorem to that approximating sequence.

For every simple polygonal arc in the stabilized approximating sequences, Wichiramala's endpoint-contact parameter
\[
T(\theta)=\{\min C_\theta,\max C_\theta\}
\]
has monotone increasing and then monotone decreasing graph over a full period; this is Theorem~\ref{thm:imported-support-contact}, proved self-containedly in Appendix~\ref{app:unimodality} following \cite[Lemmas~1--3]{wichiramala2019a}. Together with the cyclic exposure order (Lemma~\ref{lem:cyclic-exposure}) and the interleaving consequences (Corollary~\ref{cor:role-trichotomy}), the selected contacts of the polygonal representatives occur in weak normal-fan order on each exposed boundary chain. The weak role data involve only finitely many directions and finitely many named subarcs. Applying Lemma~\ref{lem:closed-contact-order} to this finite witness list shows that the limiting contacts of the length-minimal weak representative retain the same weak normal-fan order. If two limiting contacts coincide or a whole exposed segment is present, the witness convention records exactly that weak equality.
\end{proof}

\begin{proof}[Proof of Lemma~\ref{lem:reduced-obstruction}]
Suppose that $T$ does not cover a unit arc. Apply Lemma~\ref{lem:reduction} with $C=T$. We obtain a noncovered weak obstruction $\gamma_0$ of length at most one with limiting hull $K=\conv(\gamma_0)$.

The horizontal $\Lambda$-configuration is selected on the defining simple polygonal approximants, where the simple-arc $\Lambda$-property applies. For each approximant choose one such $\Lambda$-frame and apply a rigid motion sending the selected lower support line to $y=0$, the upper support line to $y=H_n$, and the first lower contact $P_{-3}^{(n)}$ to the origin; reflect if necessary so that the second lower contact has nonnegative $x$-coordinate. These normalized approximants still have length at most one and have hull diameter at most one, so their images and normalized hulls lie in a common compact ball. The rotation angles have a convergent subsequence, the reflection choice is finite, and the translation is fixed by the anchoring $P_{-3}^{(n)}=0$ after normalization; therefore, after passing to a further subsequence, the normalizing rigid motions converge to a limiting congruence. Passing to this subsequence, the normalized curves converge uniformly and the normalized hulls converge in the Hausdorff metric to the image of the unnormalized limiting hull under that congruence; we denote this congruent limiting hull by $K=\conv(\gamma_0)$. Noncoverage by $T$ and arclength are invariant under the normalizing rigid motions and pass to the closed Hausdorff limit. No support-functional maximality is used after this normalization; only the existence of the noncovered limiting hull and its stabilized role data is needed.

Since there are only finitely many special-angle role assignments in the case tree, a further subsequence fixes the weak support-role assignment and the selected delimiter witnesses. Lemma~\ref{lem:closed-contact-order} then supplies the limiting support contacts on the normalized weak limit and preserves the finite weak order relations.

Among weak obstructions of length at most one that realize the same convex hull $K$ and the same weak support-role data, choose one of minimal length in the class $\mathcal A(K,\mathcal R)$ defined above. Existence follows from arclength parametrization, Arzel\`a--Ascoli compactness in a fixed ball, closedness of $\mathcal A(K,\mathcal R)$, and lower semicontinuity of length.

By Lemma~\ref{lem:closed-order-selection}, among the length minimizers for this hull and weak role assignment there is one whose exposed-boundary contacts are in weak normal-fan order.

The selected obstruction is therefore weak support-reduced in the sense used below. Its length is still at most one, and it is still not covered by $T$ because its convex hull is the same noncovered hull.
\end{proof}

\subsection{Support-contact unimodality and interleaving: self-contained proofs}\label{app:unimodality}

This subsection proves Theorem~\ref{thm:imported-support-contact} and the interleaving consequences used in Proposition~\ref{prop:support-input} and Lemma~\ref{lem:contact-alternatives}. The statements and the overall strategy follow Wichiramala~\cite{wichiramala2019a}; we include complete arguments so that the paper does not rely on the unpublished source. Throughout, $\gamma:[0,\ell]\to\R^2$ is a simple polygonal arc that is not a line segment, $H=\conv(\gamma)$ (a convex polygon with nonempty interior), and contacts are taken with the oriented support lines $L_\theta$ of Section~\ref{subsec:support-distance}.

We use two elementary facts. The first is the standard normal-fan statement that exposure directions order contacts cyclically along $\partial H$.

\begin{lemma}[Cyclic exposure order]\label{lem:cyclic-exposure}
Let $X_1,X_2,X_3,X_4\in\partial H$ be four distinct points, exposed by support lines at directions $\theta_1<\theta_2<\theta_3<\theta_4<\theta_1+2\pi$ respectively (that is, $X_i\in L_{\theta_i}$). Then $X_1,X_2,X_3,X_4$ occur in this cyclic order along $\partial H$, with the counterclockwise orientation matching increasing direction.
\end{lemma}

\begin{proof}
For a compact convex set, the exposed face in direction $\theta$ is the intersection of $H$ with $L_\theta$, and as $\theta$ increases by $2\pi$ the outward normal $n_\theta$ rotates once counterclockwise; the corresponding exposed faces sweep $\partial H$ once in the counterclockwise order of the normal fan. Distinct points exposed at distinct directions therefore inherit the cyclic order of their directions.
\end{proof}

\begin{lemma}[Weak cyclic exposure with a repeated support face]\label{lem:cyclic-exposure-same-face}
Let $F=H\cap L_\theta$ be an exposed face and let $X_1,X_2\in F$ be selected in their counterclockwise boundary order along $F$. If $Y_1,\ldots,Y_m$ are selected exposed contacts whose support directions follow $\theta$ in the normal-fan order and precede the next exposure of $F$, then the boundary order obtained in Lemma~\ref{lem:cyclic-exposure} extends weakly as
\[
X_1,\ X_2,\ Y_1,\ldots,Y_m .
\]
Equivalently, the same order is obtained by perturbing the common support direction within the normal cone of $F$ and then taking the limit. In the normalized $\Lambda$-position this applies to the two lower contacts $P_{-3},P_3\in L_0$, ordered by $x(P_{-3})\le x(P_3)$.
\end{lemma}

\begin{proof}
Along an exposed segment of a convex polygon, the counterclockwise boundary order is the usual order on that segment. Perturbing the support direction slightly to the two sides of the normal cone of $F$ exposes the endpoints of $F$ in this same boundary order, while all contacts exposed by directions outside the normal cone retain their normal-fan order. Letting the perturbation tend to the original direction gives the stated weak order; if $X_1=X_2$ or $F$ degenerates to a point, the assertion is just the corresponding weak equality.
\end{proof}

The second fact is the planarity obstruction that powers all order and exclusion statements below.

\begin{lemma}[Interleaving exclusion]\label{lem:interleaving}
Let $I=[a_1,a_2]$ and $J=[b_1,b_2]$ be parameter intervals with disjoint interiors and $\{a_1,a_2\}\cap\{b_1,b_2\}=\emptyset$, and suppose the four points
\[
\gamma(a_1),\ \gamma(a_2),\ \gamma(b_1),\ \gamma(b_2)\in\partial H
\]
are distinct and the pair $\{\gamma(a_1),\gamma(a_2)\}$ separates the pair $\{\gamma(b_1),\gamma(b_2)\}$ on $\partial H$: each open boundary arc determined by $\gamma(a_1),\gamma(a_2)$ contains exactly one of $\gamma(b_1),\gamma(b_2)$. Then the subarcs $\gamma|_I$ and $\gamma|_J$ intersect. In particular, since $\gamma$ is simple and the parameter intervals are disjoint, this configuration cannot occur.
\end{lemma}

\begin{proof}
Write $\alpha=\gamma|_I$ and $\beta=\gamma|_J$ and suppose, for contradiction, that $\alpha\cap\beta=\emptyset$; by simplicity this holds automatically, so what we must rule out is the boundary-separation hypothesis itself. Let $\Gamma$ and $\Gamma'$ be the two closed boundary arcs of $\partial H$ with endpoints $\gamma(a_1),\gamma(a_2)$, labeled so that $\gamma(b_2)\in\operatorname{int}\Gamma$ and $\gamma(b_1)\in\operatorname{int}\Gamma'$.

Fix an interior point $O$ of $H$ and let $\delta_\lambda(X)=O+\lambda(X-O)$ for $\lambda>1$. Define the outer detour
\[
\tau=[\gamma(a_1),\delta_\lambda(\gamma(a_1))]\cup\delta_\lambda(\Gamma)\cup[\delta_\lambda(\gamma(a_2)),\gamma(a_2)],
\]
the dilated copy of $\Gamma$ joined to its base points by radial segments. By convexity, $\tau$ meets $H$ exactly in the two points $\gamma(a_1),\gamma(a_2)$, and $\tau$ is a simple arc. Hence $C=\alpha\cup\tau$ is a Jordan curve: $\alpha$ is simple because $\gamma$ is, and $\alpha\cap\tau$ consists exactly of the two common endpoints.

The ray from $O$ through $\gamma(b_1)$, beyond $\gamma(b_1)$, leaves $H$ immediately and meets neither $\delta_\lambda(\Gamma)$ (it is the dilation ray of a point of $\operatorname{int}\Gamma'$) nor the two radial connector segments (distinct rays from $O$) nor $\alpha\subset H$. Hence $\gamma(b_1)$ lies in the unbounded component of $\R^2\setminus C$. The analogous ray beyond $\gamma(b_2)$ crosses $C$ exactly once, transversally in $\delta_\lambda(\Gamma)$ at the dilation image of $\gamma(b_2)$, so $\gamma(b_2)$ lies in the bounded component. Thus $\gamma(b_1)$ and $\gamma(b_2)$ lie in different components of $\R^2\setminus C$.

The subarc $\beta$ is a connected set joining them. By construction $\beta\cap\tau=\emptyset$ (points of $\beta$ lie in $H$ and differ from $\gamma(a_1),\gamma(a_2)$), so $\beta$ must meet $\alpha$. This contradicts the simplicity of $\gamma$ on disjoint parameter intervals and proves the lemma.
\end{proof}

\begin{proof}[Proof of Theorem~\ref{thm:imported-support-contact}]
Since $\gamma$ is a simple polygonal arc that is not contained in a line, $H=\conv(\gamma)$ is a convex polygon with nonempty interior; this is used in part (2).

(1) For all but finitely many directions $\theta$, the exposed face $H\cap L_\theta$ is a single vertex $v$ of the polygon $H$; since $\gamma$ is simple, the position $v$ corresponds to a unique parameter $t_v$, so $C_\theta=\{t_v\}$ and $T(\theta)=\{t_v\}$ on the whole open arc of directions exposing $v$. The finitely many exceptional directions expose edges of $H$; there $T(\theta)$ records the first and last contact parameters on the edge. Collinear subarcs of $\gamma$ lying on an edge of $H$, and repeated visits to the same exposed face, are handled by exactly this first/last-contact convention: only the two endpoint parameters of $C_\theta$ enter $T(\theta)$. Hence the graph of $T$ is a step function whose plateau levels are parameters of hull-vertex positions, with two-valued jumps at the edge directions.

(2) Fix $s$ with $\gamma(s)\in\partial H$. The set of directions exposing the point $\gamma(s)$ is the closed arc of outward normals in the normal cone at $\gamma(s)$, which is a single direction if $\gamma(s)$ is interior to an edge and an arc of length $\pi-\iota$ if $\gamma(s)$ is a vertex with interior angle $\iota>0$. Since $H$ has nonempty interior, $\iota>0$, so the arc has length less than $\pi$.

(3) Let $\hat t=\min\{t_v: v\ \text{vertex of}\ H\}$ be the smallest vertex parameter and fix a direction $\hat\theta$ in the interior of the exposure arc of the corresponding vertex $\hat v$, so $T(\hat\theta)=\{\hat t\}$ and $\hat t$ is the global minimum level of $T$. Suppose, for contradiction, that some selection from $T$ fails to be nondecreasing-then-nonincreasing over $[\hat\theta,\hat\theta+2\pi]$. By the step structure (1), there are then directions $\hat\theta<\theta_1<\theta_2<\theta_3<\hat\theta+2\pi$, each in the interior of a vertex plateau, with single values $s_i=T(\theta_i)$ satisfying
\[
s_1>s_2<s_3 .
\]
First, $s_2>\hat t$: otherwise $s_2=\hat t$, the $\theta_2$-vertex coincides with $\hat v$ (equal levels of a simple arc are equal parameters, hence equal positions), and by (2) the whole interval $[\hat\theta,\theta_2]$ exposes $\hat v$, forcing $s_1=\hat t<s_1$, a contradiction. Second, $s_1\ne s_3$: otherwise the $\theta_1$- and $\theta_3$-vertices coincide and (2) makes $T$ constant on $[\theta_1,\theta_3]\ni\theta_2$, contradicting $s_2<s_1$.

The four contact points $\hat v=\gamma(\hat t)$, $\gamma(s_1)$, $\gamma(s_2)$, $\gamma(s_3)$ are distinct (their parameters $\hat t<s_2<\min(s_1,s_3)$, $s_1\neq s_3$ are distinct) and exposed at the directions $\hat\theta<\theta_1<\theta_2<\theta_3$, so by Lemma~\ref{lem:cyclic-exposure} they occur in this cyclic order along $\partial H$. Consider the disjoint parameter intervals
\[
I=[\hat t,\;s_2],
\qquad
J=[\min(s_1,s_3),\;\max(s_1,s_3)] .
\]
The endpoints of $\gamma|_I$ are $\hat v$ and $\gamma(s_2)$, which by the cyclic order separate $\gamma(s_1)$ from $\gamma(s_3)$, the endpoints of $\gamma|_J$, on $\partial H$. Lemma~\ref{lem:interleaving} forbids this configuration. Hence no zigzag exists and every selection from $T$ is unimodal over the period, proving (3).
\end{proof}

The interleaving consequences in the normalized $\Lambda$-position follow. Recall the closed witness convention: $P_{-3}=\gamma(t_{-3})$ and $P_3=\gamma(t_3)$ are the first and last contacts on the lower horizontal support $L_0$, $P_0=\gamma(t_0)$ is a top contact with $t_{-3}<t_0<t_3$, and the frame is reflected if necessary so that $x(P_{-3})\le x(P_3)$.

Table~\ref{tab:cyclic-order} records the cyclic normal-fan order of the exposed faces used below, in the convention $n_\theta=(\sin\theta,-\cos\theta)$ of Section~\ref{subsec:support-distance}. The counterclockwise boundary order of the listed contacts along $\partial H$ follows the listed sweep order of directions by Lemmas~\ref{lem:cyclic-exposure} and~\ref{lem:cyclic-exposure-same-face}; in particular the $120^\circ$ and $150^\circ$ witnesses lie on the right upper boundary chain, between the lower face and the top face, in exactly this order.

\begin{table}[t]
\centering
\caption{Cyclic normal-fan order of the exposed faces used in Corollary~\ref{cor:role-trichotomy}. The sweep is counterclockwise (increasing direction $\theta$), anchored at the lower face; positions~1--4 are the positive side and positions~4--6 its reflected analogue.}\label{tab:cyclic-order}
\scriptsize
\setlength{\tabcolsep}{3pt}
\renewcommand{\arraystretch}{1.14}
\begin{tabularx}{\textwidth}{@{}cll>{\raggedright\arraybackslash}X@{}}
\toprule
\textbf{Sweep pos.} & \textbf{Line direction $\theta$} & \textbf{Support normal $n_\theta$} & \textbf{Exposed face / selected witness} \\
\midrule
1 & $0^\circ$ & $(0,-1)$ & lower face: $P_{-3}$, then $P_3$, ordered by $x(P_{-3})\le x(P_3)$ \\
2 & $120^\circ$ & $(\sqrt3/2,\,1/2)$ & $\gamma(w_{120^\circ})$, selected $L_{120^\circ}$ witness, right upper chain \\
3 & $150^\circ$ & $(1/2,\,\sqrt3/2)$ & $\gamma(w_{150^\circ})$, selected $L_{150^\circ}$ witness, right upper chain \\
4 & $180^\circ$ & $(0,\,1)$ & top face: $P_0$ \\
5 & $-150^\circ\;(\equiv210^\circ)$ & $(-1/2,\,\sqrt3/2)$ & $\gamma(w_{-150^\circ})$, left upper chain \\
6 & $-120^\circ\;(\equiv240^\circ)$ & $(-\sqrt3/2,\,1/2)$ & $\gamma(w_{-120^\circ})$, left upper chain \\
\bottomrule
\end{tabularx}
\end{table}

Thus the counterclockwise boundary order of the contacts used below is
\[
P_{-3},\;P_3,\;\gamma(w_{120^\circ}),\;\gamma(w_{150^\circ}),\;P_0,\;\gamma(w_{-150^\circ}),\;\gamma(w_{-120^\circ}),
\]
returning to $P_{-3}$. Intermediate special directions ($30^\circ$, $60^\circ$, $90^\circ$ and their reflections) expose contacts---for example $P_6$ between $P_3$ and $\gamma(w_{120^\circ})$---in the same sweep order; they are not needed in Corollary~\ref{cor:role-trichotomy}. The passage from strict vertex contacts to exposed support segments and weak equalities is the closed witness convention: if an exposed face is a segment, Lemma~\ref{lem:cyclic-exposure-same-face} orders its selected witnesses internally by the boundary order of the face, and weakly against all other listed contacts, by perturbing the support direction within the normal cone of the face and passing to the limit; if two selected witnesses coincide, every strict separation statement involving them degenerates to the corresponding weak inequality, which is exactly what the closed SOCP rows record. Table~\ref{tab:cyclic-order} therefore remains valid verbatim with ``weak order'' in place of ``order'' whenever a coincidence or an exposed segment occurs.

\begin{corollary}[Role trichotomy, inner order, and mixed-pair exclusion]\label{cor:role-trichotomy}
Let $\gamma$ be a simple polygonal arc in the normalized $\Lambda$-position, and for $\theta\in\{150^\circ,120^\circ\}$ let $w_\theta$ denote the parameter of a selected $L_\theta$ witness. Then, with all inequalities weak and coincidences treated as weak-order degeneracies:
\begin{enumerate}
\item[(i)] $w_\theta\ge t_0$; hence each positive-side witness lies on the middle chain $\gamma_{P_0P_3}$ (inner role) or on the right tail $\gamma_{P_3Q}$ (outer role);
\item[(ii)] if both witnesses are inner, then $w_{150^\circ}\le w_{120^\circ}$;
\item[(iii)] the strictly mixed pair is impossible: if $w_{150^\circ}>t_3$ then $w_{120^\circ}\ge t_3$.
\end{enumerate}
The reflected statements hold for $\theta\in\{-150^\circ,-120^\circ\}$ on the negative side, with $w_\theta\le t_0$ and the left tail in place of the right tail.
\end{corollary}

\begin{proof}
All three parts instantiate Lemma~\ref{lem:interleaving}; in each case the four contact points are distinct in the strict configuration to be excluded, and boundary positions are ordered by Lemmas~\ref{lem:cyclic-exposure} and~\ref{lem:cyclic-exposure-same-face} via their exposure directions, with the lower face exposed at $0^\circ$, the top face at $180^\circ$, and the $120^\circ$ and $150^\circ$ witnesses on the right upper boundary chain in the normal-fan order tabulated in Table~\ref{tab:cyclic-order}.

(i) Suppose first $w_\theta<t_{-3}$ strictly. Take $I=[w_\theta,t_{-3}]$ with endpoint contacts $\gamma(w_\theta)$ (right upper boundary) and $P_{-3}$ (lower face), and $J=[t_0,t_3]$ with endpoint contacts $P_0$ (top face) and $P_3$ (lower face). The intervals are disjoint, and the cyclic boundary order is $P_{-3},P_3,\gamma(w_\theta),P_0$ by Lemma~\ref{lem:cyclic-exposure-same-face} on the lower face and Lemma~\ref{lem:cyclic-exposure} for the subsequent directions $\theta\in\{120^\circ,150^\circ\}$ and $180^\circ$. The pair $\{\gamma(w_\theta),P_{-3}\}$ separates $\{P_0,P_3\}$, so Lemma~\ref{lem:interleaving} applies and excludes the configuration. Suppose next $t_{-3}<w_\theta<t_0$ strictly. Take $I=[t_{-3},w_\theta]$ with endpoints $P_{-3}$ and $\gamma(w_\theta)$, and $J=[t_0,t_3]$ as before; the same cyclic order shows the endpoint pairs interleave, again a contradiction. Boundary coincidences ($w_\theta=t_{-3}$ or $w_\theta=t_0$) are weak degeneracies. Hence $w_\theta\ge t_0$.

(ii) Suppose both witnesses are inner and $w_{150^\circ}>w_{120^\circ}$ strictly. Take $I=[t_0,w_{120^\circ}]$ with endpoints $P_0$ (top) and $\gamma(w_{120^\circ})$ ($120^\circ$-contact), and $J=[w_{150^\circ},t_3]$ with endpoints $\gamma(w_{150^\circ})$ ($150^\circ$-contact) and $P_3$ (lower face). The intervals are disjoint, and the cyclic boundary order $P_3,\gamma(w_{120^\circ}),\gamma(w_{150^\circ}),P_0$ (directions $0^\circ,120^\circ,150^\circ,180^\circ$) makes the endpoint pairs interleave. Lemma~\ref{lem:interleaving} excludes this, so $w_{150^\circ}\le w_{120^\circ}$.

(iii) Suppose $w_{150^\circ}>t_3$ and $w_{120^\circ}<t_3$ strictly; by (i), $w_{120^\circ}\ge t_0$. Take $I=[t_0,w_{120^\circ}]$ with endpoints $P_0$ and $\gamma(w_{120^\circ})$, and $J=[t_3,w_{150^\circ}]$ with endpoints $P_3$ (lower face) and $\gamma(w_{150^\circ})$ ($150^\circ$-contact). The intervals are disjoint, and the same cyclic boundary order $P_3,\gamma(w_{120^\circ}),\gamma(w_{150^\circ}),P_0$ makes the endpoint pairs interleave, which Lemma~\ref{lem:interleaving} excludes. Hence $w_{120^\circ}\ge t_3$, a weak-order degeneracy at $P_3$ being the boundary case.

In each part, if two of the four named contact points coincide, the strict configuration degenerates and the corresponding weak inequality holds by the witness convention. The reflected statements follow by applying the vertical reflection, which reverses orientation, exchanges the direction pairs, and preserves Lemmas~\ref{lem:cyclic-exposure} and~\ref{lem:interleaving}.
\end{proof}

Corollary~\ref{cor:role-trichotomy} also excludes the crossed configuration in which a positive-side witness has parameter on the left tail while a negative-side witness has parameter on the right tail: both are instances of part (i) and its reflection.

\subsection{Proofs of the support-order and case-tree lemmas}\label{app:order-proofs}

\begin{proof}[Proof of Proposition~\ref{prop:support-input}]
We first prove the statement for the simple polygonal approximants, using the interleaving machinery of Appendix~\ref{app:unimodality}, and then pass to the weak limit.

Let $\gamma$ be one of the simple polygonal approximants in the normalized $\Lambda$-position, with the closed witness convention of Section~\ref{subsec:lambda}: $P_{-3}$ and $P_3$ are the first and last contacts on the lower support $L_0$, and $P_0$ is a top contact between them, so that
\[
t(P_{-3})<t(P_0)<t(P_3),
\qquad
x(P_{-3})\le x(P_3).
\]
Part (1) is Corollary~\ref{cor:role-trichotomy}(i) and (iii): every selected positive-side witness at $150^\circ$ or $120^\circ$ has parameter weakly after $t(P_0)$, hence lies on the middle chain $\gamma_{P_0P_3}$ (inner role) or on the right tail $\gamma_{P_3Q}$ (outer role), and the strictly mixed pair (outer at $150^\circ$, inner at $120^\circ$) is impossible. Coincidence of a witness with $P_3$ is recorded as a weak equality and is absorbed as a degeneracy of an adjacent alternative, as in the closed witness convention. The negative side is the reflection of this argument across the vertical axis, which exchanges the roles of $150^\circ,120^\circ$ with $-150^\circ,-120^\circ$ and reverses the parameter inequalities accordingly.

Part (2): when both positive-side roles are inner, Corollary~\ref{cor:role-trichotomy}(ii) gives $t(P_1)\preceq t(P_2)$, and the inner-role definition gives $t(P_0)\preceq t(P_1)$ and $t(P_2)\preceq t(P_3)$; together with the reflected inequalities on the negative side and $t(P_{-3})\preceq t(P_{-2})$, this is exactly the displayed chain in the all-inner branch. On a side with outer roles, no relation between the tail witnesses and $P_{\pm3}$ is asserted; the admissible weak tail orders are enumerated in Lemmas~\ref{lem:case2-order-reduction} and~\ref{lem:case3-order-reduction} and certified exhaustively.

Finally, all statements above are finite lists of closed weak parameter-order relations and support-face incidences for selected witnesses. Lemma~\ref{lem:closed-contact-order} passes them to the weak limiting obstruction, with coincidences appearing as weak equalities.
\end{proof}

\begin{proof}[Proof of Lemma~\ref{lem:middle-skeleton}]
Let $K=\conv(\gamma)$. In the present normalization, $P_{-3}$ and $P_3$ lie on the lower horizontal support line, while $P_0$ lies on the upper horizontal support line. The relevant exposed faces from $P_{-3}$ to $P_3$ through $P_0$ occur on the upper boundary chain of $K$. Convexity orders those faces by the normal fan; exposed segments and coincident faces are allowed as weak equalities.

By Proposition~\ref{prop:support-input}(2), the named middle witness contacts with inner roles satisfy
\[
P_{-3}\preceq P_{-2}\preceq P_{-1}\preceq P_0\preceq P_1\preceq P_2\preceq P_3,
\]
restricted to the witnesses present in the branch under consideration. This proves only the ordered finite support skeleton, which is the statement of the lemma; no relation between tail witnesses and $P_{\pm3}$ is asserted.
\end{proof}

\begin{proof}[Proof of Lemma~\ref{lem:order}]
By Lemma~\ref{lem:middle-skeleton}, the selected middle contacts already occur in the weak parameter order
\[
P_{-3},\,P_{-2},\,P_{-1},\,P_0,\,P_1,\,P_2,\,P_3.
\]
If a support face is a segment or two adjacent contacts coincide, the convention uses the parameter endpoints of that face and gives a weak equality in the same order.

The remaining special contacts cannot interleave with this middle chain. The directions $30^\circ$ and $60^\circ$ expose the lower/right side after the lower contact $P_3$ in the normal-fan sweep, so their named positive contacts are tail roles $P_4$ and $P_5$ except for the endpoint degeneracy at $P_3$. The directions $120^\circ$ and $150^\circ$ either expose the inner roles $P_2,P_1$ on $\gamma_{P_0P_3}$ or the outer right-tail roles $P_7,P_8$, as recorded in Lemma~\ref{lem:contact-alternatives}; the negative directions are reflected. Therefore no tail role can occur between two of the ordered middle contacts, which proves the lemma.
\end{proof}

\begin{proof}[Proof of Lemma~\ref{lem:contact-alternatives}]
We prove the contact-alternative lemma in the present oriented-support notation.
Consider the positive side; the negative side follows by reflecting the normalized configuration across the vertical axis.

By Proposition~\ref{prop:support-input}(1), equivalently Corollary~\ref{cor:role-trichotomy}(i), the selected $L_{150^\circ}$ witness has parameter weakly after $t(P_0)$: it is either the inner contact $P_1$ on $\gamma_{P_0P_3}$ or the right-tail contact $P_8$ on $\gamma_{P_3Q}$. Likewise the selected $L_{120^\circ}$ witness is either $P_2$ on $\gamma_{P_0P_3}$ or $P_7$ on $\gamma_{P_3Q}$. The witness convention records weak equalities when an exposed face is a segment or when a witness coincides with $P_0$ or $P_3$.

The strictly mixed pair $(P_8,P_2)$---the $L_{150^\circ}$ contact strictly on the tail while the $L_{120^\circ}$ contact is strictly inner---is impossible by Corollary~\ref{cor:role-trichotomy}(iii): such a configuration would force two disjoint subarcs of $\gamma$ whose endpoint contacts interleave on the boundary of $\conv(\gamma)$, contradicting Lemma~\ref{lem:interleaving} and the simplicity of $\gamma$. Equivalently, in the contrapositive form used in the case tree: if the $L_{120^\circ}$ contact is inner, then the $L_{150^\circ}$ contact is inner as well.
When an exposed face contains more than one selected witness, this gives a weak-order degeneracy of one of the same alternatives and introduces no additional branch. The statement passes to weak support-reduced obstructions by Lemma~\ref{lem:closed-contact-order}.
\end{proof}

\begin{proof}[Proof of Lemma~\ref{lem:wedge}]
Let
\[
\phi(Y)=|XY|+|YZ|,
\qquad Y\in\partial W=\ell_1\cup\ell_2.
\]
For every boundary point $Y$, the triangle inequality gives
\[
|XY|+|YZ|\ge |XZ|.
\]
Equality can occur only when $Y$ lies on the segment $[X,Z]$. Since $X\in\ell_1$ and $Z\in\ell_2$, the full-boundary minimum is therefore attained at the degenerate boundary contacts $Y=X$ and $Y=Z$, which are points of $[X,Z]\cap\partial W$.

The lemma is used only when the closed branch model contains the degenerate chain obtained by deleting the kink. It does not assert that the minimum over an arbitrary closed subinterval of a wedge side occurs at an endpoint.
\end{proof}

\begin{proof}[Proof of Lemma~\ref{lem:subcase11-tail-reduction}]
We give the right-side argument; the left side is its reflection. In Subcase~1.1 the two $T_W$ sides through the common $30^\circ$ corner cut the right tail into a wedge sector whose boundary rays contain the relevant middle contact $P_{2.5}$ and the possible later tail contacts. There are two alternatives, according as the corresponding $15^\circ$ support contact occurs at a point $P_{3.5}\in\gamma_{P_3P_6}$ or at a later point $P_{6.5}\in\gamma_{P_6Q}$.

Parameterize the right active ray from $W$ by
\[
Z_+(s)=W+s e_{-75^\circ},\qquad s\ge0.
\]
For this ray the relevant escape functional is
\[
f_+(s)=d(\overrightarrow{WZ_+(s)},-75^\circ)
=\langle s e_{-75^\circ},e_{-75^\circ}\rangle=s,
\]
so $f_+'(s)=1$ and the closed escape half-plane on this ray is exactly $s\ge b$. Thus a later point $Y=Z_+(s_Y)$ with $s_Y\ge b$ does not by itself imply that the nearest vertical delimiter $Z=P_6=Z_+(s_Z)$ satisfies the same row; the replacement is made only in the sub-branch where $s_Z\ge b$.

If the first alternative occurs, or if the later-contact alternative has $s_Z\ge b$, the chain used in Subcase~1.1a records the point $P_{3.5}$, allowing the degenerate choice $P_{3.5}=P_6$ in the latter case. Apply Lemma~\ref{lem:wedge} to the boundary kink $X$--$Y$--$Z$, where $Y$ is the later tail contact and $Z=P_6$ is the nearest vertical delimiter on the same active boundary ray. Deleting the kink gives the segment $X$--$Z$, and the triangle inequality gives no larger polygonal length. The displayed calculation gives the row preservation: because $s_Z\ge b$, the degenerate contact still satisfies the same closed escape row. In the closed SOCP model these are exactly the constraints
\[
d(\overrightarrow{WP_{3.5}},-75^\circ)\ge b
\qquad\text{and}\qquad
d(\overrightarrow{WP_{-3.5}},-105^\circ)\ge b
\]
together with the two side-incidence equations for $P_{\pm2.5}$. Therefore the later-contact alternative is already covered by the closed model for the first alternative.

If $s_Z<b$, then the vertical delimiter is not used as a substitute for the escaping point. Instead the late contact is retained. The left active ray is parameterized by
\[
Z_-(s)=W+s e_{-105^\circ},\qquad
d(\overrightarrow{WZ_-(s)},-105^\circ)=s,
\]
so the same calculation applies on the reflected side. If exactly one side is late, the one-sided closed model 1.1b, or its reflection, is used: it keeps the active late-side incidence, keeps the opposite delimiter-side escape row, and keeps the side-containment row on the nonactive side. If both sides are late, no opposite pre-delimiter escape row is substituted. The direct model 1.1c keeps the chain
\[
P_{-6.5},P_{-6},P_{-3},P_{-2},P_{-1},P_1,P_2,P_3,P_6,P_{6.5},
\]
the two side-incidence rows for $P_{-6.5}$ and $P_{6.5}$, and the two late escape rows
\[
d(\overrightarrow{WP_{-6.5}},-105^\circ)\ge b,\qquad
d(\overrightarrow{WP_{6.5}},-75^\circ)\ge b .
\]
Table~\ref{tab:wedge-audit} records the four routing rows and identifies the preserved and omitted constraints. Coincident contacts are included as degenerate closed-boundary cases.
\end{proof}

\begin{proof}[Proof of Lemma~\ref{lem:case2-order-reduction}]
Consider the right tail in Case~2. The active right-tail roles used later are the $L_{30^\circ}$ contact $P_4$, the forced $L_{120^\circ}$ escape contact $P_7$, and the outer $L_{150^\circ}$ contact $P_8$. All three lie on the tail interval $\gamma_{P_3Q}$, with coincident exposed-face witnesses allowed in the closed branch. A weak parameter order of three named contacts is therefore one of the six orders
\[
\begin{array}{c|c}
\text{row} & \text{right-tail order after }P_3\\
\hline
C2\text{-}1 & P_4P_7P_8\\
C2\text{-}2 & P_8P_7P_4\\
C2\text{-}3 & P_4P_8P_7\\
C2\text{-}4 & P_7P_4P_8\\
C2\text{-}5 & P_7P_8P_4\\
C2\text{-}6 & P_8P_4P_7 .
\end{array}
\]
The first two rows are the representative orders used in 2.1a--b and 2.2a--b; the four remaining right-tail orders are the Case~2 raw-order SOCP models in Table~\ref{tab:raw-order-cert}.

In Subcase~2.2 the additional left-tail roles are $P_{-7}$ and $P_{-6}$, both on the interval before $P_{-3}$. Hence their weak order is either
\[
P_{-7}P_{-6}P_{-3}
\qquad\text{or}\qquad
P_{-6}P_{-7}P_{-3}.
\]
The first left order gives the representative 2.2 models and the right-tail raw-order additions; the second is paired with all six right-tail orders in the raw-order certificate. This proves the stated enumeration without any endpoint replacement.
\end{proof}

\begin{proof}[Proof of Lemma~\ref{lem:case3-order-reduction}]
On the right tail in Case~3, the active roles are $P_4,P_5,P_7,P_8$, all lying on the closed interval $\gamma_{P_3Q}$. A weak parameter order of these four named witness contacts is represented by at least one of the $24$ strict linear orders; coincident contacts are handled as weak degenerations of those strict orders. The reflected left tail is treated in the same way, with the reflected labels $P_{-4},P_{-5},P_{-7},P_{-8}$.

The raw-order certificate file contains all $24\times24$ left-right order pairings except the three representative pairings already included as 3a--3c in the main representative model file. Those three representative pairings are $(L_1,R_1)$, $(L_1,R_{24})$, and $(L_{24},R_{24})$ in the appendix branch ledger. Hence every Case~3 branch is certified by a closed SOCP model without using any geometric exclusion of tail orders.
\end{proof}

\clearpage
\end{appendices}

\end{document}